\documentclass[a4paper,11pt, oneside, leqno]{amsart}
\usepackage{graphicx}
\usepackage{relsize}
\usepackage{xfrac}
\usepackage{color}
\usepackage{xcolor}
\usepackage[colorlinks=true ]{hyperref}
\usepackage{pifont}
\makeatletter
\def\mathcolor#1#{\@mathcolor{#1}}
\def\@mathcolor#1#2#3{%
  \protect\leavevmode
  \begingroup
    \color#1{#2}#3%
  \endgroup
}
\makeatother

\usepackage[backend=bibtex, style=alphabetic, isbn=false, url=false]{biblatex}
\bibliography{refs-2}

\usepackage[multiple]{footmisc}
\usepackage{amsmath, amsthm, slashed}
\usepackage{amsfonts, slashed, cancel}
\usepackage{amssymb}
\usepackage{rotating}
\usepackage{epsfig}
\usepackage{verbatim}
\usepackage{rotating}
\usepackage{graphicx}
\usepackage{amsmath,amsthm,amssymb,verbatim,a4}
\usepackage{mathrsfs,stmaryrd}
\usepackage{marginnote}
\newtheorem{theorem}{Theorem}

\usepackage{chngcntr}
\newcounter{dummy}
\counterwithin{dummy}{section}
\newtheorem{definition}[dummy]{Definition}

\newtheorem{corollary}[dummy]{Corollary}
\newtheorem{proposition}[dummy]{Proposition}
\newtheorem{lemma}[dummy]{Lemma}
\newtheorem{remark}[dummy]{Remark}

\usepackage{color}

\newcommand{\zz}{\mathfrak z}

\newcommand{\R}{\mathbb{R}}

\newcommand{\ee}[1]{\mathbf{e}_{#1}}
\newcommand{\VV}[1]{\mathbf{V}_{#1}}
\newcommand{\WW}{\mathbf{W}}
\newcommand{\YY}{\mathbf Y}
\newcommand{\T}{\mathbf{T}}
\newcommand{\Tp}{\mathbf{T}_{\phi}}
\newcommand{\hZ}{\widehat{Z}}

\newcommand{\al}{\alpha}
\newcommand{\be}{\beta}
\newcommand{\ab}[1]{|#1|}

\newcommand{\eq}[1]{\begin{equation}#1\end{equation}}
\newcommand{\alg}[1]{\begin{aligned}#1\end{aligned}}
\newcommand{\ZZ}[1]{\mathbf{Z}_{#1}}
\newcommand{\p}[1]{\partial_{#1}}
\newcommand{\pv}[1]{\partial_{v^{#1}}}

\newcommand{\px}[1]{\partial_{x^{#1}}}

\date{}

\newcommand\blfootnote[1]{%
  \begingroup
  \renewcommand\thefootnote{}\footnote{#1}%
  \addtocounter{footnote}{-1}%
  \endgroup
}%

\date{}

\author{David Fajman}
\address{Gravitational Physics,
Faculty of Physics,
University of Vienna,
Boltzmanngasse 5,
1090 Wien,
Austria.}
\email{david.fajman@univie.ac.at}
\author{J\'er\'emie Joudioux}
\address{Gravitational Physics,
Faculty of Physics,
University of Vienna,
Boltzmanngasse 5,
1090 Wien, Austria.}
\email{jeremie.joudioux@univie.ac.at}
\author{Jacques Smulevici}
\address{Laboratoire de Math\'ematiques, Univ. Paris-Sud, CNRS, Universit\'e Paris-Saclay, 91405 Orsay, and D\'epartement de math\'ematiques et applications, \'Ecole Normale Sup\'erieure, CNRS, PSL Research University, 75005 Paris, France.}
\email{jacques.smulevici@math.u-psud.fr}

\title[]{Sharp asymptotics for small data solutions of the Vlasov-Nordstr\"om system in three dimensions}
\begin{document}

\maketitle

\begin{abstract}
This paper proves almost-sharp asymptotics for small data solutions of
the Vlasov-Nordstr\"om system in dimension three. This system consists
of a wave equation coupled to a transport equation and describes an
ensemble of relativistic, self-gravitating particles. We derive sharp decay estimates using a variant of the vector-field method introduced in previous work. More precisely, we construct
modified vector fields, depending on the solutions, to propagate
$L^1$-bounds for the distribution function and its derivatives. The
modified vector fields are designed to have improved commutation
properties with the transport operator and yet to still provide
sufficient control on the solutions to allow for a sharp
Klainerman-Sobolev type inequality. Our method does not require any compact support assumption in the velocity variable nor do we need strong interior decay for the solution to the wave equation.
\end{abstract}
\blfootnote{Preprint number:UWThPh-2017-4}
\tableofcontents

\section{Introduction}
This paper is concerned with the asymptotic behaviour of small data solutions to the Vlasov-Nordstr\"om system in dimension three, i.e.~the system
\begin{eqnarray} \label{eq:vnp}
\square \phi&=& m^2 \int_v f \frac{dv}{\sqrt{m^2+|v|^2}}\equiv\boldsymbol{\rho}(f), \\
\Tp(f)\equiv\T(f)-\left(\T(\phi) v^i +m^2 \nabla^i \phi \right)\frac{\partial f}{\partial v^i}&=&(n+1) f\, \T (\phi), \label{eq:vnf}
\end{eqnarray}
where $m>0$ is the mass of particles,  $\T\equiv v^\alpha \partial_{x^\alpha}$, with $v^0=\sqrt{m^2+|v|^2}$, is the \emph{relativistic free transport operator}, $\square \equiv-\partial_t^2+\sum_{i=1}^n \partial_{x^i}^2$ is the standard wave operator of Minkowski space, $\phi$ is a scalar function of $(t,x)$ and $f$ is a function of $(t,x,v^i)$ with $x \in \mathbb{R}^n$, $v \in \mathbb{R}^n$. A detailed introduction to this system can be found in \cite{MR1981446}. See also the classical works \cite{MR2238881,MR2194582, MR2085540}.

In \cite{sf:gssvns}, a small data global existence in dimension three was obtained for this system, deriving in particular decay estimates in time for the wave $\phi$ and the velocity averages of $f$ for data of compact support. The strategy of \cite{sf:gssvns}, similar to the strategy of \cite{MR919231} for the Vlasov-Maxwell system\footnote{See also \cite{bd:gevp} for the Vlasov-Poisson system.}, consists in using decay estimates for the velocity averages of $f$ based on the method of characteristics and the compact support assumptions, together with representation formulae for the wave equation. In particular, no decay estimates for the derivatives of the velocity averages of $f$ or the higher order derivatives of the wave were derived.

\subsection{Vector-fields and modified-vector-fields approach}
In \cite{fjs:vfm}, we introduce a novel approach to the study of coupled systems of wave and transport equations, based on the vector-field method of Klainerman. In particular, this method allows for a systematic study of systems such as the Vlasov-Nordstr\"om system and we obtain sharp (or almost sharp) asymptotics for the solution and its derivatives in the case of either massive particles ($m>0$) in dimensions $n \ge 4$ or massless particles up to dimension $3$. Our strategy is based on commuting the transport equation by the complete lift $\widehat{Z}$ of the Killing fields $Z$ of Minkowski space. By construction, the vector fields $\widehat{Z}$ are then differential operators that commute exactly with the free transport operator $\T$.

For the non-linear system, the commutation of $\Tp$ and $\widehat Z$ introduces error terms which then need to be integrable in space-time for the estimates to close. Contrary to, for instance, a non-linear wave equation of the form $\square \phi= Q(\partial \phi, \partial \phi)$, the transport equation \eqref{eq:vnf} enjoys poor commutation properties, in the sense that commuting with any of the vector fields $\widehat{Z}$ generates error terms of the form $(Z \partial \phi)\cdot \partial_v f$. These error terms are problematic because the vector fields $\partial_{v^j}$ do not commute\footnote{In the case of massless particles $m=0$, the exact vector field hitting $f$ in these error terms would be $v^i \partial_{v^i}$, which actually commutes with the free, massless, transport operator, which explains (partly) why the massless case is much easier than the massive one. See \cite{fjs:vfm} for more on the massless Vlasov-Nordstr\"om system.} with the free transport operator, so that they generate an extra growth which is roughly proportional to $t$. Because of this extra growth, the techniques introduced in \cite{fjs:vfm} could only handle massive particles in high dimensions ($n \ge 4$).

In fact, this difficulty is already present for the much simpler Vlasov-Poisson system. In that case, the difficulty was resolved \cite{MR3595457} by modifying the commutation vector fields, replacing the lifted vector fields $\widehat{Z}$ by some $Y= \widehat{Z}+ \Phi^i \partial_{x^i}$, where the coefficients $\Phi^i$ are functions in the variable $(t,x,v)$, depending on the solution and constructed to cancel the worst error terms in the commutator formulae. See also \cite{hwang11} for previous results concerning sharp asymptotics for solutions of the Vlasov-Poisson system based on the method of characteristics.

In this paper, we pursue a similar strategy and in particular, construct an "algebra" of modified vector fields, specifically designed to obey improved commutation properties with $\Tp$, yet to still allow for an (almost) sharp Klainerman-Sobolev inequality. As in \cite{fjs:vfm}, we use the hyperboloidal foliation by the hypersurfaces $H_\rho$ of constant hyperboloidal time $\rho:= \sqrt{t^2-|x|^2}$. In particular, all the energies and norms we consider are constructed with respect to this foliation. Standard references concerning hyperboloidal foliations for wave equations are \cite{MR1299110, MR868737, MR1199196}. See also the recent results \cite{MR3362362, MR3535896, qw:ihaekg} concerning the stability of the Minkowski space for the Einstein-Klein-Gordon system.

To deal with the wave equation, we therefore consider energy norms $\mathscr{E}_N[\phi](\rho)$ obtained out of the standard energy momentum tensor integrated on $H_\rho$. The only multiplier that we consider here is $\partial_t$ and the only decay estimate required  for $\phi$ is given by a standard Klainerman-Sobolev inequality (associated to the hyperboloids). In particular, we only use an interior (i.e.~away from the light cone) decay estimate for $\phi$ of the type $t^{3/2}|\partial \phi| + t^{1/2} |\phi| \lesssim 1$ which is much weaker than the interior decay used for instance in \cite{sf:gssvns}.

For the distribution function, our norm, denoted $E_N[f](\rho)$, is constructed out of modified vector fields $\mathbf{Y}$. Moreover, we actually consider weighted norms, where the extra weights are of the form $\zz= t \frac{v^i}{v^0} -x^i$. Note that these weights are actually propagated by the linear flow, i.e.~they solve $\T(\zz)=0$. The norm $E_N[f](\rho)$ is thus constructed out of $L^1$-type norms on $H_\rho \times \mathbb{R}^3_v$ of $\zz^q \YY^\alpha(f)$. The $\zz$-weights appear naturally in the commutator formula in conjunction with our choice of modified vector fields. We refer to Sections \ref{se:cvf} and \ref{se:nbs} respectively for the precise definitions of the modified vector fields and the norms used in this paper.

\subsection{The main result}

The main theorem of this paper establishes the stability of the trivial solution to the Vlasov-Nordstr\"om system, and provides in particular an almost sharp description of the asymptotic behaviour of the fields.
\begin{theorem} \label{th:asmsl4}
Let $N \ge 10$. There exists an $\varepsilon_0 >0 $ so that, for any initial data $(\phi_0, \phi_1, f_0)$ on the hyperboloid $H_1$, satisfying
 $$
 \mathscr{E}_N[\phi_0, \phi_1] + E_{N+3}[f_0] \le \varepsilon,
 $$
  the unique maximal solution $(\phi,f)$ to the Cauchy problem \eqref{eq:vnp} satisfying the initial conditions
  $$
  \phi_{H_1}=\phi_0, \quad \partial_t \phi_{H_1}=\phi_1, \quad f_{H_1 \times \mathbb{R}^3_v}=f_0
  $$
   is defined globally in the future of $H_1$ and verifies, for all $\rho \geq 1$,
\begin{eqnarray*}
E_{N}[f](\rho)&\lesssim& \varepsilon\rho^{\delta(\varepsilon)}. \\
\mathscr{E}_{N}[\phi](\rho)&\lesssim&  \varepsilon\rho^{ \delta(\varepsilon)}. \\
\mathscr{E}_{N-1}[\phi](\rho)&\lesssim& \varepsilon,
\end{eqnarray*}
where $0 \le \delta(\varepsilon) \rightarrow 0$, as $\varepsilon\rightarrow 0$.
\end{theorem}

From the above statement and Klainerman-Sobolev inequalities, one then obtains decay estimates for $\phi$ and velocity averages of $f$ (and their derivatives), which are sharp in the case of $\phi$ (since there is no loss apart from the top order energy estimate) and almost sharp for $f$, in the sense that there is a $\rho^{\delta(\varepsilon)}$ loss compared to the linear estimate. Moreover, if we allow stronger initial decay for $\phi$, then, by standard techniques, we would obtain improved interior decay that would allow us to also remove the small loss of decay for the velocity averages of $f$. It is in that sense that we obtain sharp asymptotics for this system.

\subsection{Elements and difficulties of the proof}
\subsubsection{Large velocities}

An important aspect of Theorem \ref{th:asmsl4} is that no compactness assumptions on the $v$ support of the solutions\footnote{We do not also assume any compact support in $x$, but we do start from some hyperboloid. To go from an initial $t=const$ slice to a future hyperboloid typically requires strong initial decay in $x$, see the discussion in \cite[Appendix A]{fjs:vfm}.} are required. The only $v$ decay that we need is that the initial norms, which are integrals in $v$ (and $x$) with polynomial weights, are bounded.

A strong advantage of compact support assumptions is that they allow for a clean separation of the characteristics associated with the wave equation (the null geodesics) and the characteristics of the distribution function (which are timelike). This means that, in that situation, when estimating products of the form $\partial Z^\alpha (\phi) \YY^\beta(f)$, because of the support assumptions, one can always assume that one lies far from the light cone $t \equiv |x|$, since otherwise one must be away from the support of $f$ (for $t$ sufficiently large). In our case, no such separation occurs. Essentially, for large $v$, $\sqrt{m^2+ |v|^2} \sim |v|$ holds, so that the characteristics of the distribution function converge in some sense to that of the wave. Using the hyperboloidal foliation, the present norms for $f$ contain the weight $v^\rho:= v^\alpha n_\alpha$, where $n$ is the future unit normal to the hyperboloid. Moreover, one can prove an estimate of the form $ \frac{t}{\rho v^0} \lesssim v^\rho $. Since a weight $t$ is stronger than a weight $\rho$, this allows to extract more decay from the wave to estimate the above products, but at a cost of losing in powers of $v^0$, consistent with the fact that at large $v$ our estimates get worse. Thus, we need to carefully take into account powers of $v$ in all the equations. Looking at the structure of the transport operator on the left-hand side of \eqref{eq:vnf}, we notice that two different terms arise, $\T(\phi) v^i \partial_{v^i} f$ and $\partial_{x^i} \phi\cdot \partial_{v^i}f$, which have different homogeneity in $v$, the second term being better in this regard. A basic application of our estimates would in fact not allow to estimate the error terms coming from this first term due to the high number of powers of $v$. Instead, the structure of this term plays an important role. A heuristic picture of this structure is the following. As discussed earlier, the difficulty originates in large $v$ and, at large $v$, $v^0 \sim |v|$ holds, meaning that it becomes increasingly hard to distinguish massive from massless particles. However, the vector field $v^i \partial_{v^i}$, which appears in the error terms coming from the product $\T(\phi) v^i \partial_{v^i} f$, actually commutes with the massless transport operator $|v| \partial_t + v^i \partial_{v^i}$. This implies that, even though $v^i \partial_{v^i}$ does not commute with the massive transport operator, the error terms generated have strong decay properties in $v$ (they indeed contain negative powers of $v^0$).

\subsubsection{Modified vector fields}

It turns out that a first modification allows us to capture the aforementioned mechanism concerning the vector field $v^i \partial_{v^i}$ and provides already an improved commutator. We replace each translation $\partial_{x^\alpha}$ by a \emph{generalized translation}
\eq{
\ee{\alpha}\equiv \partial_{x^\alpha} - \left(\partial_{x^\alpha} \phi\right) \cdot v^i \partial_{v^i}.
} We then replace, in each of the Killing fields and complete lifts of the Killing fields the usual translations by their generalized versions. For instance, for a Lorentz boost $\Omega_{0i}=t \partial_{x^i}+ x^i \partial_t$, we obtain the field $t \ee 0 + x^i \ee i$. The use of the generalized translations then implies (see Lemma \ref{lem:blockcommut}) that the resulting commutators have improved properties in terms of powers of $v$ (though still bad in terms of space-time decay). In a second step, we further modify  the vector fields coming from the homogeneous vector fields (rotations, boost and scaling). If $\mathbf{Z}$ denotes any of these fields, the modification takes the form $\mathbf{Y}= \mathbf{Z} + \Phi^i \mathbf{X}_i$, where $\mathbf{X}_i= \ee{i}+ \ee{0} \frac{v^i}{v^0}$. The reason for the introduction of the fields  $\mathbf{X}_i$ is that, when applied to $\phi$, a decomposition of the form
\eq{
\mathbf{X}_i (\phi)= \frac{Z(\phi) }{t}+ \frac{\zz}{t} \partial \phi,
}
holds, where the right-hand side enjoys improved decay. This is due to the overall $t^{-1}$ factor and, in the second term, the fact that the weight $\zz$ is one of the weights discussed above which are propagated by the linear flow. Together, this implies a strong improved decay for velocity averages of products of type $\mathbf{X}_i (Z^\beta(\phi))\YY^\alpha(f)$. Finally, the coefficient $\Phi^i$ appearing in the definition of the modified vector fields are designed to cancel the worst terms in the original commutation formula, see Lemma \ref{lem:firstorder}.

\subsubsection{The $\zz$ weights}
As explained above, our choice of modified vector fields naturally introduces the additional $\zz$ weights. However, in order to avoid having to estimate $L^1$ norms of $\zz^q \YY^{\alpha}(f)$ in terms of $\zz^{q+1} \YY^{\alpha}(f)$, which would not allow us to close the estimates, the number of weights $q$ depends on the multi-index $\alpha$. Essentially, generalized translations have better commutation properties and allow for additional $\zz$ weights.

\subsubsection{Hierarchy}
Due to the presence of weights $\zz$ and because of the small loss for the norm of $f$, it is important to exploit a certain hierarchy to close the estimates. For instance, despite the growth of the norm of $f$, the energy estimates for $\phi$ can still close, without any loss, up to order $N-1$. However, this relies on a crucial integration by parts which, at top order, can no longer be used due to a lack of regularity (cf.~ Section \ref{se:tocf}). This eventually results in the growth of the top order energy for the solution to the wave equation. In the argument, it is essential that the problematic source terms in the wave equation at top order always contain $\mathbf{Y}^\alpha(f)$  with $|\alpha|$ small (say less than $N/2$), or otherwise the top order estimate would not close. This loss at higher order is reflected in the hierarchy of norms present in the bootstrap assumptions, see Section \ref{se:nbs}.

\subsubsection{$L^2$-decay estimates and $L^1$-estimates for high-low products}
As in \cite{fjs:vfm}, the present approach relies on $L^2$-decay estimates to control the velocity averages of $\mathbf{Y}^\alpha(f)$ for $\alpha$ large. Moreover, in the analysis, we also need $L^1$-estimates on products of type $\mathbf{Y}^\alpha(\Phi) \mathbf{Y}^\beta(f)$ and $|\mathbf{Y}^\alpha(\Phi)|^2 \mathbf{Y}^\beta(f)$, in the situation when $\alpha$ is so large than one does not have access to pointwise estimates on $\mathbf{Y}^\alpha(\Phi)$.

\subsection{Perspectives of the method and the Einstein-Vlasov system}
One of the main motivations for the present paper comes from the Einstein-Vlasov system.  We refer to the recent book\footnote{Apart from a general introduction to the Einstein-Vlasov system, the main purpose of this book is to present a proof of stability of exponentially expanding space-times for the Einstein-Vlasov system, see \cite{MR3186493}.} \cite{MR3186493} for a thorough introduction to this system. The small data theory around the Minkowski space is still incomplete for the Einstein-Vlasov system. The spherically symmetric cases in dimension $(3+1)$ have been treated in \cite{rr:gesssvsssd, rr:err} for the massive case and in \cite{md:ncsdsgm} for the massless case with compactly supported initial data. A proof of stability for the massless case without spherical symmetry but with compact support in both $x$ and $v$ has been given in \cite{mt:smsmevs}. As in \cite{md:ncsdsgm}, the compact support assumptions and the fact that the particles are massless are important as they allow to reduce the proof to that of the vacuum case outside from a strip going to null infinity. Let us also refer to stability results for the Einstein-Vlasov system without symmetry assumptions in the cosmological case  \cite{MR3186493}, \cite{Fajman2017}.

 We consider the present work as a first step towards a proof of stability of the Minkowski space for the massive Einstein-Vlasov system. Indeed, due to the highly non-linear structure of the Einstein equations, any precise global analysis of the solutions relies on commuting the Einstein equations and by the coupling, one is forced to estimate derivatives of velocity averages of the distribution function as well. Note that our method only uses basic energy techniquess to estimate the solution to the wave equations and is therefore fully compatible with the relevant techniques used in the study of the Einstein equations. An alternative approach to the study of the stability of the Minkowski space for massive particles has been recently announced in \cite{mt:smsm2}.

Let us finally mention that the vector-field method introduced in \cite{fjs:vfm} has been extended to prove decay of massless distribution functions on Kerr black holes \cite{abj:hsdvk}, as well as to derive decay estimates for other dispersive partial differential equations \cite{ww:cvfa}.

\subsection{Overview of the paper} In Section \ref{sec:prelimineries}, we introduce the geometric frameworkr (the hyperboloidal foliation), as well as the basic energy estimates for distribution functions. At the beginning of Section \ref{se:cvf}, we define the vector fields used to build up our algebra of modified commutators, and derive the necessary commutation relations for the next sections. In particular, Section \ref{sec:firstorderf} contains the definition of the modified vector fields (see Lemma \ref{lem:firstorder}), in relation to the first order commutator formula stated in Proposition \ref{prop:firstorder}. The higher order commutator formula is proven in Section \ref{sec38}; a concise version of the formula is stated in Corollary \ref{cor-comm-f}. A similar formula for the wave equation is stated in Section \ref{sec:commwave}. The low order formula is in Lemma \ref{lem:wavelow} of Section \ref{sec:commwave1}; its higher order counterpart can be found in Lemma \ref{lem:topw} of Section \ref{sec:commwave2}. Section \ref{se:nbs} contains the norms (Section \ref{sec:norms}), and the bootstrap assumptions (Section \ref{sec:bootstrap}). Pointwise estimates for the coefficients of the modified vector fields are established in Section \ref{sec:Phiest}, Lemmata \ref{es:inht} and \ref{lem-foe}. The bootstrap assumptions for the $L^1$-norms of the distribution functions are improved in Section \ref{sec:L1estimates}, see Proposition \ref{prop:L1est}. The estimates for the norms combining the distribution function, and the coefficients of the modified vector fields are in Section \ref{sec:Phif}, and stated in Proposition \ref{prop:Phif}. The $L^2$-estimates for the similar products containing higher order derivatives of $f$ are proven in Section \ref{sec:L2wave} (Proposition \ref{pro:L2estfull}). Section \ref{sec:btwave} contains the estimates for the wave equation, based on the $L^2$-estimates for $f$. The necessary Klainerman-Sobolev estimates with modified vector fields are proven in Section \ref{sec:kse}. Finally, the appendix contains an integral estimate (Appendix \ref{sec:intes}) used in the sections devoted to estimate the distribution function, and a remark regarding a rescaling by $\rho$ (Appendix \ref{sec:rhomult}).

\subsection*{Acknowledgements}
The authors are grateful to the Erwin-Schr\"odinger Institute for Mathematics and Physics (ESI) for hospitality during the program \emph{Geometric Transport equations in General Relativity}, where part of this work have been written. D.F.~acknowledges support of the Austrian Science Fund (FWF) through the START- Project Y963-N35 of Michael Eichmair as well as through the project \emph{Geometric transport equations and the non-vacuum Einstein flow (P 29900-N27)}. J.J.~and J.S.~ are supported in part by the ANR grant AARG, "Asymptotic
Analysis in General Relativity" (ANR-12-BS01-012-01). J.S.~acknowledges funding from the European Research Council under the European Union's Horizon 2020 research and innovation progam (project GEOWAKI, grant agreement 714408).

\section{Preliminaries}\label{sec:prelimineries}
We start by recalling basic facts about the $H_\rho$ foliation and its geometry already covered in \cite{fjs:vfm}.

\subsection{The hyperboloidal foliation}\label{se:tf}
We introduce in that section the different sets of coordinates which are used throughout the paper.\\

\textbf{Cartesian Coordinates}\\
Fix global Cartesian coordinates $(t,x^i)$, $1 \le i \le 3$ on $\mathbb{R}^{3+1}$. For any $\rho > 0$, define $H_\rho$ by

$$
H_{\rho}=\left\{ (t,x)\,\, \big |\,\, t \ge |x|\,\, \mathrm{and}\,\, t^2-|x|^2= \rho^2 \right\}.
$$

\noindent Note that
\eq{
\displaystyle \bigcup_{\rho \ge 1} H_\rho=\left\{ (t,x)\in \mathbb{R}^{3+1}\, \big|\, t \ge (1+|x|^2)^{1/2} \right\}.
}
The above subset of $\mathbb{R}^{3+1}$ is referred to as \emph{the future of the unit hyperboloid} and denoted by $J^+(H_1)$.

On this set, we use as an alternative the following two other sets of coordinates.\\

\textbf{Spherical coordinates}\\
We first consider spherical coordinates $(r,\omega)$ on $\mathbb{R}^3_x$, where $\omega$ denotes spherical coordinates on the $2$-dimensional spheres and $r=|x|$. Then, $(\rho\equiv\sqrt{t^2-|x|^2},r,\omega)$ defines a coordinate system on the future of the unit hyperboloid. These new coordinates are defined globally on the future of the unit hyperboloid apart from the usual degeneration of spherical coordinates and at $r=0$.\\

\textbf{Pseudo-Cartesian coordinates}\\
These are the coordinates $(y^0,y^j)\equiv(\rho,x^j)$, which are also defined globally on the future of the unit hyperboloid.\\

For any function defined on (some part of) the future of the unit hyperboloid, we move freely between these three sets of coordinates.

\subsection{Geometry of hyperboloids} \label{se:gh}
The Minkowski metric $\eta$ is given in $(\rho,r,\omega)$ coordinates by
$$
\eta= -\frac{\rho^2}{t^2} \left( d\rho^2-dr^2 \right) -\frac{2 \rho r }{t^2} d\rho dr+ r^2 \sigma_{\mathbb{S}^{2}},
$$
where $\sigma_{\mathbb{S}^{n-1}}$ is the standard round metric on the $2$ dimensional unit sphere, so that for instance
$$
\sigma_{\mathbb{S}^2}=\sin \theta^2 d\theta^2+d\phi^2,
$$
in standard $(\theta, \phi)$ spherical coordinates for the $2$-sphere.
The $4$-dimensional volume form is thus given by $$\frac{\rho}{t}r^{2} d\rho dr d\sigma_{\mathbb{S}^{2}},$$ where $d\sigma_{\mathbb{S}^{2}}$ is the standard volume form of the 2-dimensional unit sphere. The Minkowski metric induces on each $H_\rho$ a Riemannian metric given by
$$
ds_{H_\rho}^2=\frac{\rho^2}{t^2}dr^2 +r^2 \sigma_{\mathbb{S}^{2}}.
$$

A normal differential form to $H_\rho$ is given by $t dt-r dr$ while $t \partial_t +r \partial_r$ is a normal vector field. Since
$$
\eta\left( t \partial_t +r \partial_r,  t \partial_t +r \partial_r\right)=-\rho^2,
$$
the future unit normal vector field to $H_\rho$ is given by the vector field
\eq{
\nu_\rho \equiv\frac{1}{\rho}\left(t \partial_t +r \partial_r \right).
}

Finally, the induced volume form on $H_\rho$, denoted $d\mu_{H_\rho}$, is given by
$$
d\mu_{H_\rho}=\frac{\rho}{t}r^{n-1}drd\sigma_{\mathbb{S}^{2}}.
$$

\subsection{The massive Vlasov-Nordstr\"om system} \label{se:msvcd4}
We consider the Vlasov-Nordstr\"om system (VNS) for particles of mass $m=1$ given by
\begin{eqnarray}
\square \phi &=& \int_{v} f \dfrac{d v}{v^0}\equiv\boldsymbol{\rho}(f), \label{eq:wpmsv}\\
\Tp(f)\equiv\T(f)-\left( \T( \phi)v^i + \nabla^i \phi \right) \frac{ \partial f}{\partial v^i}&=& 4 \T (\phi) f \label{eq:tfmsv},
\end{eqnarray}
where $v^0=\sqrt{1+|v|^2}$ and $\T= v^\alpha \partial_{x^\alpha}
$ denotes the free transport operator and the unknowns are $\phi\equiv\phi(t,x)$, a scalar field and $f\equiv(t,x,v)$ a distribution function. We refer to \cite{MR1981446,fjs:vfm} for a derivation and discussion of the system. In the present paper, the variables $(t,x,v)$ take values in $J^+(H_1) \times \mathbb{R}^3$. For data compactly supported in space and given on a $t=\mathrm{const}$ slice, a standard domain of dependence argument allows to reduce the problem only to this domain (see for instance \cite{fjs:vfm}, Appendix A). We complete the system \eqref{eq:wpmsv}, \eqref{eq:tfmsv} by the initial conditions
\begin{eqnarray} \label{eq:id}
\phi_{|H_1}=\phi_0, \quad \partial_t\phi_{|H_1}= \phi_1, \\
f\big |_{ H_1 \times \mathbb{R}^n_v}=f_0. \label{ic:fmsv}
\end{eqnarray}

\subsection{The energy-momentum tensor for distribution functions} \label{se:pvfset}
We recall here some basic facts of the energy-momentum tensor of distribution functions.
Let us first define the volume form
\begin{equation}\label{def:vf}
d\mu(v)\equiv\frac{dv^1 \wedge \hdots \wedge dv^n}{v^0}=\frac{dv}{\sqrt{1+\ab v^2}}.
\end{equation}

We then define the energy-momentum tensor as
$$
T^{\mu \nu} \equiv \int_{\mathbb R^n}f v^{\mu}v^{\nu}d\mu(v)
$$
and recall that
\begin{eqnarray} \label{eq:divt0}
\nabla_\mu T^{\mu \nu}&=& \int_{\mathbb R^n}\T(f)v^{\nu}d\mu(v). \label{eq:divtm}
\end{eqnarray}

We are interested in particular in the energy density
\eq{
\boldsymbol\rho(f)\equiv T( \partial_t, \partial_t)=\int_{\mathbb R^n }f v^0 dv
}
as well as

\eq{
\chi(f)\equiv T( \partial_t, \nu_\rho),
}
where $\nu_\rho$ is the future unit normal to $H_\rho$ introduced in Section \ref{se:gh}. We compute
\begin{eqnarray*}
\chi(f)&=&\int_{v \in \mathbb{R}^n} fv_0\left(\frac{t}{\rho}v_0+\frac{r}{\rho}v^r \right) d\mu(v), \\
&=&\int_{v \in \mathbb{R}^n } f\left(\frac{t}{\rho}v^0-\frac{x^i}{\rho}v_i \right) dv.
\end{eqnarray*}

The following lemma was proved in \cite{fjs:vfm}.
\begin{lemma}[Coercivity of the energy density normal to the hyperboloids]\label{lem:coercivity}
Assuming that $t \ge r$, we have
\begin{eqnarray} \label{ineq:coeednh}
\chi(f) &\ge& \frac{t}{2\rho} \int_{v \in \mathbb{R}^n}f\left[ \left(1-\frac{r}{t}\right)\left((v^0)^2+v_r^2\right)+r^2 \sigma_{AB} v^A v^B+1\right] \frac{dv}{v^0}.
\end{eqnarray}
\end{lemma}

From \eqref{eq:divt0} and the divergence identity, we obtain

\begin{lemma}\label{lem:macl} Let $h$ be a sufficiently regular distribution function. Let $f$ be a sufficiently regular solution to $\T(f)=v^0 h$, defined on $\bigcup_{\rho \in[1,P]}H_\rho \times \mathbb{R}^n_v$ for some $P>1$. Then, for all $\rho \in [1,P]$,
\begin{equation} \label{eq:aclrm}
\int_{H_\rho} \chi(f)(\rho,r,\omega)d\mu_{H_\rho}= \int_{H_1} \chi(f)(1,r,\omega)d\mu_{H_1}+\int_{1}^\rho \int_{H_\rho}\boldsymbol\rho(h)(s,r,\omega)d\mu_{H_s}ds,
\end{equation}
and
\begin{equation} \label{eq:aclavm}
\int_{H_\rho} \chi(|f|)(\rho,r,\omega)d\mu_{H_\rho} \le \int_{H_1} \chi(|f|)(1,r,\omega)d\mu_{H_1}+\int_{1}^\rho \int_{H_\rho} \boldsymbol\rho(|h|)(s,r,\omega)d\mu_{H_s}ds.
\end{equation}
\end{lemma}

\subsection{Notations}
\begin{itemize}
\item  Our convention for indices is as follows.
\begin{itemize}
\item Greek letters stand for the natural numbers $\{0,1,2,3\}$;
\item Latin letters, \underline{from the letter $i$} stand for the natural numbers $\{1,2,3\}$;
\item Latin letters, \underline{from the letter $a$ to the letter $g$} stand for the natural numbers $\{0, \dots, 10 \}$. These typically label vector fields within our ordered algebra of commutation vector fields.
\end{itemize}
\item Einstein summation convention: We use the Einstein summation convention both for coordinates indices and labels of the vectors.
\end{itemize}

\subsection{Good symbols} It is convenient to work with the following notion of \emph{good symbols}.
\begin{definition} A function $c$ on $J^+(H_1)\times \R^3_v$ is a \emph{good symbol} if $c$ is a smooth bounded function and for all $(p,q,r)$ a triplet of nonnegative integers, there exists a constant $C$ such that for all $(t,x,v)$ in $J^+(H_1)\times \R^3_v$,
$$
|\partial^p_t \partial^q_ x \partial^r_v c| \leq C t^{-p-q} (v^{0})^{-r}.
$$
\end{definition}
\begin{remark}
Examples of good symbols that appear in the text are
$$
\frac{1}{v^0}, \quad \frac{1}{t},  \quad \frac{\rho}{t}, \quad \dfrac{\zz}{t v^0}
$$
where $\zz \in \mathfrak Z$ is introduced in Section \ref{se:we}.
\end{remark}

\subsection{Klainerman-Sobolev inequalities for the wave equation in the hyperboloidal foliation}
We use the following estimates, obtained for instance in \cite{fjs:vfm}, Proposition 4.12. We denote by $\eta$ the Minkowski metric on $\mathbb{R}^4$.
\begin{proposition}[Klainerman-Sobolev inequality for the wave equation using the hyperboloidal foliation]\label{prop:hypwave}$\quad$
 For any sufficiently regular function $\psi$ of $(t,x)$ defined on $J^+(H_\rho)$, let $\mathscr{E}_{2}[\psi](\rho)$ denote the energy
$$
\mathscr{E}_{2}[\psi](\rho)=\sum_{ | \alpha | \le 2} \int_{H_\rho} T[Z^\alpha[\psi]] ( \partial_t, \nu_\rho) d\mu_\rho,
$$
where the vector fields $Z$ are the standard Killing fields of the Minkowski space, and $T[\psi]$ is the energy-momentum tensor for solutions to the wave equation:
$$
T[\psi]  = d \psi \otimes d \psi -  \frac12 \eta(\nabla\psi, \nabla\psi) \eta.
$$
Then, for all $(t,x)$ in the future of the unit hyperboloid,
\begin{equation} \label{eq:kswh}
|\partial \psi |(t,x) \lesssim \frac{1}{t(1+t-|x|)^{1/2}} \mathscr{E}_2^{1/2}[\psi](\rho(t,x)).
\end{equation}
\end{proposition}

Moreover, we also have
\begin{lemma}\label{lem:hypwave1}
Let $\psi$ be such that $\mathscr{E}_{2} [\psi](\rho)$ is uniformly bounded on $[1, P]$, for some $P > 1$. Assume moreover that $\psi_{| \rho=1}$ vanishes at $\infty$. Then $\psi$ satisfies, for all $(t,x)$ in $J^+(H_1)$,
$$
\vert \psi (t,x) \vert \lesssim \sup_{[1,P]}\left[\mathscr{E}_2^{1/2}[\psi] \right] \dfrac{(1+u)^{\frac12}}{t},
$$
where $u=t-|x|$.
\end{lemma}

\section{The commutation vector fields} \label{se:cvf}
The main norm we use for the distribution function $f$ is constructed out of  (weighted) $L^1$-type norms for $\YY^\alpha(f)$, where the $\YY^\alpha$ are differential operators of order $|\alpha|$ obtained as a combinations of $|\alpha|$ vector fields in $\mathbb Y\cup \mathbb E$ ($\alpha$ denotes some multi-index, cf.~definition of $\mathbb Y$ and $\mathbb E$ below). These vector fields are divided into different sets corresponding to their commutation properties with the transport operator and additional weights. We introduce these vector fields in the following. We consider first the so-called \emph{generalized translations},
\eq{
\ee \mu\equiv \p \mu - (\p \mu \phi) v^i\pv i, \qquad \mu \in \{0,1,2,3\},
}
where $\phi$ denotes the solution of the wave equation in the VNS. We denote the set of generalized translations by
\eq{
\mathbb{E}=\{\ee 0,\hdots,\ee 3\}.
}
As a second set we consider vertical auxiliary vector fields
\eq{
\VV i\equiv \pv i,\qquad i \in \{1,2,3\},
}
and
\eq{
\WW\equiv v^i\pv i.
}

Interestingly, these two sets of vector fields appear mutually in the respective commutator with the transport operator $\T_\phi$. In this notation, the transport operator reads
\eq{
\Tp = v^\al \ee \al- (\nabla^i\phi) \VV i.
}

We state the commutators explicitly.

\begin{lemma} \label{lem:blockcommut}
\eq{\alg{
\left[\Tp, \ee \mu \right] f&=\p \mu (\nabla^i\phi) \VV i f+(\p \mu \phi)\Big[-\frac1{v^0}\ee tf +\Tp f+2\nabla^i\phi \VV if\Big]\\
\left[\Tp, \VV i \right] f&=-\frac{v^i}{v^0}\ee t f-\ee if+(\T\phi)\VV i f\\
\left[\Tp, \WW  \right] f&=\frac 1{v^0} \ee t f- \Tp f - 2(\nabla^i\phi) \VV if
}}
\end{lemma}

It is useful to have the following mutual commutators of the vector fields at hand.

\begin{lemma}
\eq{\alg{
\left[\ee \al ,\ee\be\right]  f &= 0 \\
\left[\VV i ,\VV j\right] f &= 0\\
\left[\VV i , \ee \al\right] f&= - \p \al(\phi) V_i\\
\left[\WW , \ee \al\right] f&= 0\\
\left[\VV i,\WW \right] f& = \VV i f
}}
\end{lemma}

\subsection{The $\mathbf X_i$ vector fields}
We define a certain combination of the generalized translations by
\eq{
\mathbf X_i\equiv \ee i+\frac{v^i}{v^0}\ee t
}
and denote the set of these fields by
\eq{
\mathbb X=\{\mathbf X_1,\mathbf X_2,\mathbf X_3\}.
}
\begin{remark}\label{rem:formX} Noticing that for a Lorentz boost $Z_i= \Omega_{0i}=t \partial_{x^i}+ x^i \partial_t$ and a weight $\zz_i=\frac{v^i t - x^i v^0}{v^0}$,
$$
\frac{Z_i}{t}+ \frac{\zz_i }{t} \partial_t= \partial_{x^i} +\frac{v^i}{v^0} \partial_t
$$
holds, we obtain that
$$
\mathbf X_i = \frac{\mathbf{Z}_i}{t}+ \frac{\zz_i }{t} \mathbf{e}_0,
$$
where we denote by $\mathbf Z_i$ the $Z_i$ with translations replaced by generalized translations (see next section).
\end{remark}
The corresponding commutator with the transport operator reads
as follows.
\begin{lemma} \label{lem:xcom}
\begin{eqnarray*}
[\Tp, \mathbf X_i]&=&  \mathbf X_i (\nabla^j\phi) \VV j +\mathbf X_i (\phi)\Big[-\frac1{v^0}\ee t +\Tp +2\nabla^j\phi \VV j \Big]\\
&&\hbox{}+ \Tp\left( \frac{v^i}{v^0}\right) \ee t \\
&=&  \mathbf X_i (\nabla^j\phi) \VV j +\mathbf X_i (\phi)\Big[-\frac1{v^0}\ee t +\Tp +2\nabla^j\phi \VV j \Big]\\
&&\hbox{}- \frac{1}{v^0}\left( \frac{v^i}{v^0} \partial_t + \partial_{x^i} \right)(\phi)  \ee t \\
&=&  \mathbf X_i (\nabla^j\phi) \VV j +\mathbf X_i (\phi)\Big[-\frac1{v^0}\ee t +\Tp +2\nabla^j\phi \VV j \Big]\\
&&\hbox{}- \frac{1}{v^0}\mathbf X_i(\phi)  \ee t
\end{eqnarray*}
\end{lemma}

\subsection{Modified vector fields }
Instead of commuting with the vector fields associated to the free transport operator $\T=v^\alpha \partial_{x^\alpha}$, we modify those vector fields to improve the commutation relation with the perturbed transport operator $\T_\phi$. The generalized translations $\ee \mu$ and the $\mathbf X_i $ are already modified vector fields, but with an explicit  modification (in terms of $\phi$). In the following we define the rest of the modified vector fields. Let $\mathbb{P}$ be the Poincar\'e algebra of Minkowski space
$$
\mathbb{P} = \{ \partial_{\al}, \Omega_{\al \be}  \}
$$
and
let  $\mathbb{K}$ be the extended algebra obtained by adding the scaling vector field

\begin{equation} \label{eq:kalg}
\mathbb{K} = \{ \partial_{\al}, \Omega_{\al \be}, x^\al \partial_{x^\al}  \}.
\end{equation}

We consider this set to be an ordered set labelled from 0 to 10. The scaling is labelled by $0$,  the Lorentz boosts $\Omega_{0i}$ with the label corresponding to $i$, the rotations from 5 to 7, and the translations from $7$ to $10$. Given a vector field $X$, we denote by $\widehat{X}$ its complete lift\footnote{See for instance \cite{fjs:vfm} for a presentation of complete lifts.} given in coordinates by
$$
\widehat X= X^i \partial_{x^i}+ v^\alpha \frac{\partial X^i}{\partial x^\alpha} \partial_{v^i} .
$$

Let $\widehat{\mathbb{K}}$ be the set containing the scaling vector field and the complete lifts of the elements in $\mathbb{P}$. We order  $\widehat{\mathbb{K}}$ accordingly to $\mathbb{K}$. The set of vector fields which we further modify is
$$
\overset{\circ}{\mathbb{K}} = \left\{ \widehat{\Omega}_{\al\be}, x^\al \partial_{x^\al}  \right\},
$$
where it is important to note that the rotations and hyperbolic rotations are lifted and then corrected, while the non-lifted scaling is directly corrected. First, we replace in each of the above vector fields, the translations by generalized translations. Given $Z \in \overset{\circ}{\mathbb{K}}$, we denote by ${\mathbf Z}$ the resulting vector field obtained after this transformation. For instance, for the Lorentz boosts $\Omega_{0i}=Z_i$,

$$
\mathbf{Z}_i= t \ee i + x^i \ee 0.
$$

For $i\in\{4,\hdots, 10\}$, we define the $i$-th modified vector field by
\eq{ \label{def:mvf}
\mathbf Y_i\equiv \widehat {\mathbf Z}_i +\Phi_i^j \mathbf X_j,
}
where $\widehat{\mathbf Z}_i$ is the complete lift of the $i-th$ field where in addition the translations are replaced by the generalized translations as above and where the $\Phi_i^j:=\Phi_i^j(t,x,v)$ are coefficients we determine below.
In particular, for hyperbolic rotations and for rotations
\eq{\alg{
\mathbf Y_a &\equiv  t\ee i+x^i \ee t+v^0\pv i +\Phi^k_a \mathbf X_k,\\
\mathbf Y_b &\equiv x^i\ee j-x^j\ee i+v^i\pv j-v^j\pv i+\Phi^k_b\mathbf X_k,
}}
respectively. Similarly, for the scaling we define
\eq{
\mathbf Y_{0}\equiv x^\al \ee \al+\Phi_{0}^k\mathbf X_k.
}

The set of modified vector fields containing the extra $\Phi$ coefficients is denoted by
\eq{
\mathbb Y\equiv\{\mathbf Y_0,\hdots,\mathbf Y_{7}\}.
}

Note that for any $\mathbf Y_i\in\mathbb Y$
\eq{
\mathbf Y_i(f)= \widehat {Z}_i (f) +Z_i(\phi)\WW f+\Phi_i^j \mathbf X_j(f).
}

\subsection{Weights} \label{se:we}
We consider the following set of weights
\eq{ \label{def:w}
\mathfrak Z\equiv\left\{\dfrac{tv^i-x^iv^0}{v^0},\dfrac{ x^iv^j-x^jv^i}{v^0}\right\}_{1\leq i<j\leq 3}.
}

One easily verifies that
\begin{lemma}[Commutation between weights and the transport operator]
Let $\zz\in \mathfrak{Z}$, then
\eq{\label{zz-comm}
[\Tp,\zz]= -\T(\phi)\zz - \dfrac{ Z\phi}{v^0}  - \dfrac{\left(v^0\T(\phi) + \partial_t \phi\right)}{v^0}\zz .
}
\end{lemma}
\begin{proof} The proof of this lemma is elementary. Recall
\eq{
\partial_{v^i} v^0 =  \dfrac{v_i}{v^0}, \text{ and } \Tp (v^0) = v^0 \T(\phi) - \partial_t \phi;
}
hence, one has, for a Killing vector $Z$,
\begin{align*}
\Tp\left(Z^\al v_{\al}\right) =&- (\nabla^i \phi + \T(\phi) v^i) \partial_{v^i} (Z^\alpha v_{\alpha}) \\
=&- \nabla^i \phi \left( -Z^0 \frac{v_i}{v^0} + Z_i\right)  -\T(\phi) v^i \left( -Z^0\frac{v_i}{v^0}  + Z_i\right) \\
=& - Z(\phi) + Z^0 \left(\partial_{x^0} \phi + \partial_{x^i}\phi\frac{v^i}{v^0}  \right)  - \T(\phi) \left( -Z^0\frac{(v^0)^2 - 1}{v^0}  + v^i Z_i\right) \\
=& - Z(\phi) -  \T(\phi)\zz.
\end{align*}

\end{proof}

The above lemma also holds for the weights $v^\alpha$:
$$
[\Tp,v^\alpha]= -\T(\phi)v^\alpha-\partial_{x^\alpha}\phi.
$$

\begin{lemma}[Commutation between weights and modified vector fields]
\label{lem:Yphiz}
Let $\mathbf Y\in \mathbb Y$ and $\zz\in\mathfrak Z$, then
\eq{
\YY\left(  \zz\right)=  c_1\cdot \Big[\Phi\cdot\left(1+\partial \phi \cdot \frac{q_1(x^\mu)}{(v^0)^2}\right)\Big]+c_2 \cdot p_1[\mathfrak Z]+c_3\cdot \frac{Z\phi}{(v^0)^2}\cdot q_2(x^\mu).
}
Here, we denote by $c_1$, $c_2$, $c_3$ good symbols and by $p_1,q_1, q_2$ homogeneous polynomials of first order.
\end{lemma}

\begin{proof}
The formula follows from a direct computation.
\end{proof}

\begin{remark}\label{rem:Yphiz}Below, the formula is exploited in the form
  \begin{equation*}
  \partial \phi \cdot \YY\left(  f \zz\right) = \partial \phi \cdot \left(\Phi + p_1[\mathfrak Z]\right) f +\partial \phi \cdot \zz \YY f+ R,
  \end{equation*}
up to multiplications by good symbols, and the remainder $R$ contains terms that we consider as cubic or higher order terms (cf.~discussion below).
\end{remark}

\subsection{Cubic terms and lower order terms} In what follows, we systematically decompose all the error terms into terms that have a bad behaviour either from the point of view of decay or from the point of view of regularity and terms, which, for the problem at hand, have a better behaviour, in the sense that they decay faster and do not pose any regularity problems. We refer to these terms as \emph{cubic terms} and, typically, do not write them down explicitly. Those contain elements of the forms
$$
(\partial \phi)^2 f \text{ or }(Z\phi)^2 f \text{ or }(Z\phi) \partial \phi f,
$$
as well as the equivalent terms depending on the derivatives of $f$ and $\phi$.
We emphasize that these cubic terms are not only better with respect to decay, but also in terms of regularity (so that they do not prevent from closing the estimates at top order).

For simplicity, we also sometimes refer to \emph{lower order terms} when an expression contains several terms depending on $\phi$ and its derivatives to denote certain terms containing few derivatives of $\phi$ (or none) and that do not have worse decay that the other terms. For instance, for any Killing vector field $Z$, $\frac{1}{1+u}Z( \phi)$ is lower order compared with $\partial Z(\phi)$ because it has better regularity and the same decay. We often keep explicitly only the less regular term in that case in order to have a presentation as a concise and relevant as possible.

\subsection{Action of vector fields on good symbols}
We collect a number of auxiliary formulae, which are relevant in computations below beginning with a formula for the action of the elements of $\mathbb K$ \eqref{eq:kalg} on the function $u=t-|x|$.

\begin{lemma}
Let $Z\in\mathbb K$, then
\eq{
Z u^{-1}= c u^{-1}, \quad \partial_{x^\alpha} u^{-1}=c' u^{-2}.
}
where $c$ and $c'$ are good symbols depending only on $(t,x)$.
\end{lemma}
The action of a modified field on a good symbol is given in the following lemma.

\begin{lemma}
Let $\mathbf Y\in\mathbb Y$ and let $c$ be a good symbol, then
\eq{
\mathbf Y c= c_1+c_2\cdot\frac{\Phi}{t+r} ,
}
where $c_1$ and $c_2$ are good symbols.
\end{lemma}

\subsection{Commutators of modified vector fields}
We consider the commutator of the modified algebra of vector fields, $\mathbb Y$, and generalized translations, $\mathbb E$.
We first state a preparatory lemma for the vector fields denoted by $\mathbf Z$, where translations are replaced by generalized translations.

\begin{lemma}Let $Z\in\mathbb K$. Then, there exist constant coefficients $a_{Z\beta}^\mu$ such that
$$
[\mathbf{Z}, \ee{\beta}]=a_{Z\beta}^\mu \ee{\mu},
$$
where $\mathbf{Z}=Z-Z(\phi) \WW$.
\end{lemma}
\begin{proof}
We do the proof for $Z=t\partial_{x^i} + x^i \partial_{t}$ and $\beta=j$, the other cases being similar. Recall that
\begin{eqnarray*}
\mathbf{Z}&=& t \ee i + x^i \ee t, \\
{}[\ee \mu  ,\ee{\alpha}]&=&0.
\end{eqnarray*}
Thus,
\begin{eqnarray*}
[ \mathbf{Z}, \ee j]&=& [t \ee i + x^i \ee t, \ee j] \\
&=& -\ee j(x^j) \ee t \\
&=& -\delta^j_i \ee t.
\end{eqnarray*}
\end{proof}
Next, we derive a preparatory lemma for the vector fields where in addition to the previous case the complete lifts are considered.
\begin{lemma}Let $Z\in\mathbb K$. Let $\widehat{Z}$ be the complete lift of $Z$ and $\widehat{\mathbf{Z}}=\widehat{Z}-Z(\phi) v^i \partial_{v^i}$.
Then, there exist constant coefficients $a_{Z\beta}^\mu$ and $b_{Z\beta}^{\mu,i}$ such that
$$
[\widehat{\mathbf{Z}}, \ee{\beta}]=a_{Z\beta}^\mu \ee{\mu}+b_{Z\beta}^{\mu,i}\frac{\partial_{\mu} \phi}{v^0} \VV i.
$$
\end{lemma}
\begin{proof}
This follows from $ [v^0 \partial_{v^j}, v^i \partial_{v^i}]= \frac{1}{v^0} \partial_{v^j}$.
\end{proof}

For the commutators of elements in $\mathbb X$ and generalized translations we have the following lemma.
\begin{lemma} For any $i,\beta$,
$$
[\mathbf{X}_i,\ee \beta]=-\partial_\beta (\phi) \frac{v^i}{(v^0)^3} \ee t,
$$
where $\mathbf{X}_i\in\mathbb X$.
\end{lemma}
\begin{proof}
\eq{\alg{
\/[\mathbf{X}_i,\ee \beta]&=[ \ee i + \frac{v^i}{v^0} \ee t, \ee \beta ] \\
&=-\ee \beta\left(  \frac{v^i}{v^0} \right) \ee t \\
&= -\partial_\beta (\phi) \frac{v^i}{(v^0)^3} \ee t
}}
\end{proof}
Finally, we obtain a commutator lemma for the modified fields
$\mathbf Y\in \mathbb Y$ and the generalized translations.
\begin{lemma} \label{lem:comze} Let $Z\in\mathbb K$ and denote by $\mathbf{Y}_{\al}=\widehat{\mathbf{Z}}_\al+\Phi^{\i}_{\al}\mathbf{X}_{\al}$, the corresponding modified vector field. Then, $[ \mathbf Y, \ee{} ]$ can be written as a linear combination (with coefficients which are good symbols) over the set of vector fields
\eq{
\left\{
\ee{\mu}, \,\, \frac{\partial_{\mu} \phi}{v^0} \VV i, \,\,  \ee{}(\Phi)\ee, \,\, \Phi\cdot \partial_{\mu} \phi \frac{1}{(v^0)^2} \ee {}
\right\}.
}
\end{lemma}
\begin{remark}
Note that in view of Lemma \ref{lem-foe}, it follows that the worst terms, $\ee{}(\Phi)\ee{}$, in the above commutator formula, only add a growth of the order $\sqrt{\varepsilon} \ln \rho$. This implies that the order in which one applies $\mathbf Y$ or $\ee{}$ does not matter in the estimates to follow. Note that in the sequel the energy for the distribution function contains all possible permutations (at the order N) of $\YY {}$ and $\ee {}$, so that the issue of the commutation does actually not occur.
\end{remark}

\subsection{First order commutator formula}\label{sec:firstorderf}

The purpose of this section is to establish a formula for the commutator of elements in $\mathbb Y$ or generalized translations with $\Tp$. The following commutator formula holds.

\begin{proposition}\label{prop:firstorder}
Let $\mathbf Y\in\mathbb Y$, then the commutator
\eq{
[\Tp, \mathbf Y ] f
} can be written as a linear combination over elements of the set
\begin{gather*}
\left\{\sum_{|\alpha| \leq 1}\dfrac{\Phi^\alpha}{v^0} p(\partial Z\phi, (1+u)^{-1} \partial \phi )\YY{} f,\quad \sum_{|\alpha| \leq 1}\dfrac{\Phi^\alpha}{v^0} p(\partial Z\phi, (1+u)^{-1} \partial \phi )\ee{} f,\right. \quad \\
\left.\dfrac{\Phi\cdot\Phi}{tv^0} \left(1 + \zz \right) p (\partial Z\phi, \partial \phi) \ee {} f,
\quad \Phi \dfrac{\partial \phi}{v^0}\ee {}  f, \quad  p(\partial Z\phi, (1+u)^{-1} \partial \phi ) \zz\ee {}f\right\}.
\end{gather*}
up to cubic or lower order terms; with coefficients which are all good symbols and where the $p$ are homogeneous first order polynomials of degree $1$ (i.e.~linear functions of their arguments), $\Phi$ denotes a generic coefficient appearing in \eqref{def:mvf} and $\zz$ denotes a generic weight as introduced in \eqref{def:w}. Moreover, if $\mathbf e\in\mathbb E$, then the commutator
\eq{
[\Tp,  \ee{} ] f
} can be written as a linear combination of elements in the set

\eq{
\left\{ \frac{\partial^2\phi}{v^0}\YY f, \frac{t}{v^0}\partial^2\phi \,\ee{}f,\frac{\Phi}{v^0}\partial^2\phi\, \ee{}f, \frac{\partial \phi}{v^0}\ee{} f\right\}
}
up to cubic or lower order terms with coefficients as in the above case .

\end{proposition}

\begin{proof}[Proof of Proposition \ref{prop:firstorder}]
We start from the� following commutation formula which can be easily checked.
\begin{lemma}
The following identity holds, for all vectors $\widehat{Z}_a$ in $\widehat{\mathbb{K}}$,
\eq{
[\Tp,\hZ_a]f = \left( \T(Z_a\phi)\right)\WW f + q_{a}^{j\beta}(v/v^0)(\px \be \phi)\VV j f+  p_{a}^{j\beta b}(v/v^0)(\px \be Z_b\phi)\VV j f,
}
where the $q_a^{j\beta}$ and the $p_{a}^{j\beta b}$ are homogeneous polynomials of degree at most one.
\end{lemma}

The first term in the previous commutator is a priori the worst, because it has stronger $v$ weights (included in both $\T$ and $\mathbf W$). It can be removed if we replace in $\hZ_a$ the standard translations with the generalized translations which leads to the following first improvement.
\begin{lemma} Let $\hat{Z}$ be a vector field in $\hat{\mathbb{K}}$, and consider the field $\mathbf{Z}$, where the translation have been replaced by genereralized translations. The following commutator formula holds:
\begin{align*}
[\Tp,\mathbf{\hZ}_a]f = & q_{a}^{j\beta}(v/v^0)(\px \be \phi)\VV j f+ p_{a}^{j\beta b}(v/v^0)(\px \be Z_b\phi)\VV j f  \\
&-  Z_a(\phi)\left(\frac 1{v^0} \ee t f - 2(\nabla^i\phi) \VV if \right)+C
\end{align*}
where the $q_a^{j\beta}$ and the $p_{a}^{j\beta b}$ are homogeneous polynomials of degree at most one and $C$ denotes cubic terms.
\end{lemma}
\begin{remark}
We use the formula of the lemma in the form
\begin{equation}
[\Tp,\mathbf{\hZ}_a]f = \sum_{ | \alpha | \le 1} b_{a, \alpha }\left(\partial Z^\alpha(\phi), \VV{} f\right) - Z_a(\phi) \frac 1{v^0} \ee t f +C,
\end{equation}
where the $b_{a,\alpha}$ are bilinear forms with coefficients which are good symbols.
\end{remark}
\begin{proof} Recall first that
$$
\mathbf{\hZ}_a f = \hZ_af  - Z_a(\phi) \WW  f.
$$
Thus, we have
\begin{eqnarray*}
[\Tp,\hZ_a  - Z_a\phi \WW  ]f  &=& [\Tp,\hZ_a]f  -  Z_a\phi [\Tp,\WW ]f  - \T\left(Z_a\phi \right)\WW f \\
&=& \left( \T(Z_a\phi)\right)\WW f + p_{a}^{j\beta b}(v/v^0)(\px \be Z_b\phi)\VV j f +q_{a}^{j\beta}(v/v^0)(\px \be \phi)\VV j f\\
&&  -  Z_a\phi\left(\frac 1{v^0} \ee t f- \Tp f - 2(\nabla^i\phi) \VV if \right)- \T\left(Z_a\phi \right)\WW f \\
&=&q_{a}^{j\beta}(v/v^0)(\px \be \phi)\VV j f+ p_{a}^{j\beta b}(v/v^0)(\px \be Z_b\phi)\VV j f  \\
&&-  Z_a\phi\left(\frac 1{v^0} \ee t f - 2(\nabla^i\phi) \VV if \right)+C.
\end{eqnarray*}
\end{proof}

We now use that $(1+u)\partial_t= \frac{t}{t+r}S-\frac{x^i}{t+r}\Omega_{0i}+\partial_t=\sum_{\alpha=1} a_\alpha Z^\alpha$ with the $a_\alpha$ being good symbols to replace the $\ee 0(f)$. One easily checks that in fact
$$(1+u)\ee 0 =\sum_{\alpha=1} a_\alpha \ZZ{}^\alpha $$ so that we have

\begin{lemma}Let $\hat{Z}$ be a vector field in $\hat{\mathbb{K}}$, and consider the field $\mathbf{Z}$, where the translation have been replaced by genereralized translations. There exist good symbols $a_\alpha$ such that
\begin{equation} \label{eq:czb}
[\Tp,\mathbf{\hZ}_a]f = \sum_{ | \alpha | \le 1} b_{a, \alpha }\left(\partial Z^\alpha(\phi), \VV{} f\right) + \frac{Z_a(\phi)}{v^0(1+u)}\sum_{\alpha=1}a_\alpha \ZZ{}^{\alpha}(f)+ C,
\end{equation}
\end{lemma}

We finally complete the vector field $\hat{Z}_a$ using the $\Phi_a^i$s coefficients, and states the equations satisfied by the laters.
\begin{lemma}\label{lem:firstorder} Assume that, for $\al = 0, 1,2 ,3$,
 \begin{equation} \label{eq:phi}
\Tp \left(\Phi_\al^i \right) =  -\frac{t}{v^0}\left( \sum_{ | \alpha | \le 1} b_{a,\alpha}^{\beta i} \partial_{x^\beta} Z^\alpha(\phi) +\dfrac{Z_a \phi }{1+u}\cdot  a^i \right),
 \end{equation}
 where $b_{a,\alpha}^{\beta i}$, the coefficients of $b_{a,\alpha}$, and $a^i$ are as in \eqref{eq:czb}.
Then, the following exact commutator formula holds, for $\YY \in \mathbb{Y}$:
\begin{eqnarray*}
[\Tp,\mathbf{Y}_a  ]f    &=& \sum_{ | \alpha | \le 1} s_{a, \alpha }\partial_{} Z^\alpha(\phi)  \frac{1}{v^0}\left( \mathbf{Y}_j f  - \Phi^k_j \mathbf{X}_kf  \right)  \\
&&+ \dfrac{Z_a \phi }{v^0(1+u)}\cdot\left( a_0 \left( \YY_0 f   -  \Phi^i_0 \mathbf{X}_i  (f)\right)+a_8 \ee f \right)  \nonumber\\
&& +\Phi_a^i \left( \mathbf{X}_i (\nabla\phi) .   \frac{1}{v^0}\left( \mathbf{Y}_j f -\mathbf{Z}_j f  - \Phi^k_j \mathbf{X}_kf  \right)+ \frac{\mathbf{X}_i \phi}{v^0}  \ee 0  f\right) \\
 &&
 + \left( \sum_{ | \alpha | \le 1} s_{a,\alpha} \partial_{x^i} Z^\alpha(\phi) + \dfrac{Z_a \phi }{1+u}a^i \right)\left(\dfrac{\zz_i \ee 0}{(v^0)^2} \right) f+C.
\end{eqnarray*}

\end{lemma}
\begin{proof}Recall first that
$$
\mathbf{Y}_a  = \mathbf{\hZ}_a  + \Phi^i \mathbf{X}_i  =  \mathbf{\hZ}_a  + \Phi_a^i \left(\mathbf{e}_i + \dfrac{v^i}{v^0}\mathbf{e}_0  \right).
$$
Hence, we obtain
\begin{eqnarray}
[\Tp,\mathbf{Y}_a  ]f  &=&[\Tp,\mathbf{\hZ}_a ]f + \left[\Tp, \Phi_a^i \mathbf{X}_i \right] f \nonumber \\
&=&[\Tp,\mathbf{\hZ}_a ]f   + \Phi_a^i\left[\Tp,  \mathbf{X}_i \right] f  + \Tp\left(\Phi_a^i\right)\mathbf{X}_i f.
\end{eqnarray}

Recall that from Lemma \ref{lem:xcom}, we have
\begin{eqnarray*}
[\Tp, \mathbf X_i]
&=&  \mathbf X_i (\nabla^j\phi) \VV j +\mathbf X_i (\phi)\Big[-\frac1{v^0}\ee t +\Tp +2\nabla^j\phi \VV j \Big]\\
&&\hbox{}- \frac{1}{v^0}\mathbf X_i(\phi)  \ee t .
\end{eqnarray*}

Together with the previous lemma, this gives

\begin{eqnarray}
[\Tp,\mathbf{Y}_a  ]f  &=&   \sum_{ | \alpha | \le 1} b^{\beta i}_{a, \alpha }\partial_{x^\beta} Z^\alpha(\phi) \VV{i} f - Z_a(\phi) \left(\frac 1{v^0} \ee t f - 2(\nabla\phi).\VV{}f \right)+C \nonumber\\
&&\hbox{}+ \Phi^i_a \bigg[ \mathbf X_i (\nabla^j\phi) \VV j +\mathbf X_i (\phi)\Big[-\frac1{v^0}\ee t +\Tp +2\nabla^j\phi \VV j \Big]- \frac{1}{v^0}\mathbf X_i(\phi)  \ee t \bigg]\nonumber \\
&&\hbox{} + \Tp\left(\Phi_a^i\right)\mathbf{X}_i f. \label{eq:yc1}
\end{eqnarray}

In \eqref{eq:yc1}, the $\mathbf{X}_i$ of in the last term  is now expanded as follows:
$$
\Tp \left(\Phi_a^i \right)\mathbf X_i f = \Tp \left(\Phi_a^i \right)\left(\frac{\mathbf{Z}_i}{t}+ \frac{\zz_i }{ t }\ee 0 \right)f.
$$
Finally, $\ee 0$ is substituted as follows, up to lower order terms and good symbols:
\begin{align}
Z_a\phi \ee 0 =& \dfrac{Z_a\phi  }{1+u} \mathbf{S} + \dfrac{Z_a\phi }{1+u} \boldsymbol{\Omega }_{0i} + \dfrac{Z_a\phi  }{1+u} \ee 0  \\
=& \dfrac{a_0Z_a\phi  }{1+u}\left(\YY_{0}  - \Phi_{0}^i\mathbf{X}_i f\right)  + \dfrac{a_i tZ_a\phi }{1+u}\cdot \dfrac{ \mathbf{Z }_{i}}{t} + \dfrac{a_8Z_a\phi  }{1+u} \ee 0 .
\end{align}

The last manipulation is the elimination of the term $\VV j f$ by
\begin{equation} \label{eq:velim}
\VV j f  = \frac{1}{v^0}\left( \mathbf{Y}_j f -\mathbf{Z}_j f  - \Phi^k_j \mathbf{X}_kf  \right).
\end{equation}

We now substitute in Equation \eqref{eq:yc1}
\begin{eqnarray*}
[\Tp,\mathbf{Y}_a  ]f    &=& \sum_{ | \alpha | \le 1} b^{\beta i}_{a, \alpha }\partial_{x^\beta} Z^\alpha(\phi)  \frac{1}{v^0}\left( \mathbf{Y}_i f -\mathbf{Z}_i f  - \Phi^k_i \mathbf{X}_kf  \right) +C \\
&&+ \dfrac{Z_a \phi }{v^0(1+u)}\left( a^0\left(  \YY_0 f-\Phi^i_0 \mathbf{X}_if\right) + a_8 \ee 0  \right)f  \nonumber\\
&& +\dfrac{Z_a \phi }{v^0(1+u)}a^i\cdot t \dfrac{ \mathbf{Z}_i f}{t} \\
&& +\Phi_a^i \bigg[ \mathbf X_i (\nabla^j\phi) \VV j +\mathbf X_i (\phi)\Big[-\frac1{v^0}\ee t +\Tp +2\nabla^j\phi \VV j \Big]- \frac{1}{v^0}\mathbf X_i(\phi)  \ee 0  \bigg]f\nonumber \\
 &&- \Tp \left(\Phi_a^i \right)\left(\frac{\mathbf{Z}_i}{t}+ \frac{\zz_i }{ t }\ee 0 \right) f.
\end{eqnarray*}
Finally, using \eqref{eq:phi} on $\Phi$, we are left with

 \begin{eqnarray*}
[\Tp,\mathbf{Y}_a  ]f    &=& \sum_{ | \alpha | \le 1} b_{a, \alpha }^{\beta i}\partial_{x^\beta} Z^\alpha(\phi)  \frac{1}{v^0}\left( \mathbf{Y}_i f  - \Phi^k_i \mathbf{X}_kf  \right) +C \\
&&+ \dfrac{Z_a \phi }{v^0(1+u)}\left( a_0 \left(  \YY_0 f-\Phi^i_0 \mathbf{X}_i f\right)+a_8 \ee 0 f \right)  \nonumber\\
&& +\Phi_a^i \left( \mathbf{X}_i (\nabla\phi) . \VV{} f + \frac{\mathbf{X}_i \phi}{v^0}  \ee 0 f+C \right) \\
 &&- \Tp \left(\Phi_a^i \right)\left(\frac{\zz_i }{t }\ee 0 \right) f,
\end{eqnarray*}
and substituting $\VV{}$ by \eqref{eq:velim} ends the proof of the Lemma.

\end{proof}

\paragraph{End of the proof of Proposition \ref{prop:firstorder}} We now need to sort the terms of the commutator expression contained in Lemma \ref{lem:firstorder}. We need to analyse each of the terms listed below to understand the commutator with the vector $\YY$:
\begin{gather}
 \dfrac{\partial_{}Z^\alpha(\phi)}{v^0} \YY_j f,
  \quad    \dfrac{\partial_{}Z^\alpha(\phi)}{v^0}\Phi^k_j \mathbf{X}_kf\\
  \dfrac{Z_a \phi }{v^0(1+u)}\YY_0 f,
   \quad   \dfrac{Z_a \phi }{v^0(1+u)}\Phi^i_0 \mathbf{X}_i  (f),
    \quad  \dfrac{Z_a \phi }{v^0(1+u)}\ee f \\
  \dfrac{\Phi_a^i \mathbf{X}_i (\nabla\phi)}{v^0}\mathbf{Y}_j f,
   \quad  \dfrac{\Phi_a^i \mathbf{X}_i (\nabla\phi)}{v^0}\mathbf{Z}_j f,
    \quad   \dfrac{\Phi_a^i \mathbf{X}_i (\nabla\phi)}{v^0}\Phi^k_j \mathbf{X}_kf,
    \quad \Phi_a^i \frac{\mathbf{X}_i \phi}{v^0}  \ee 0  f \\
  \dfrac{\partial_{x^i} Z^\alpha(\phi)}{v^0}\zz_i \ee 0 f,\quad \dfrac{Z_a \phi }{v^0(1+u)}\zz_i \ee 0 f.
\end{gather}
The next step consists in inserting the expression for $\mathbf{X}_i $ in terms of the weight:
$$
\mathbf{X}_i  = \frac{\mathbf{Z}_i}{t}+ \frac{\zz_i }{ t }\ee 0,
$$
so that, removing the indices, the list of terms extends to,:
\begin{gather}
 \dfrac{\partial_{}Z(\phi)}{v^0} \YY f,
 \quad  \Phi  \dfrac{\partial_{}Z\phi}{v^0} \dfrac{\mathbf{Z}f}{t},
  \quad \Phi \dfrac{\partial_{}Z\phi}{v^0} \frac{\zz }{ t}\ee {}  f,\\
  \dfrac{Z \phi }{v^0(1+u)}\YY f,
   \quad   \Phi \dfrac{Z \phi }{v^0(1+u)} \frac{\mathbf{Z}f}{t},
   \quad \Phi \dfrac{Z\phi }{v^0(1+u)} \frac{\zz }{ t}\ee {} f,
    \quad  \dfrac{Z \phi }{v^0(1+u)}\ee{} f ,\\
  \dfrac{\Phi \partial Z \phi}{t v^0}\mathbf{Y} f,
  \quad   \dfrac{ \zz \Phi \partial  \phi}{t v^0}\mathbf{Y} f,
   \quad  \dfrac{\Phi  \partial Z \phi}{tv^0}\mathbf{Z}  f,
      \quad  \dfrac{\Phi \zz \partial  \phi}{t v^0}\mathbf{Z}  f,\\
      \Phi \cdot \Phi \dfrac{  \partial Z \phi}{t v^0}\cdot \dfrac{\mathbf{Z}f}{t},
        \quad  \Phi \cdot \Phi \dfrac{  \partial Z \phi}{t v^0} \dfrac{\zz \ee{} f}{ t},
    \quad  \Phi \cdot \Phi \dfrac{ \zz  \partial  \phi}{t v^0 }\cdot \dfrac{\mathbf{Z}f}{t},
        \quad  \Phi \cdot \Phi \dfrac{ \zz  \partial  \phi}{t  v^0} \dfrac{\zz \ee{} f}{ t},\\
     \Phi \frac{Z\phi}{tv^0}  \ee {}  f,
     \quad      \Phi \frac{\zz \partial \phi}{tv^0}  \ee {}  f,
  \quad\dfrac{\partial Z\phi}{v^0}\zz \ee {} f,
  \quad \dfrac{Z \phi }{(1+u)}\zz \ee 0 f
\end{gather}

Estimating $\partial \phi \mathbf{Z}f$ as $ t\partial \phi \ee{}f$ (up to lower order terms), one obtains
\begin{gather}
 \dfrac{\partial_{}Z(\phi)}{v^0} \YY f,
 \quad  \Phi  \dfrac{\partial_{}Z\phi}{v^0}\ee{} f,
  \quad \Phi \dfrac{\partial_{}Z\phi}{v^0} \frac{\zz }{ t}\ee {}  f,\\
  \dfrac{Z \phi }{v^0(1+u)}\YY f,
   \quad   \Phi \dfrac{Z \phi }{v^0(1+u)} \ee{} f,
   \quad \Phi \dfrac{Z\phi }{v^0(1+u)} \frac{\zz }{ t}\ee {} f,
    \quad  \dfrac{Z \phi }{v^0(1+u)}\ee{} f ,\\
  \dfrac{\Phi \partial Z \phi}{t v^0}\mathbf{Y} f,
  \quad   \dfrac{ \zz \Phi \partial  \phi}{t v^0}\mathbf{Y} f,
   \quad  \dfrac{\Phi  \partial Z \phi}{v^0}\ee{} f,
      \quad  \dfrac{\Phi \zz \partial  \phi}{ v^0}\ee{}  f,\\
      \Phi \cdot \Phi \dfrac{  \partial Z \phi}{t v^0}\cdot \ee {}f,
        \quad  \Phi \cdot \Phi \dfrac{  \partial Z \phi}{t v^0} \dfrac{\zz \ee{} f}{t},
    \quad  \Phi \cdot \Phi \dfrac{ \zz  \partial  \phi}{t v^0}\cdot \ee {}f,
        \quad  \Phi \cdot \Phi \dfrac{ \zz  \partial  \phi}{t  v^0} \dfrac{\zz \ee{} f}{ t},\\
     \Phi \frac{Z\phi}{tv^0}  \ee {}  f,
     \quad      \Phi \frac{\zz \partial \phi}{tv^0}  \ee {}  f,
  \quad\dfrac{\partial Z\phi}{v^0}\zz \ee {} f,
  \quad \dfrac{Z \phi }{v^0(1+u)}\zz \ee 0 f
\end{gather}

The terms
\begin{equation}
 \dfrac{\partial_{}Z(\phi)}{v^0} \YY f, \quad  \dfrac{Z \phi }{v^0(1+u)}\YY f\quad,   \dfrac{\Phi \partial Z \phi}{t v^0}\mathbf{Y} f,   \quad   \dfrac{ \zz \Phi \partial  \phi}{t v^0}\mathbf{Y} f
\end{equation}
can all be gathered (up to lower order terms, and multiplication by good symbols) in a term of the form
\begin{equation}
\sum_{|\alpha| \leq 1}\dfrac{\Phi^\alpha}{v^0} p(\partial Z\phi, (1+u)^{-1}\partial \phi )\YY f.
\end{equation}
The terms
\begin{equation}
 \quad  \Phi  \dfrac{\partial_{}Z\phi}{v^0}\ee{} f, \quad \Phi \dfrac{\partial_{}Z\phi}{v^0} \frac{\zz }{ t}\ee {}  f, \quad  \Phi \dfrac{Z \phi }{v^0(1+u)} \ee{} f, \quad\Phi \dfrac{Z\phi }{v^0(1+u)} \frac{\zz }{ t}\ee {} f,
\end{equation}
can all be gathered (up to lower order terms, and multiplication by good symbols) in a term of the form
\begin{equation}
\sum_{|\alpha| \leq 1}\dfrac{\Phi^\alpha}{v^0} p(\partial Z\phi, (1+u)^{-1} \partial \phi )\ee{} f.
\end{equation}

The terms
\begin{equation}
      \Phi \cdot \Phi \dfrac{  \partial Z \phi}{t v^0}\cdot \ee {}f,
        \quad  \Phi \cdot \Phi \dfrac{  \partial Z \phi}{t v^0} \dfrac{\zz \ee{} f}{ t},
    \quad  \Phi \cdot \Phi \dfrac{ \zz  \partial  \phi}{t v^0}\cdot \ee {}f,
        \quad  \Phi \cdot \Phi \dfrac{ \zz  \partial  \phi}{t  v^0} \dfrac{\zz \ee{} f}{ t},
\end{equation}
can all be gathered (up to lower order terms, and multiplication by good symbols) in a term of the form
\begin{equation}
\dfrac{\Phi\cdot\Phi}{tv^0} \left(1+ \zz\right) p (\partial Z\phi, \partial \phi) \ee {} f.
\end{equation}

The two terms
\begin{equation}
     \Phi \frac{Z\phi}{tv^0}  \ee {}  f,
     \quad      \Phi \frac{\zz \partial \phi}{tv^0}  \ee {}  f,
\end{equation}
can all be gathered (up to lower order terms, and multiplication by good symbols) in a term of the form
\begin{equation}
\Phi \dfrac{\partial \phi}{v^0}\ee {}  f.
\end{equation}

The two remaining terms
\begin{equation}
  \quad\ \partial Z\phi\zz \ee {} f,
  \quad \dfrac{Z \phi }{(1+u)}\zz \ee 0 f
\end{equation}
 are of the form:
\begin{equation}
\ p(\partial Z\phi, (1+u)^{-1} \partial \phi ) \zz\ee {}f.
\end{equation}

The last part of the proof consists in considering the commutator of $\Tp$ with the vectors $\ee {}$. Considering the formula stated in Lemma \ref{lem:blockcommut}
$$
\left[\Tp, \ee \mu \right] f=\p \mu (\nabla^i\phi) \VV i f+(\p \mu \phi)\Big[-\frac1{v^0}\ee tf +\Tp f+2\nabla^i\phi \VV if\Big],
$$
we substitute in it the expression
\begin{equation*}
\VV j f  = \frac{1}{v^0}\left( \mathbf{Y}_j f -\mathbf{Z}_j f  - \Phi^k_j \mathbf{X}_kf  \right),
\end{equation*}
to obtain
\begin{eqnarray*}
\left[\Tp, \ee \mu \right] f &=&\dfrac{ \p \mu (\nabla^j\phi) }{v^0}\left( \mathbf{Y}_j f -\mathbf{Z}_j f  - \Phi^k_j \mathbf{X}_kf  \right) +  \frac{\p \mu \phi}{v^0}\ee tf \\
&&+ 2\nabla^j\phi (\p \mu \phi) \frac{1}{v^0}\left( \mathbf{Y}_j f -\mathbf{Z}_j f  - \Phi^k_j \mathbf{X}_kf  \right).
\end{eqnarray*}
The term
$$
2\nabla^i\phi (\p \mu \phi) \frac{1}{v^0}\left( \mathbf{Y}_j f -\mathbf{Z}_j f  - \Phi^k_j \mathbf{X}_kf  \right)
$$
is cubic in nature, and hence neglected. The term
\eq{
\dfrac{ \p \mu (\nabla^j\phi) }{v^0}\left( \mathbf{Y}_j f \right)
}
contributes to the commutator.
The term
\eq{
\p \mu (\nabla^j \phi) \frac{1}{v^0} \mathbf{Z}_j f
}
can schematically be written as
\eq{
\frac{t}{v^0}\partial^2 \phi \ee {}(f)
}
and contributes to the commutator. The term
\eq{
\p \mu (\nabla^j \phi) \frac{1}{v^0} \Phi^k_j \mathbf{X}_kf
}
can be schematically written
\eq{
\frac{\Phi }{v^0} \partial^2 \phi \ee{}(f).
}
and contributes to the commutator such as the last term,
\eq{
 \frac{\p \mu \phi}{v^0}\ee tf,
}
which is schematically written as
\eq{
 \frac{\p {} \phi}{v^0}\ee tf.
}
\end{proof}

\subsection{Higher order commutators}\label{sec38}

We compute the higher order commutators based on the first order formula in Proposition \ref{prop:firstorder}. We begin with some notations and preliminaries. Let $A=(A_0,\hdots, A_7)$ and $B=(B_0,\hdots, B_3)$ denote multi-indices and $\pi=\pi_{\ab{A}+\ab B}\in\Pi_{\ab A+\ab B}$ an element in the set of permutations of an ordered set with $\ab A+\ab B$ elements. Then define
\eq{
L_{A,B}^{\pi}\equiv W_{\pi(1)}\circ W_{\pi(2)}\circ\hdots\circ W_{\pi(\ab A+\ab B)},
}
where $W_i$ is the $i-th$ element of the ordered set
\eq{
\left\{ \underbrace{\YY_0,\YY_0, \hdots, \YY_0}_{A_0}, \underbrace{\YY_1,\hdots,\YY_1}_{A_1}, \hdots , \underbrace{\ee 3,\hdots,\ee 3}_{B_3}\right\}.
}
We suppress the explicit dependence on the permutation $\pi$, i.e.~write $L_{A,B}$, whenever an equation holds for any permutation of the respective type.
Moreover, we denote by
\eq{
\hat\zz_i\equiv 1+\frac{\zz_i}{\sqrt \rho}.
}
And for a multi-index $C=(C_1,C_2,C_3)$ we denote
\eq{
\ab{\hat\zz}^C\equiv \ab{\hat\zz_1}^{C_1}\cdot \ab{\hat \zz_2} ^{C_2}\cdot \ab{\hat \zz_3}^{C_3}
}

In addition, we use the notation
\eq{
q^{A,B,K}(\Phi)\equiv \sum_{K'\leq K } c_{K'} \sum_{A_i,B_i,\pi_i}c_{A_i,B_i,\pi_i}\prod_{i=1,\hdots,K'}L^{\pi_i}_{A_i,B_i}(\Phi),
}
with the additional conditions
\eq{\alg{
\sum{\ab{A_i}}\leq A,\\
\sum\ab{B_i}\leq B
}}
and $c_{K'}$, $c_{A_i,B_i,\pi_i}$ are constants which may vary and are therefore not explicitly kept track of. Note that here $A_i$ and $B_i$ denote multiindices (while above these were the components of the multi-indices $A$ and $B$). The sum is then considered over all 8-multiindices $A_i$, and 4-multiindices $B_i$ with corresponding permutation $\pi_i$.

We state two preparatory lemmata.
\begin{lemma} \label{lem:Yzcomm}
Up to good symbols and lower order terms the following formula holds
\eq{
\YY_1\circ\hdots\circ \YY_\ell \left(\frac{\Phi}{t}\zz\right)=\frac{q^{\ell,0,\ell}(\Phi)}{t}\left(1+\zz\right).
}
\end{lemma}
\begin{proof} The proof of the lemma relies on the formula stated on Lemma \ref{lem:Yphiz} and its practical use stated in Remark \ref{rem:Yphiz}.
\end{proof}
The second preparatory lemma is
\begin{lemma}\label{lem:firstorderYp}Let $p$ be a polynomial of degree at most one in each of its variables. Then, the following identity holds.
\eq{
\YY p(\partial Z\phi, u^{-1}Z\phi)=\left(1+\frac{\Phi}{t}\right)p(\partial Z^2\phi,u^{-1}Z^2\phi)
+\frac{\Phi}{t}\zz p\left(\partial^2Z\phi,u^{-1}\partial Z\phi,u^{-2}Z\phi\right)
}

\end{lemma}
\begin{remark}
Note that the second argument in the first polynomial as well as the second and third argument in the second polynomial are in fact lower order terms according to our notation. Thus, for simplicity, we write
\eq{ \nonumber
\YY p(\partial Z\phi, u^{-1}Z\phi)=\left(1+\frac{\Phi}{t}\right)p(\partial Z^2\phi )
+\frac{\Phi}{t}\zz p\left(\partial^2Z\phi\right)
}
\end{remark}

The final necessary lemma is the following identity extending the formula of Lemma \ref{lem:firstorderYp} to higher order commutators. This is the most fundamental building block of the main commutator formula. In the next lemma, we use the notation $\YY^\ell {}$, for a positive integer $\ell$. This notation means any combination of $\ell$ modified vector field $\ell$.
\begin{lemma}\label{lem:intYp} Let $\ell$ be a positive integer. The following identity holds:
\eq{
\YY^\ell \left(p\left(\partial \phi  \right)\right) = \mathlarger{\mathlarger{\sum}_{(\ast)}}   \left(1 +  \dfrac{q^{\alpha, 0, \beta}\left(\Phi\right)}{t^\beta}  \right)\cdot \left(\dfrac{q^{\gamma, 0, \delta}
\left(\Phi\right)}{t^{\frac{\delta}{2}}}  \right) p\left( \partial^{\delta +1}Z^{\beta} \phi  \right) \cdot \left(\dfrac{\zz}{\sqrt{t}}\right) ^{\mu}
}
where $p$ is always a generic first degree polynomial, and the sum is taken over the set of variables $\al, \gamma, be, \delta, \mu$ non-negative integers such that:
\eq{ (\ast)\left\{
\begin{array}{l}
  \al \leq \ell- 1,\, \be \leq \ell,\, \gamma\leq \ell- 1,\, \delta \leq \ell, \mu\leq \delta, \\
\al + \be + \gamma +  \delta  \leq \ell
\end{array}\right.}
\end{lemma}
\begin{proof} The proof of this lemma relies on Lemmata \ref{lem:Yzcomm}, \ref{lem:firstorderYp} as well as the formula for the vectors in $\mathbb X$ stated in Remark \ref{rem:formX}:
\eq{\label{eq:formX}
\mathbf X_i = \frac{\mathbf{Z}_i}{t}+ \frac{\zz_i }{ t} \mathbf{e}_0.
}
We apply the following set of rules:
\begin{itemize}
\item the commutation of $Z \partial$ into $\partial Z$ generates terms which are of lower order in the number of derivatives hitting the function $\phi$ and can consequently be suppressed according to our notation.
\item whenever the vector $\mathbf{X}$ hits the function $\phi$, one exploits the form stated in Equation \eqref{eq:formX}.
\item The vector $\ZZ {}$ built from the generalized translation can be reduced to the vector $Z$, since the terms arising are then be cubic.
\item Finally, any $\YY$ hitting the potential $\Phi$ is kept in this form.
\end{itemize}

We first perform the proof by recursion on $\ell$. The initialization of the recursion is trivial. Let us assume that the formula stated in the lemma holds at the order $\ell$. And consider the $\ell+1$-order. It is then sufficient to apply the vector $\YY$ to the generic term of the sum
$$
A= \left(1 +  \dfrac{q^{\alpha, 0, \gamma}\left(\Phi\right)}{t^\beta}  \right)\cdot \left( q^{\beta, 0, \delta}  \left(\Phi\right)   \right) p\left( \partial^{\delta +1}Z^{\beta} \phi\right) \cdot \zz ^{\mu} t^{-\frac12 \left(\delta + \mu\right)}.
$$
This is done as follows:
\begin{align}
&\YY {}A\nonumber\\
 &= \left( \dfrac{q^{\alpha + 1, 0, \beta}\left(\Phi\right)}{t^\beta}  \right)\cdot \left(q^{\gamma, 0, \delta} \left(\Phi\right) \right) p\left( \partial^{\delta +1}Z^{\beta} \phi \right) \cdot \zz ^{\mu} t^{ - \frac12 \left(\delta + \mu\right)}\label{eq:A1}\\
&+\left(1 +  \dfrac{q^{\alpha, 0, \beta}\left(\Phi\right)}{t^\beta}  \right)\cdot \left( q^{\gamma + 1, 0, \delta}  \left(\Phi\right)  \right) p\left( \partial^{\delta +1}Z^{\beta} \phi  \right) \cdot \zz ^{\mu} t^{ - \frac12 \left(\delta + \mu\right)} \label{eq:A2}\\
&+\left(1 +  \dfrac{q^{\alpha, 0, \beta}\left(\Phi\right)}{t^\beta}  \right)\cdot \left( q^{\gamma, 0, \delta}  \left(\Phi\right)  \right) \YY \left( p\left( \partial^{\delta +1}Z^{\beta} \phi \right)   \right)   \zz ^{\mu} t^{- \frac12 \left(\delta + \mu\right)}\label{eq:A3}\\
&+\left(1 +  \dfrac{q^{\alpha, 0, \beta}\left(\Phi\right)}{t^\beta}  \right)\cdot \left(  q^{\gamma, 0, \delta} \left(\Phi\right)   \right) p\left( \partial^{\delta +1}Z^{\beta} \phi   \right) \cdot  \YY \left(\zz^{\mu} t^{ -\frac12 \left(\delta + \mu\right)} \right)\label{eq:A4}
\end{align}
The term
$$
\YY \left(p\left(\partial \phi \right)  \right)
$$
is computed using Lemma \ref{lem:firstorderYp} to
\begin{align*}
\YY \left( p\left( \partial^{\delta +1}Z^{\beta} \phi \right)   \right)=&\left(1+\frac{\Phi}{t}\right)p\left( \partial^{\delta +1} Z^{\beta+1}\phi\right) +\frac{\Phi}{t}\zz p\left(\partial^{\delta+2} Z^{\beta} \phi\right).
\end{align*}
Hence term \eqref{eq:A3} can be schematically written as
\begin{align}
\left(1 +  \dfrac{q^{\alpha+1, 0, \beta}\left(\Phi\right)}{t^{\beta + 1}}  \right)\left(q^{\gamma , 0, \delta} \left(\Phi\right)  \right) p\left( \partial^{\delta +1} Z^{\beta+1}\phi\right)  \zz  ^{\mu} t^{- \frac12 \left(\delta + \mu\right)} \label{eq:A5}\\
 + \left(1 +  \dfrac{q^{\alpha, 0, \beta }\left(\Phi\right)}{t^{\beta }}  \right)\left( q^{\gamma , 0, \delta + 1} \left(\Phi\right)  \right) p\left(\partial^{\delta+2} Z^{\beta} \phi\right) \zz ^{\mu} t^{- \frac12 \left(\delta + \mu\right) + 1}\label{eq:A6} .
\end{align}
Finally, the term \eqref{eq:A4} can be rewritten by the mean of Remark \ref{rem:Yphiz}:
\eq{
 \YY \left( \zz ^{\mu} t^{- \frac12 \left(\delta + \mu\right)}\right) =  \zz ^{\mu} t^{- \frac12 \left(\delta + \mu\right)}  +   t^{- \frac12 \left(\delta + \mu\right)} \left(\Phi + \dfrac{p_1[\zz]}{v^0}
\right)\zz ^{\mu - 1}.
}
Hence, the term \eqref{eq:A4} can be rewritten as:
\begin{align}
\left(1 +  \dfrac{q^{\alpha, 0, \beta}\left(\Phi\right)}{t^\beta}  \right)\cdot \left(\dfrac{q^{\gamma, 0, \delta + 1} \left(\Phi\right)}{t^{\frac12\left(\delta+1\right)}}  \right) p\left( \partial^{\delta +1}Z^{\beta} \phi  \right) \cdot  \left(\dfrac{\zz}{\sqrt{t}}\right) ^{\mu - 1}  \label{eq:A7}\\
+\left(1 +  \dfrac{q^{\alpha, 0, \beta}\left(\Phi\right)}{t^\beta}  \right)\cdot \left(q^{\gamma, 0, \delta} \left(\Phi\right)\right) p\left( \partial^{\delta +1}Z^{\beta} \phi   \right) \cdot   \zz ^{\mu} t^{- \frac12 \left(\delta + \mu\right)} . \label{eq:A8}
\end{align}
The terms \eqref{eq:A1}, \eqref{eq:A2}, \eqref{eq:A6}, \eqref{eq:A8} and \eqref{eq:A7} are correct extensions to the order $\ell + 1$ of the formula stated in the lemma, with compatible indices since, per terms, only one is incremented. This then proves the formula at order $\ell +1$ and concludes the recursion.
\end{proof}

\begin{lemma} \label{Lonp}Let $A,B$ be two multi-indices. The following identity holds:
\eq{
L_{A, B}\left(p\left(\partial \phi  \right)\right) = \mathlarger{\mathlarger{\sum_{(\ast)}}}\left(1 +  \dfrac{q^{\alpha, P, \beta}\left(\Phi\right)}{t^\beta}  \right)\cdot \left(\dfrac{q^{\gamma, Q, \delta} \left(\Phi\right)}{t^{\frac{\delta}{2}}}  \right) p\left( \partial^{\delta +R+1}Z^{\beta} \phi \right) \cdot \left(\dfrac{\zz}{\sqrt{t}}\right) ^{\mu},
}
where the sum is taken over $\al, P, \gamma, Q,\be, \delta, R, \mu$ non-negative integers such that:
\eq{ (\ast)\left\{
\begin{array}{l}
\al \leq |A|- 1,\, \be \leq |A|,\, \gamma\leq |A|- 1,\, \delta \leq |A|, \mu\leq \delta, \\
\al + \be + \gamma +  \delta  \leq |A|\\
P+Q+R \leq |B|
\end{array}\right..}
\end{lemma}
\begin{proof} The proof is a direct consequence of Lemma \ref{lem:intYp} and relies solely on the use of the Leibniz rule applied to the term
$$
\left(1 +  \dfrac{q^{\alpha, 0, \beta}\left(\Phi\right)}{t^\beta}  \right)\cdot \left(\dfrac{q^{\gamma, 0, \delta} \left(\Phi\right)}{t^{\frac{\delta}{2}}}  \right) p\left( \partial^{\delta +1}Z^{\beta} \phi \right) \cdot \left(\dfrac{\zz}{\sqrt{t}}\right) ^{\mu}
$$
when the operator $L_{A,B}$ hits it. There exists three integers $P,Q,R$  satisfying
$$
P+Q+R \leq B
$$
such that
\eq{
L_{A, B}\left(p\left(\partial \phi  \right)\right)
}
is the sum over $P+Q+R\leq B$ of terms of the forms
\eq{
\left(1 +  \dfrac{q^{\alpha, P, \beta}\left(\Phi\right)}{t^\beta}  \right)\cdot \left(\dfrac{q^{\gamma, Q, \delta} \left(\Phi\right)}{t^{\frac{\delta}{2}}}  \right) p\left( \partial^{\delta +R+1}Z^{\beta} \phi  \right) \cdot \left(\dfrac{\zz}{\sqrt{t}}\right) ^{\mu}
}
since the term $\ee {}\left(\frac{\zz}{\sqrt{t}}\right) $ is, up to lower order terms and good symbols, of the nature of $\frac{\zz}{\sqrt{t}}$.
\end{proof}
Before we state the higher order commutators, we recall the first order commutators in a corollary of Proposition \ref{prop:firstorder}, which reduces them to their principal part and terms which are negligible with respect to the principal terms.

\begin{corollary}
Let $\mathbf Y\in \mathbb Y$, then the commutator
\eq{
[\mathbf T_\phi,\mathbf Y]f
}
can be written as a linear combination over elements of the set
\eq{\label{foc-Y}
\left\{\frac{\Phi}{v^0}p(\partial Z\phi)L_{A,B}f, \frac{\Phi^2}{v^0 t}{(1+\zz)}p(\partial Z\phi) \ee{} f, {p(\partial Z\phi) \cdot \zz}\ee{} f\right\},
}
where $A, B$ are multi-indices satisfying $|A|+|B|=1$, and additional cubic or lower order terms. Furthermore, let $\ee{}\in\mathbb E$, then the commutator
\eq{
[\Tp,\ee{}]f
}
can be written as linear combination of elements in the set
\eq{\label{foc-e}
\left\{\frac{\partial^2\phi}{v^0}\YY f,\frac{t\cdot \partial^2\phi}{v^0}\ee{}f,\frac{\Phi\cdot \partial^2\phi}{v^0}\ee{} f\right\}
}
and additional cubic or lower order terms.
\end{corollary}

We state now the higher order commutators based on the previous corollary and lemmata.

\begin{proposition}\label{prp-comm-prl}
Let $A,B$ be two multi-indices such that $|A|+|B|= N$, then the commutator
\eq{
\left[
\Tp,L_{A,B}
\right] f
}
can be written as a linear combination of the following terms, with coefficients which are good symbol and a remainder which consists only of cubic and lower order terms:
\begin{enumerate}
\item This term arises if $|A|\geq 1$.
\eq{\label{comm-first-type}\alg{
\sum_{(\ast)}\frac{1}{v^0}\Bigg\{L_{K,L}(\Phi)&\cdot
\left(1+\frac{q^{\al,R,\be}(\Phi)}{t^\be}\right)\left(\frac{q^{\gamma,S,\delta}(\Phi)}{t^{\delta/2}}\right)\left(\frac{\zz}{\sqrt{t}}\right)^\delta\cdot p\left(\partial^{\delta+I}(\partial Z^{\beta}\phi)\right)
\cdot L_{U,V }(f)
\Bigg\}
}}
where the sum is taken over $K, L, U, V$ multi-indices, and $\al, R, \gamma, S, \beta, \delta, I$ non-negative integers such that
\eq{(\ast)\left\{\alg{
&|L|+|V|+R+S+I\leq |B|\\
&|K|+|U|+\al+\be+\gamma+\delta\leq |A|\\
&\be,\delta\leq |A|-|K|-|U|\\
&\al,\gamma\leq |A|-1-|K|-|U|\\
&1\leq |U|+|V|
}\right.}

\item This term arises if $|A|\geq 1$.
\eq{\label{comm-prl-3}\alg{
\sum_{(\ast)}\frac{1}{v^0t}\Bigg\{\left(1+q^{K,L,1}(\Phi)+\zz\right)&\cdot
\left(1+\frac{q^{\al,R,\be}(\Phi)}{t^\be}\right)\left(\frac{q^{\gamma,S,\delta}(\Phi)}{t^{\delta/2}}\right)\left(\frac{\zz}{\sqrt{t}}\right)^\delta\cdot \left(L_{P,Q}(\Phi^2)\right)\\
&\cdot p\left(\partial^{\delta+I}(\partial Z^{\beta}\phi)\right)
\cdot L_{U,V+1 }(f)
\Bigg\}
}}
where the sum is taken over $U, V$ multi-indices, and $K, L,\al, R, \gamma, S, \beta, \delta, I$ non-negative integers such that
\eq{(\ast)\left\{\alg{
&L+Q+V+R+S+I\leq |B|\\
&K+P+|U|+\al+\be+\gamma+\delta\leq |A|-1\\
&\be,\delta\leq |A|-1-K-P-|U|\\
&\al,\gamma\leq |A|-2-K-P-|U|
}\right.}

\item This term arises if $|A|\geq 1$.
\eq{\label{comm-prl-4}\alg{
\sum_{(\ast)}\Bigg\{\left(1+q^{K,L,1}(\Phi)+\zz\right)&\cdot
\left(1+\frac{q^{\al,R,\be}(\Phi)}{t^\be}\right)\left(\frac{q^{\gamma,S,\delta}(\Phi)}{t^{\delta/2}}\right)\left(\frac{\zz}{\sqrt{t}}\right)^\delta\cdot p\left(\partial^{\delta+I}(\partial Z^{\beta}\phi)\right)
\cdot L_{U,V+1 }(f)
\Bigg\}
}}
where the sum is taken over $U, V$ multi-indices, and $K, L,\al, R, \gamma, S, \beta, \delta, I$ non-negative integers such that
\eq{(\ast)\left\{\alg{
&L+V+R+S+I\leq |B|\\
&K+|U|+\al+\be+\gamma+\delta\leq |A|-1\\
&\be,\delta\leq |A|-1-K-|U|\\
&\al,\gamma\leq |A|-2-K-|U|
}\right.}

\item This term arises if $|B|\geq 1$.
\eq{\label{comm-prl-5}\alg{
\sum_{(\ast)}\frac{1}{v^0}\Bigg\{&
\left(1+\frac{q^{\al,R,\be}(\Phi)}{t^\be}\right)\left(\frac{q^{\gamma,S,\delta}(\Phi)}{t^{\delta/2}}\right)\left(\frac{\zz}{\sqrt{t}}\right)^\delta \cdot p\left(\partial^{\delta+I}(\partial^2 Z^{\beta}\phi)\right)
\cdot L_{U+1,V }(f)
\Bigg\}
}}
where the sum is taken over $U, V$ multi-indices, and $K, L,\al, R, \gamma, S, \beta, \delta, I$ non-negative integers such that
\eq{(\ast)\left\{\alg{
&|V|+R+S+I\leq |B|-1\\
&|U|+\al+\be+\gamma+\delta\leq |A|\\
&\be,\delta\leq |A|-|U|\\
&\al,\gamma\leq |A|-1-|U|
}\right.}

\item This term arises if $|B|\geq 1$.
\eq{\label{comm-prl-6}\alg{
\sum_{(\ast)}\frac{t}{v^0}\Bigg\{&
\left(1+\frac{q^{\al,R,\be}(\Phi)}{t^\be}\right)\left(\frac{q^{\gamma,S,\delta}(\Phi)}{t^{\delta/2}}\right)\left(\frac{\zz}{\sqrt{t}}\right)^\delta\cdot p\left(\partial^{\delta+I}(\partial^2 Z^{\beta}\phi)\right)
\cdot L_{U,V+1 }(f)
\Bigg\}
}}
where the sum is taken over $U, V$ multi-indices, and $\al, R, \gamma, S, \beta, \delta, I$ non-negative integers such that
\eq{(\ast)\left\{\alg{
&|V|+R+S+I\leq |B|-1\\
&|U|+\al+\be+\gamma+\delta\leq |A|\\
&\be,\delta\leq |A|-|U|\\
&|\al|,|\gamma|\leq |A|-1-|U|
}\right.}

\item This term arises if $|B|\geq 1$.
\eq{\label{comm-eighth-type}\alg{
\sum_{(\ast)}\frac1{v^0}\Bigg\{L_{K,L}(\Phi)&\cdot
\left(1+\frac{q^{\al,R,\be}(\Phi)}{t^\be}\right)\left(\frac{q^{\gamma,S,\delta}(\Phi)}{t^{\delta/2}}\right)\left(\frac{\zz}{\sqrt{t}}\right)^\delta\cdot p\left(\partial^{\delta+I}(\partial^2 Z^{\beta}\phi)\right)
\cdot L_{U,V+1}(f)
\Bigg\}
}}
where the sum is taken over $K,L,U, V$ multi-indices, and $\al, R, \gamma, S, \beta, \delta, I$ non-negative integers such that
\eq{(\ast)\left\{\alg{
&|L|+|V|+R+S+I\leq |B|-1\\
&|K|+|U|+\al+\be+\gamma+\delta\leq |A|\\
&\be,\delta\leq |A|-|K|-|U|\\
&\al,\gamma\leq |A|-1-|K|-|U|
}\right.}
\end{enumerate}

\end{proposition}

\begin{proof}
We consider an arbitrary combination of vector fields
\eq{
L_{A,B}=X_1\cdots X_N,
}
of length $N$ (for less fields it works identical) where $X_i$ stands for an element in the set $\mathbb Y\cup\mathbb E$ and $X_0=1$ for convenience. Then the general commutator formula
\eq{\label{gen-comm}
[\Tp, L_{A,B}] f=\sum_{k=1}^N X_1\hdots X_{k-1} [\Tp,X_k] X_{k+1}\hdots X_N f
}
allows for a decomposition into summands, each of which contains precisely one first order commutator. The six different types of terms in the proposition then occur according to the presence of at least one modified field ($|A|\geq 1$)  or at least one generalized translation ($|B|\geq 1$). If a condition, $|A|\geq 1$ or $|B|\geq 1$, is satisfied, then a summand of the sum \eqref{gen-comm} with $X_k\in\mathbb Y$ or $X_k\in\mathbb E$ occurs, respectively. Then, considering terms in the decomposition \eqref{gen-comm}, the fields $X_{k+1}\hdots X_N$ contribute to the $L_{U,V}f$ terms in the expressions \eqref{comm-first-type} - \eqref{comm-eighth-type} while the fields $X_1\hdots X_{k-1}$ either contribute to the latter terms or act on the commutator term in the sense of the Leibniz rule. To evaluate the case when these fields act on the commutator we consider the two types of terms in \eqref{foc-Y} and the three types of terms in \eqref{foc-e}. The action of these fields on polynomials is evaluated using Lemma \ref{Lonp} and the action on weights is evaluated using Remark \ref{rem:Yphiz}. This finishes the proof.
\end{proof}

We reduce the previous lemma to a representation of the commutator which only consists of two types of terms. In particular, we suppress the dependence of the second argument of the polynomial in $\phi$, since it is of the same decay but better regularity. The remark following the corollary explains how the latter should be read.

\begin{corollary}\label{cor-comm-f}
Let $|A|+|B|= N$, then the commutator
\eq{
\left[
\Tp,L_{A,B}
\right] f
}
consists of terms of the following type up to good symbols and bulk terms.

\eq{\label{comm-1-f}\alg{
\sum_{(\ast)}\Bigg\{&\left(\frac{q^{\alpha,S,\delta+\beta+\mu+2\gamma}(\Phi)}{t^{(\delta+\beta+2\sigma)/2}}\right)\left(\frac{\zz^{\delta+\nu}}{\sqrt{t}^{\delta}(v^0)^{1-\sigma}}\right)\cdot t^{\kappa}\\
&\cdot p\left(\partial^{\delta+I+1+\gamma-\sigma+\kappa+\lambda(1-\mu)}Z^{\beta-(\gamma-\sigma+\kappa+\lambda(1-\mu))}\phi\right)
\cdot L_{U',V'}(f)
\Bigg\}
}}
where the sum is taken on the variables $U', V'$ multi-indices, and $\al, S, \delta, \be, \mu, \gamma, \sigma, \kappa, I, \lambda$ such that
\eq{\label{comm-T-cond}(\ast)\left\{\alg{
& |U'| = U+ \lambda, |V'| = V - \lambda +1,\\
&V+S+I\leq B\\
&U+\al+\be+\delta+\nu\leq A\\
&\alpha+S+U+V+\delta+I+\beta\leq N\\
&1\leq\delta+I+\beta\\
&\mu+\nu\leq 1\\
&\mu=\gamma=1\,\Rightarrow \sigma=1\\
&\sigma\leq\gamma\leq1\\
&\kappa=1 \Rightarrow \mu=\gamma=\nu=\lambda=\sigma=0, \kappa\leq 1\\
&\lambda=1 \Rightarrow \gamma=\nu=\sigma=0, \lambda\leq 1\\
&\lambda=\mu=1\Rightarrow U+\lambda+\alpha+\beta+\delta\leq A
}\right.}
where we define $Z^{-m}=\mathrm{id}$ for $m\geq 0$ for convenience.
\end{corollary}

\begin{proof}
Formula \eqref{comm-1-f} in combination with the conditions \eqref{comm-T-cond} includes all terms listed in Proposition \ref{prp-comm-prl}. We have suppressed some structure to reduce the complexity of the formula. The corresponding terms are absorbed into the bulk terms, which are better in regularity or in decay in comparison with the terms that are explicitly listed. \\

In particular, the indices $\mu,\nu,\sigma,\gamma,\lambda$ and $\kappa$ encode the different cases as we point out in the following. We refer to them for convenience as \emph{switch indices}. If $\kappa=1$, this encodes the case \eqref{comm-prl-6}. Note that the condition on $\kappa$ ensures that the extra $t$ factor does not occur in any other case. Case \eqref{comm-first-type} is contained when $\mu=1$, $\lambda\leq 1$ and all other switch indices vanish. Case \eqref{comm-prl-3} occurs if either $\nu=1$, $\gamma=1$, $\sigma=1$ and $\mu=0$ or if $\nu=0$, $\mu=1$, $\gamma=1$ and $\sigma=1$.
Case \eqref{comm-prl-4} occurs when either $\mu=1$ and all other switch indices vanish or when $\mu=0$ and $\gamma=\sigma=\nu=1$. Note that we have used the fact that the $q$ expression contains constant terms in this case not to additionally complicate the formula.
Case \eqref{comm-prl-5} occurs when $\lambda=1$. Finally, case \eqref{comm-eighth-type} occurs when $\mu=1$ and $\sigma=0$, where we ignore that there are always at least two derivatives acting on $\phi$, since it simplifies the structure of the formula and this better behaviour is not required in the estimates to follow.
\end{proof}

\begin{remark}\label{rem-comm-f}
The application of the previous formulas require only certain aspects of its structure. We point out how the index conditions can be interpreted to reveal these aspects.

The first relevant structure of expression \eqref{comm-1-f} is the fact that it reads as a sum of products with three factors, which contain derivatives of $\Phi$, $\phi$ and $f$, respectively. For the energy estimates it is relevant that when one of these terms contains a high number of derivatives, the others are low. In particular, an upper bound for their sums has to be provided. Note that the maximal number of derivatives acting on $\Phi$ is $\al+S$, on $\partial\phi$ is $\delta+I+\beta$ and on $f$ is $U+V+1$. The desired upper bound results from the third line of \eqref{comm-T-cond} which yields
\eq{\label{small-big-cond}
\underbrace{\al+S}_{\Phi}+\underbrace{U+ V}_{f}+\underbrace{\delta+I+\be}_{\phi}\leq N.
}
The second important comment concerns the fact that the indices $\mu,\nu,\sigma,\gamma$ and $\lambda$ are at most equal to one and are mutually exclusive in the sense of the conditions. They ensure that different structural cases which nevertheless both lead to sufficient decay are captured by the formula. We explain these cases below.

The condition $\mu+\nu\leq 1$ ensures that either the weights in the second factor are compensated by the $\sqrt t$ in the denominator ($\nu=0$, $\mu=1$) or that in the complementary case ($\nu=1, \mu=0$) the factor of $\Phi$ in the first term given by $\mu$, which is not compensated in the respective denominator, is not present.

A third important aspect concerns the absorption of weights $\zz$, which are renormalized below by a factor $\rho^{-1/2}$. As we discuss later in detail in our eventual energies, an operator of the type $L_{A,B}$ is weighted by a power of the weights $\zz$ of the order $N-|A|+3$. We read from the commutator formula that the number of weights that appear are of the order $\delta+\nu$. It is then important to assure that the conditions on the exponents and indices assure that the weight coming from the original $L_{A,B}$ in the commutator (which is $N-A+3$) and the additional factor of order $\delta+\nu$ can in fact be absorbed by the new operator $L_{U+\lambda,V+1-\lambda}$. This is the case if indeed
\eq{
N+3-|A|+\delta+\nu\leq N+3-(U+\lambda).
}

For $\lambda=0$ this follows immediately from the second line of the conditions. If $\lambda=1$ we cannot absorb all weights into $L_{U+\lambda,V+1-\lambda}$, but have to estimates one weight pointwise.

Finally,  the indices $\gamma$ and $\sigma$ are relevant only when $\gamma=1$, which provides two additional factors of $\Phi$ or its derivatives in the first term. Then, if also $\sigma=1$, this factor is compensated by the denominator of the first term. If however, $\sigma=0$, then in the second term there is a $(v^0)^{-1}$ and in the fourth term there is an additional $\partial$ acting on $\phi$, which is used in combination to compensate for the occurring $\Phi$ factors due to $\gamma$ in the first term.

Other details of the formula seem to be not relevant for the estimates in the remainder.
\end{remark}


\section{Commutation formula for the wave equation} \label{sec:commwave}
We prove in this section several commutation formulae for the wave equation \eqref{eq:wpmsv}.

\subsection{First order commutator formula} \label{sec:commwave1}

We first commute with a vector field $Z \in \mathbb K$ which brings

$$
\square Z(\Phi)= c_Z \square \Phi + \int_v Z(f) \frac{dv}{v^0},
$$
where $c_Z$ is $0$ unless $Z$ is the scaling vector field.
Replacing $Z$ by its complete lift $\widehat Z$, except for the scaling where we keep it, and $\square \Phi$ by the equation, we then obtain
$$
\square Z(\Phi)= c_Z \int_v f \frac{dv}{v^0} + \int_v \widehat Z (f) \frac{dv}{v^0}.
$$
We now rewrite $\widehat{Z}(f)$ using the modified vector fields

\begin{eqnarray*}
\square Z(\Phi)&=& c_Z \int_v f \frac{dv}{v^0} + \int_v \left(\mathbf Y - Z(\phi)v^k \partial_{v^k} -\Phi^i \mathbf X_i \right)(f)\frac{dv}{v^0}\\
&=& \sum_{|\alpha| \le 1} c_\alpha \int_v \mathbf Y^\alpha(f) \frac{dv}{v^0}- Z(\phi)\int_v  \frac{v^k}{v^0} \partial_{v^k}fdv -\int_v \Phi^i \mathbf X_i(f) \frac{dv}{v^0}.
\end{eqnarray*}
The second term can be integrated by parts in $v$, leaving
\begin{eqnarray}
\square Z(\Phi)
&=& \sum_{|\alpha| \le 1} c_\alpha \int_v \mathbf Y^\alpha(f) \frac{dv}{v^0}+ Z(\phi)\int_v s_Z fdv -\int_v \Phi^i \mathbf X_i(f) \frac{dv}{v^0}. \label{eq:fob}
\end{eqnarray}
where $c_\alpha$ are constants and $s_Z:=s_Z(v)$ is a good symbol in $v$.

The first two terms are fine but, because of the growth of $\Phi$, the last term is not good enough as written. We rewrite it as

\begin{eqnarray*}
\int_v \Phi^i \mathbf X_i (f) \frac{dv}{v^0} &=& \int_v \Phi \left( \frac{\ZZ{}}{t}+ \frac{\zz}{t} \ee t \right)(f) \frac{dv}{v^0}  \\
&=& \int_v \frac{\Phi}{t} \left( \ZZ {}+ \Phi\cdot\mathbf X -\Phi\cdot\mathbf X \right)(f) \frac{dv}{v^0} + \int_v \Phi  \frac{\zz}{ t} \ee t (f) \frac{dv}{v^0} \\
&=& \int_v \frac{\Phi}{t} \left( \ZZ {}+ \Phi\cdot\mathbf X \right)(f) \frac{dv}{v^0}- \int_v \frac{\Phi}{t}  \Phi\cdot\mathbf X (f) \frac{dv}{v^0} + \int_v \Phi \frac{\zz}{ t} \ee t (f) \frac{dv}{v^0}.
\end{eqnarray*}
In the second term, the $| \Phi|^2$ gives a growth of $\rho$ which is compensated by the $1/t$. In the third term, the $\Phi \cdot \zz$ induces a growth of $\rho$, again compensated by the $1/t$. On the other hand, unless $\ZZ{}$ is a generalized translation or the scaling vector field, the first term is still not good enough because we have not yet replaced $\ZZ{}$ by its complete lift.

Let us assume that $Z$ is a Lorentz boost (the case of rotations can be treated similarly) so that (dropping indices)
$$
\mathbf Y= \ZZ{} + v^0 \partial_v + \Phi\cdot\mathbf X.
$$

This gives
\begin{eqnarray*}
\int_v \frac{\Phi}{t} \left( \ZZ {}+ \Phi\cdot\mathbf X \right)(f) \frac{dv}{v^0} &=& \int_v \frac{\Phi}{t} \left( \ZZ {}+ v^0 \partial_v+ \Phi\cdot\mathbf X \right)(f) \frac{dv}{v^0}- \int_v \frac{\Phi}{t} \left( v^0 \partial_v \right)(f) \frac{dv}{v^0} \\
&=& \int_v \frac{\Phi}{t}\mathbf Y(f) \frac{dv}{v^0}+   \int_v \left( \frac{\partial_v \Phi}{t}  \right)(f) dv
\end{eqnarray*}
Now, as we will see later, the decay properties of the first term are sufficient; the second is also in fact fine, because $\partial_v$ has better properties when applied to $\Phi$ than $f$, since the source term in the equation for $\Phi$ is essentially a function of $(t,x)$ only (up to a pure weight in $v$).
We nonetheless rewrite the last term to replace $\partial_v$ by vector fields in our algebra of commutators.

\begin{eqnarray*}
\int_v \left( \frac{\partial_v \Phi^i}{t}  \right)(f) dv&=&\frac{1}{t} \int_v \left( \mathbf Y-\ZZ{} - \Phi\cdot \mathbf X \right)(\Phi)\cdot  f \frac{dv}{v^0} \\
&=&\frac{1}{t} \int_v \mathbf Y(\Phi)\cdot f \frac{dv}{v^0}- \frac{1}{t}\int_v \Phi \cdot \mathbf X(\Phi) \cdot f\frac{dv}{v^0} - \int_v \frac{\ZZ {}(\Phi)}{t} \cdot f \frac{dv}{v^0}
\end{eqnarray*}

The first two terms on the right-hand side are clearly fine in terms of decay. For the last term, we use that
$$\frac{\mathbf Z_i}{t}=\mathbf X_i-  \frac{\zz_i}{ t} \ee 0, $$ leading to

$$
\int_v \frac{\ZZ {}(\Phi)}{t} \cdot f \frac{dv}{v^0}= \int_v \mathbf X (\Phi) \cdot f \frac{dv}{v^0}- \int_v \frac{\zz_i}{ t} \ee 0 (\Phi) \cdot f \frac{dv}{v^0}.
$$
The first term on the right-hand side is now fine thanks to the fact that  $\mathbf X (\Phi)$ has only a $\log \rho$ growth (see Lemma \ref{lem-foe}), and the last term is also clearly fine in terms of decay.

We have thus obtained the following commutation formula:
\begin{lemma}\label{lem:wavelow} For any $Z$ vector field, $Z(\Phi)$ solves a wave equation of the form
$$
\square Z(\Phi) = \sum_i \int_v F_i \frac{dv}{v^0},
$$
where the $F_i$ are, modulo multiplication by good symbols, of the form

\begin{itemize}
\item $\mathbf Y^\alpha(f)$, for $|\alpha| \le 1$,
\item $Z(\phi) \cdot f,$
\item $\Phi\cdot\Phi\frac{1}{t} \cdot \ee{}(f),$
\item $ \Phi \frac{\zz}{t} \ee{}(f),$
\item $\frac{\mathbf Y(\Phi)}{t} \cdot f,$
\item $ \frac{\Phi}{t} \mathbf Y(f),$
\item $\frac{1}{t} \Phi \mathbf X(\Phi) f,$
\item $\mathbf X(\Phi) f,$
\item $\frac{\zz}{ t} \ee {} (\Phi) \cdot f.$
\end{itemize}

\end{lemma}
Many terms can be regrouped together so that the above lemma can be simplified\footnote{Note that it is only important to distinguish $\ee{}$ from a general $\mathbf Y$ when a $\zz$ weight is involved.} as

\begin{lemma} For any $Z$ vector field, $Z(\Phi)$ solves a wave equation of the form
$$
\square Z(\Phi) = \sum_i \int_v F_i \frac{dv}{v^0}
$$
where the $F_i$ are, modulo a multiplication by good symbols, of the form

\begin{itemize}
\item $\frac{1}{t^j} q^{2j, \gamma}(\Phi)\mathbf Y^\alpha(f)$, for $0 \le j \le 1$, \,\,$| \gamma| + | \alpha| \le 1$, 
\item $\ee{}(\Phi) f$,
\item $Z(\phi) f$,
\item $\Phi \frac{\zz}{ t} \ee{}f$.
\end{itemize}
where the definitions of the $q$ forms is given below.
\end{lemma}

In this section, we use the notation
$
q^{d, k} (\Phi)
$
to denote a product of at most $d$ coefficients $\Phi$  containing a total of $k$ commutation vector fields, i.e.
$$
q^{d, k} ( \Phi)=  \prod_{j=1,.., d' \le d} \YY^{\rho_j} (\Phi),
$$
where each $\rho_j$ is therefore a multi-index of size $|\rho_j|$ and such that $\sum |\rho_j| = k$.
Let us define similarly $Q^{d, k}(\phi)$ as
$$
Q^{d, k} ( \phi)= \prod_{j=1,.., d' \le d} Z^{\rho_j} (\phi)
$$
where again $\sum |\rho_j| = k$.

\subsection{Higher order commutator formula}\label{sec:commwave2}
We now obtain the general higher order form:
\begin{lemma} For any $N$, $Z^N(\phi)$ solves a wave equation of the form
\begin{equation} \label{eq:comwave}
\square Z^N(\phi) = \sum_i \int_v F_i \frac{dv}{v^0}
\end{equation}
where the $F_i$ are, modulo a multiplication by good symbols, of the form

$$
\left( \frac{\zz}{ t} \right)^{r_1} p^{d_1, k_1} ( \ee{}(\Phi) ) \frac{q^{2d_2+r, k_2}(\Phi )}{t^{d_2+r_2}} Q^{d_3+r_3,k_3}\left(  \phi \right)[ \mathbf Y^{\alpha}(f) ]
$$
where $r_1+r_2+r_3= r$ , $r +d_1+d_2+d_3 \le N$,  $|k|+|\alpha|+d_1 \le N$, $r + |\alpha|-|\alpha(\ee{})| + |k|-|k(\ee{})| \le N$, $k_3(\ee{}) \ge r_3$, with $|k|:=k_1+k_2+k_3$ and with $k(\ee{})$, $k_3(\ee{})$, $\alpha(\ee{})$ denoting the total number of $\ee{}$ in the corresponding terms.
\end{lemma}
\begin{proof}
For $N=1$, we have the possible error terms are given by $F_i$, $1  \le i \le 4$, with

\begin{itemize}
\item $F_1=\frac{1}{t^j} q^{2j, \gamma}(\Phi)\mathbf Y^\alpha(f)$, for $0 \le j \le 1$, \, $| \gamma| + | \alpha| \le 1$, 
\item $F_2=\ee{}(\Phi) f$,
\item $F_3=Z(\phi) f$,
\item $F_4=\Phi \frac{\zz}{ t} \ee{}f$.
\end{itemize}

$F_1$ is contained in the general formula, taking $r=d_1=d_3=0=k_1=k_3$, $j=d_2$, $k_2=\gamma$.

$F_2$ is contained in the general formula, taking $r=d_2=d_3=k_2=k_3=\alpha=0$, $d_1=1$, $k_1=0$.

$F_3$ is contained in the general formula, taking $r=d_1=d_2=k_1=k_2=\alpha=0$, $k_3=d_3=1$.

$F_4$ is contained in the general formula, taking $r_1=r=1$, $r_2=r_3=d_1=d_2=d_3=|k|=0$, $|\alpha|=|\alpha( \ee{} )|=1$ (so that $\mathbf Y^\alpha=\ee{}$).

Assume now that the general formula holds true for some $N$. In view of the previous lemma and the formula at rank $N$, we must consider the terms
\begin{itemize}
\item $G_1= \frac{1}{t^j} q^{2j, \gamma}(\Phi)\mathbf Y^\alpha(g)$, for $0 \le j \le 1$, \, $| \gamma| + | \alpha| \le 1$
\item $G_2= \ee{}(\Phi) g $%
\item $G_3= Z(\phi) g$
\item $G_4= \Phi \frac{\zz}{ t} \ee{}g $
\end{itemize}
with $$
g=\left( \frac{\zz}{ t} \right)^{r_1} p^{d_1, k_1} ( \ee{}( \Phi) ) \frac{q^{2d_2+r, k_2}({\Phi})}{t^{d_2+r_2}} Q^{d_3+r_3,k_3}\left(  {\phi} \right)[ \mathbf Y^{\alpha}(f) ]
$$
where $r_1+r_2+r_3= r$ , $r +d_1+d_2+d_3 \le N$,  $|k|+|\alpha|+d_1 \le N$, $r + |\alpha|-|\alpha(\ee{})| + |k|-|k(\ee{})| \le N$, $k_3(\ee{}) \ge r_3$.

We start with the $G_2$ term, which can be written as
$$
G_2=\left( \frac{\zz}{ t} \right)^{r_1} p^{d_1+1, k_1} ( \ee{}( \Phi) ) \frac{q^{2d_2+r, k_2}(\Phi)}{t^{d_2+r_2}} Q^{d_3+r_3,k_3}\left(  \phi \right)[ \mathbf Y^{\alpha}(f) ]
$$
and is therefore of the required form. Similarly, the $G_3$ term is clearly of the required form.

To evaluate the contribution of the $G_4$ term, we must distribute $\ee{}$ on each of the terms in the formula.

First note that when $\ee{}$ hits a good symbol $s$, we have
$$
\ee{}(s)=\partial_{t,x} s'+ \partial \phi\cdot  s'',
$$
where $s',s''$ are also good symbol. The resulting terms are then clearly of the correct form.

 When $\ee{}$ hits $\mathbf Y^\alpha(f)$ or any any of the $p^{d_1, k_1} ( \ee{}( \Phi) )$ or $q^{2d_2+r, k_2}({\Phi})$, the resulting terms are also clearly of the required form.
Since $\ee{}= \partial -\partial(\phi) v^k  \partial_{v^k}$, we have
$$
\ee{} (Z^\beta (\phi))= \partial Z^\beta (\phi),
$$
which implies that when $\ee{}$ hits $Q^{d_3,k_3}\left(  {\phi} \right)$, the resulting term is clearly of the required form.

Finally, note that $\frac{\zz}{v^0 t}$ is a good symbol so that
we have, for some good symbols $s'$ and $s''$,
$$
\ee{}  \left( \frac{\zz}{ t} \right)^{r_1} = r_1 \left(\frac{\zz}{t} \right)^{r_1-1} \left(\frac{s'}{t}+ \partial \phi\cdot s''\right),
$$
from which it follows that the contribution of this term and thus of $G_4$ and $G_2$ (since we have checked all possibilities) are also of the required form.

In order to compute the contribution of the $G_1$ term, we must distribute a $\mathbf Y$ on each of the terms of $g$.

First note that when $\mathbf Y$ hits a good symbol $s$, we have
$$
\mathbf Y(s)= s' + Z(\phi) s''
$$
for some good symbols $s'$ and $s''$ so that the resulting terms are of the required form.

When $\mathbf Y$ hits the $\zz$ weight factor, we simply compute
$$
\mathbf Y \frac{\zz}{ t}= - \frac{\zz}{t }\frac{\mathbf Y(t)}{t} + \frac{\mathbf Y(\zz)}{t}.
$$
For the first term, we have
\begin{eqnarray*}
\frac{\zz}{t }\frac{\mathbf Y(t)}{t} = \frac{\zz}{t } \left(\frac{Z(t)}{t}+ \Phi\cdot \frac{X(t)}{t} \right)
\end{eqnarray*}
and the resulting terms are clearly of the required form.

The second term $\frac{\mathbf Y( \zz)}{t}$, can be written as a linear combination (with coefficients that are good symbols) of terms of the form
$$
\frac{\zz}{ t}, \,\, Z(\phi),  \,\, \frac{\Phi}{t}, \,\, \Phi\cdot \partial \phi
$$
and the resulting terms are all of the required form.

When a $\mathbf Y$ hits the $\mathbf Y^\alpha$ or any of the $p^{d_1, k_1} ( \ee{}( \Phi) )$ or $q^{2d_2+r, k_2}({\Phi})$, the resulting terms are clearly of the required form.

Finally, when $\mathbf Y$ hits the $Q(\phi)$ form, we use that

\begin{eqnarray*}
\mathbf Y( Z^\beta(\phi) ) &=& \left( Z - Z(\phi) v^k \partial_{v^k} + \Phi\cdot\mathbf X  \right) ( Z^\beta (\phi)) , \\
&=& Z^{\beta+1}(\phi) + \Phi\cdot\mathbf X  \left( Z^\beta(\phi) \right) \\
&=& Z^{\beta+1}(\phi) + \frac{ \Phi}{t} \left( \ZZ{} + \zz  \ee t \right) \left( Z^\beta(\phi) \right) \\
&=& Z^{\beta+1}(\phi) + \frac{ \Phi}{t} Z^{\beta+1}(\phi)+ \Phi. \frac{\zz}{ t } \partial_t Z^\beta(\phi).
\end{eqnarray*}
The resulting terms are then all the correct form. For instance, the middle term on the right-hand side of the last equation, if we denote by $'$ the new indices, we have $d_2'=d_2+1$ and $|k'|=|k|+1$, while $r$ is kept unchanged. For the last term, we have $r_1'=r_1+1$, $r'=r_2$ and $|k'|=|k|+1$ (in that case $|k(\ee{})'|=|k(\ee{})|+1$, but we do not need this extra information so we are not keeping track of it).

From this, it is easy to see that the resulting terms are all of the required form.
\end{proof}

\subsection{Top order commutator formula} \label{se:tocf}
According to the previous higher order commutator formula, after $N$ commutations, the source term in the wave equations contains terms based on $\mathbf{Y}^\alpha( \Phi) f$, where $|\alpha|=N$. Recall that the source term in the transport equation satisfied by $\Phi$ contains a term proportional to $\partial Z(\phi)$, for some $Z$. Thus, controlling the source term in the transport equation for $\mathbf{Y}^\alpha (\Phi)$ with $|\alpha|=N$ would require a bound on $Z^{\alpha'}(\phi)$, with $|\alpha'|=N+1$, which would therefore not close. Thus, at top order, we need to decompose the source term in the wave equation differently. More precisely, we go back to the derivation of the first order commutator formula, equation \eqref{eq:fob}
\begin{eqnarray}
\square Z(\Phi)
&=& \sum_{|\alpha| \le 1} c_\alpha \int_v \mathbf Y^\alpha(f) \frac{dv}{v^0}+ Z(\phi)\int_v s_Z fdv -\int_v \Phi^i \mathbf X_i(f) \frac{dv}{v^0}.
\end{eqnarray}
In terms of regularity, there is no loss at the moment, since we have commuted once and all terms on the right-hand side are at the same level of regularity. The loss occurs in the extra manipulation we make to improve the naive estimate
$$
\left|\int_v \Phi \mathbf X(f) \frac{dv}{v^0} \right| \le \varepsilon^{1/2} \rho^{1/2}  \int_v |\ee{}(f)|  \frac{dv}{v^0},
$$
which loses a $\rho^{1/2}$. At top order, we are thus forced to use such a naive estimate.

Thus, the commutation formula that we apply is the following.

\begin{lemma} For any $N$, $Z^N(\phi)$ solves a wave equation of the form \label{lem:topw}
\begin{equation} \label{eq:comwavetop}
\square Z^N(\phi) = \sum_i \int_v F_i \frac{dv}{v^0}
\end{equation}
where the $F_i$ are, modulo a multiplication by good symbols, of the form
\begin{enumerate}
\item \label{eq:term1} $$
\left( \frac{\zz}{ t} \right)^{r_1} p^{d_1, k_1} ( \ee{}( \Phi) ) \frac{q^{2d_2+r, k_2}({\Phi})}{t^{d_2+r_2}} Q^{d_3+r_3,k_3}\left(  {\phi} \right)[ \mathbf Y^{\alpha}(\mathbf Y f) ],
$$
where $r_1+r_2+r_3= r$ , $r +d_1+d_2+d_3 \le N-1$,  $|k|+|\alpha|+d_1 \le N-1$, $r + |\alpha|-|\alpha(\ee{})| + |k|-|k(\ee{})| \le N-1$, $k_3(\ee{}) \ge r_3$,
\item \label{eq:term2} $$\left( \frac{\zz}{ t} \right)^{r_1} p^{d_1, k_1} ( \ee{}( \Phi) ) \frac{q^{2d_2+r, k_2}({\Phi})}{t^{d_2+r_2}} Q^{d_3+r_3+1,k_3+1}\left(  {\phi} \right)[ \mathbf Y^{\alpha}(f) ],
$$
where $r_1+r_2+r_3= r$ , $r +d_1+d_2+d_3 \le N-1$,  $|k|+|\alpha|+d_1 \le N-1$, $r + |\alpha|-|\alpha(\ee{})| + |k|-|k(\ee{})| \le N-1$, $k_3(\ee{}) \ge r_3$,
\item \label{eq:term3} $$\left( \frac{\zz}{ t} \right)^{r_1} p^{d_1, k_1} ( \ee{}( \Phi) ) \frac{q^{2d_2+r+1, k_2}({\Phi})}{t^{d_2+r_2}} Q^{d_3+r_3,k_3}\left(  {\phi} \right)[ \mathbf Y^{\alpha}(\mathbf X (f))],
$$
where $r_1+r_2+r_3= r$ , $r +d_1+d_2+d_3 \le N-1$,  $|k|+|\alpha|+d_1 \le N-1$, $r + |\alpha|-|\alpha(\ee{})| + |k|-|k(\ee{})| \le N-1$, $k_3(\ee{}) \ge r_3$, $|\alpha | \le N/2-1$.
\end{enumerate}

\end{lemma}
\begin{remark}
In the proof that follows, we also discuss how each term behaves (in terms of regularity and decay), so that the reader can get some heuristics/intuition, even though we are not performing any estimates here.
\end{remark}
\begin{proof}

We start from \eqref{eq:fob}. Note that, applying the higher order commutator formula at order $N-1$ (i.e. we commute $N-1$ equations \eqref{eq:fob}), we would obtain that the source terms in the wave equation after $N$ commutations is of the form

$$
\left( \frac{\zz}{ t} \right)^{r_1} p^{d_1, k_1} ( \ee{}( \Phi) ) \frac{q^{2d_2+r, k_2}({\Phi})}{t^{d_2+r_2}} Q^{d_3+r_3,k_3}\left(  {\phi} \right)[ \mathbf Y^{\alpha}(g) ]
$$
where $r_1+r_2+r_3= r$ , $r +d_1+d_2+d_3 \le N-1$,  $|k|+|\alpha|+d_1 \le N-1$, $r + |\alpha|-|\alpha(\ee{})| + |k|-|k(\ee{})| \le N-1$, $k_3(\ee{}) \ge r_3$,
with $g$ any of the three terms
$$
\mathbf{Y}f, Z(\phi)f, \Phi \mathbf{X}f.
$$
The first term contributes to the type of terms \eqref{eq:term1} of Lemma \ref{lem:topw}. For the second term, we use that
$$
Z^{N-1} \int_v Z(\phi) f \frac{dv}{v^0}= Z^\alpha (\phi) \int_v Z^\beta(f) \frac{dv}{v^0}
$$
with $|\alpha|+ |\beta| \le N$ and $|\beta| \le N-1$ to obtain the type of terms \eqref{eq:term2}.

For the last term, we could simply distribute the $\mathbf{Y}$ on the products, this would give terms of the form

$$\left( \frac{\zz}{ t} \right)^{r_1} p^{d_1, k_1} ( \ee{}( \Phi) ) \frac{q^{2d_2+r+1, k_2}({\Phi})}{t^{d_2+r_2}} Q^{d_3+r_3,k_3}\left(  {\phi} \right)[ \mathbf Y^{\alpha}(\mathbf X (f))]
$$
where $r_1+r_2+r_3= r$ , $r +d_1+d_2+d_3 \le N-1$,  $|k|+|\alpha|+d_1 \le N-1$, $r + |\alpha|-|\alpha(\ee{})| + |k|-|k(\ee{})| \le N-1$, $k_3(\ee{}) \ge r_3$.
Note that due to the extra occurrence of a $\Phi$ term in the $q$ form, these terms would decay $\rho^{1/2}$ less than the other two types. This would lead to a borderline term (a term leading to a growth in the top order energy). More precisely, it a priori contains term of the form
$$\Phi \mathbf Y^\alpha (\ee {}f),$$ for $\alpha=N-1$. In that case, the number of derivative on $f$ is high, so we would need to estimate it by the $E_N[f]$ energy. Since this energy also has growth (depending on the top order energy for the wave) that would not close. On the other hand, for such a term, the number of vector fields hitting $\Phi$ is low, so we can use the argument of the "not top order" commutation formula to essentially replace $\Phi$ by $\ee{} (\Phi)$ and thus obtain a term that decays as fast as the first two types. Instead, we are trying to get a source term of the form
$$
\mathbf Y^\alpha(\Phi) \mathbf Y^\beta f
$$
where $\alpha$ is large (say $\alpha=N-1$) and $\beta$ is not large (say less than N/2 ). Then, we can use the low order energy norm of $f$ to control $\mathbf Y^\beta f$.

More precisely, start again from formula \eqref{eq:fob} and consider the worst term $\int_v \Phi \mathbf X f \frac{dv}{v^0}$. We commute by $Z$ and follow the previous arguments. We write $G$ for good terms (i.e. terms which are good both from the point of view of decay and from the point of view of regularity and are of the form \eqref{eq:term1} or \eqref{eq:term2}). We have
\begin{eqnarray}
Z \int_v \Phi \mathbf X f \frac{dv}{v^0}&=& \int_v \left( \widehat{Z} + \Phi.\mathbf{X} \right) (\Phi \cdot \mathbf X f ) \frac{dv}{v^0} - \int_v \Phi.\mathbf{X} (\Phi \cdot \mathbf X f )  \frac{dv}{v^0}+ G \nonumber \\
&=& \int_v \mathbf Y (\Phi \cdot \mathbf X f ) \frac{dv}{v^0} - \int_v \Phi.\mathbf{X} (\Phi \cdot \mathbf X f )  \\Phi.\mathbf{X}frac{dv}{v^0} +G \nonumber  \\
&=& \int_v \mathbf Y (\Phi)  \cdot \mathbf X f  \frac{dv}{v^0} + \int_v \mathbf  \Phi \cdot \mathbf Y (\mathbf X f ) \frac{dv}{v^0} - \int_v \Phi.\mathbf{X} (\Phi \cdot \mathbf X f )  \frac{dv}{v^0} +G \nonumber \\
&=& T_1+ T_2+T_3+G.\label{eq:toc1}
\end{eqnarray}
Now the first term, $T_1$ is fine because extra vector fields have not hit the $f$. More preciseley, after an extra $N-2$ commutations, a typical term resulting from $T_1$ is of the form
$$
\mathbf{Y}^{N-1}(\Phi)\mathbf X f
$$
integrated in $v$, and these are fine because they can be estimated using only the low order energy of $f$.

For the second term $T_2$, we have

\begin{eqnarray*}
T_2 = \int_v  \Phi . \mathbf Y (\mathbf X f ) \frac{dv}{v^0} &=& \int_v  \Phi \cdot \mathbf X \mathbf Y (f ) \frac{dv}{v^0} + \int_v \Phi [ \mathbf Y, \mathbf X ] (f ) \frac{dv}{v^0} \\
&=& T_{21}+T_{22}.
\end{eqnarray*}
Now,
\begin{eqnarray}
T_{22}&=&\int_v \Phi \cdot \mathbf X \mathbf Y (f ) \frac{dv}{v^0} \nonumber \\
&=& \int_v \Phi \frac{1}{t}( \mathbf Y + \zz \ee{} -v^0 \partial_{v^i} -\Phi\cdot\mathbf X  ) \mathbf Y (f ) \frac{dv}{v^0} \label{eq:T22} \\
&=& \int_v \Phi \frac{1}{t}( \mathbf Y + \zz \ee{} ) \mathbf Y (f ) \frac{dv}{v^0} + \int_v  v^0 \partial_{v^i}( \Phi) \frac{1}{t} \mathbf Y (f ) \frac{dv}{v^0}- \int_v \frac{\Phi^2}{t} \mathbf X \mathbf Y f \frac{dv}{v^0} \nonumber
\end{eqnarray}
and only the term in the middle still have some potential growth since the others have the same bahaviour as $\int_v s(t,x,v)(1+  \frac{\zz}{t^{1/2}})) \mathbf Y^2 (f) \frac{dv}{v^0}$ for $s$ uniformly bounded, using the bounds on $\Phi$.

For the middle one, we have

\begin{eqnarray}
\int_v \frac{1}{t}v^0 \partial_v \Phi \mathbf Y(f) \frac{dv}{v^0}&=& \int_v \frac{1}{t}  \mathbf Y (\Phi)\mathbf Y(f) \frac{dv}{v^0} + \int_v s(t,x) \ee{} (\Phi) \mathbf Y(f) \frac{dv}{v^0} \nonumber \\
&&\hbox{}-\int_v \frac{1}{t} \Phi \mathbf X(\phi) \mathbf Y(f) \frac{dv}{v^0}, \label{eq:mo}
\end{eqnarray}
where the $s(t,x) \ee{} \Phi$ in the middle term comes from adding and subtracting $\frac{Z(\Phi)}{t}$ appropriately, so that $s$ is just a good symbol. Now all terms are good, both in terms of regularity and in terms of decay. For instance, for the middle one, one can use the improved estimate for $\ee{}(\Phi)$ so that it does not lead to some growth like $\Phi$ and since we have commuted twice and $\ee{}(\Phi)$ is at the level of $\partial^2 Z (\phi)$, we have not exceeded the maximal regularity $N$.

For the commutator term $T_{22}$, one can use Lemma \ref{lem:comze} and check that all resulting terms are of similar nature as the previous ones apart from the term
$$
\int_v \Phi \frac{\partial \phi} {v^0} \partial_{v^i} f \frac{dv}{v^0}.
$$
Here, one can again integrate by parts in $v$ and follow the same arguments as above. With the extra decay coming from the extra $\partial\phi$, this term is actually decaying faster.

Thus, we are left with the term $T_{3}$ in \eqref{eq:toc1}.  The arguments are similar to the above
\begin{eqnarray}
T_3&=&\int_v \Phi\cdot \mathbf{X} (\Phi \cdot \mathbf X f )  \frac{dv}{v^0}=\int_v \Phi\cdot\frac{1}{t}( \mathbf Y + \zz \ee{} -v^0 \partial_{v^i}-\Phi\cdot\mathbf X  )  (\Phi \cdot \mathbf X f )  \frac{dv}{v^0} \nonumber \\
&=&\int_v \Phi\cdot\frac{1}{t} \mathbf Y(\Phi \cdot \mathbf X f ) \frac{dv}{v^0} + \int_v \Phi\cdot\frac{1}{t} \zz \ee{}(\Phi \cdot \mathbf X f )  \frac{dv}{v^0}\nonumber \\
&&\hbox{} - \int_v \Phi\cdot\frac{1}{t} \Phi\cdot\mathbf X    (\Phi \cdot \mathbf X f ) \frac{dv}{v^0} + \int_v v^0 \partial_{v^i}\Phi\cdot\frac{1}{t}(\Phi \cdot \mathbf X f ). \label{eq:T3}
\end{eqnarray}
The first term is now good enough. The last term can be dealt with as in \eqref{eq:mo}. Compared to \eqref{eq:mo}, it has one more power of $\Phi$, leading to a $\rho^{1/2}$ growth, but again this is fine because the number of derivatives hitting $f$ for these terms is low (only one vector field is hitting $f$ while we have commuted twice).

The second and third term in  \eqref{eq:T3} are not yet good enough. The way to deal with them is similar, so we only do it for the third term. First, we distribute the $\mathbf X$ to get

\begin{equation}
\int_v \Phi\cdot\frac{1}{t} \Phi\cdot\mathbf X    (\Phi \cdot \mathbf X f ) \frac{dv}{v^0} = \int_v \Phi\cdot\frac{1}{t} \Phi\cdot\mathbf X    (\Phi)\cdot \mathbf X f  \frac{dv}{v^0}
+ \int_v \Phi^3\cdot \frac{1}{t}\cdot  \mathbf X^2 f  \frac{dv}{v^0} \label{eq:T3f}.
\end{equation}
The last term in \eqref{eq:T3f} is still not good enough, because it grows ( $\frac{ |\Phi^3| } {t} \lesssim \varepsilon^{3/2} \rho^{1/2} $) and a high number of derivatives actually hits $f$. One more application of the same argument (replacing $\mathbf X$ by $\frac{1}{t}( \mathbf Y + \zz \ee{} -v^0 \partial_{v^i}-\Phi\cdot\mathbf X  )$ as for instance in \eqref{eq:T22} is sufficient. For instance, it generates terms such as $\Phi^4 \frac{1}{t^2}\mathbf X^2 f$ or  $\ee{}(\Phi) \frac{\Phi^2}{t} \mathbf X f$ which have no growth, as well as a term with growth but with one lower vector field hitting $f$.

This essentially proves the lemma for $N=2$. The general case follows by a straightforward, induction.
\end{proof}

\section{Norms and Bootstrap assumptions} \label{se:nbs}
We base on the notations introduced in Section \ref{sec38}.
\subsection{Norms}\label{sec:norms}
\subsubsection{$L^1$-norms for $f$}
\begin{definition} Let
\eq{
E_{N}[f]\equiv \sum_{(*)} \int_{H_\rho} \chi\left(\ab{\hat \zz}^{C}\Big| L_{A,B}^\pi f\Big| \right)d\mu_{H_\rho},
}
where

\eq{\label{sdhufh} (\ast)\left\{
\begin{array}{l}
A,B,C \mbox{ multi-indices such that the following conditions hold:}\\
\mbox{If\, \,} \frac{N}{2}< \ab{A}+\ab{B}\leq N, \mbox { then }\, \ab{C}=N+3-|A|\\
\mbox{If\, \,} \ab{A}+\ab{B}\leq \frac{N}{2}, \mbox { then }\, \ab{C}=N-|A|+3+(N+3)\\
\pi \mbox{ any permutation of } \ab A+\ab B \mbox{ elements}
\end{array}\right..
}

\end{definition}

\begin{remark}
The conditions above should be read in the following way. Each generalized translation comes with a weight of order $1$ and there is a global weight of $3$. Then starting from the top order, every lower order is weighted with a weight of order $1$ more. Going to very low orders ($N/2$) this scheme is kept but another global weight of order $N+3$ is added. This is important to be able to have access to pointwise estimates when high order weights occur, low order derivatives of $f$ occur and high order derivatives of $\phi$ and $\Phi$. This becomes clear in the course of the estimates.
\end{remark}
For convenience we define by
\eq{
E^{\,\,\,\circ}_N[f]\equiv \sum_{(*)} \int_{H_\rho} \chi\left(v^0\ab{\hat \zz}^{C}\Big| L_{A,B}^\pi f\Big| \right)d\mu_{H_\rho},
}
where
\eq{\label{sdhufh} (\ast)\left\{
\begin{array}{l}
A,B,C \mbox{ multi-indices such that the following conditions hold:}\\
 \ab{A}+\ab{B}\leq \frac{N}{2}, \mbox { and }\, \ab{C}=N+3-|A|+(N+3)\\
\pi \mbox{ any permutation of } \ab A+\ab B \mbox{ elements}
\end{array}\right..
}
This \emph{low order energy} contains only the low regularity terms of $E_N[f]$.

\subsubsection{$L^2$--norms for the wave}
For the scalar field $\phi$, we use the following standard energy norms.
\begin{definition}Let $\mathscr{E}_N[\phi]$ be defined as
\begin{equation}
\mathscr{E}_N[\phi](\rho)\equiv  \sum_{|\al | \le N, Z^\alpha \in \mathbb{P}^{|\alpha|} }\int_{H_\rho} T[Z ^\al \phi](\partial_t, \nu_\rho) d\mu_{H_\rho},
\end{equation}
where, for any scalar function $\psi$, we denote by $T[\psi]= d \psi \otimes d \psi -\frac{1}{2}\eta( \nabla \psi, \nabla \psi) \eta$ its energy-momentum tensor.
\end{definition}

\subsubsection{$L^1$-energies for products}
We define a variation of the $q$ notation. Let
\eq{\label{aueflkn}
\underline{q}^{A,B,K}_{\geq l}(\Phi)\equiv \sum_{K'\leq K}\sum_{A_i,B_i,\pi_i}\rho^{-K'/2}\prod_{i=1,...,K'} \ab{L^{\pi_i}_{A_i,B_i}(\Phi)}
}
with the additional conditions
\eq{\alg{
\sum_{i=1}^{K'} \ab{A_i}&\leq A,\\
\sum_{i=1}^{K'} \ab{B_i}&\leq B,\\
\ab{A_i}+\ab{B_i}&\geq l,\\
\pi_i \in \mathfrak{S}_{K'}
}}
and the sum over $A_i$, $B_i$ is to be understood over all possible vectors of multi-indices, $(A_i)$, $(B_i)$, and in addition over all suitable permutations $\pi_i$. Note that these conditions imply that
$$
K l \leq A+ B.
$$

We define this auxiliary energy for combination of $\Phi$ and $f$.
\begin{definition}
Let
\eq{\label{eq:combinedenergyPhif}
\mathbf F_N[\Phi,f]\equiv \sum_{(*)} \int_{H_\rho} \chi\left(\ab{\hat \zz}^{C}{\Big|\underline{q}_{\geq 2}^{A,B,K}(\Phi)\Big|}\cdot\Big| L_{U,V}^\pi f\Big| \right)d\mu_{H_\rho},
}
where the sum is taken over the variables $A,B,K$ non-negative integers, and $C,U,V$ multi-indices such that:
\eq{\label{labru} (\ast)\left\{
\begin{array}{l}
1\leq |U|+|V|\leq N\\
A+B+|U|+|V|\leq N\\
\pi \mbox{ any permutation of } \ab U+\ab V \mbox{ elements}\\
\ab{C}= N+3-(|A|+|U|).
\end{array}\right.
}
\end{definition}
Note that the definition of $q_{\geq 2}^{A,B,K}$ implies
\eq{
K\leq \frac{N-1}{2}.
}
Note also that due to the lower bound on $|U|+|V|$, $|A|+|B|\leq N-1$.
In addition we define a low order energy of similar structure with higher weights.
\begin{definition}
Let
\eq{\label{lo-en-Pf}
\mathbf F_N^{\circ}[\Phi,f]\equiv \sum_{(*)} \int_{H_\rho} \chi\left(v^0\ab{\hat \zz}^{C}{\Big|\underline{q}_{\geq 2}^{A,B,K}(\Phi)\Big|}\cdot\Big| L_{U,V}^\pi f\Big| \right)d\mu_{H_\rho},
}
where the sum is taken over the variables $A,B,K$ non-negative integers, and $C,U,V$ multi-indices such that:
\eq{ (\ast)\left\{
\begin{array}{l}
1\leq |U|+|V|\leq N-2\\
A+B+|U|+|V|\leq N-2\\
\pi \mbox{ any permutation of } \ab A+\ab B \mbox{ elements}\\
\ab{C}= N+3-(A+|U|)+(N+3).
\end{array}\right.
}
\end{definition}
These low order energies allow for an absorption of more weights since the wave can be estimates pointwise in all related estimates as we discuss below.

We introduce two additional energies, which are eventually used to exploit the fact that generalized translations acting on $\Phi$ do not require to be normalized by a $\rho^{-1/2}$ factor. We only need this property for one generalized translation in the operators.

\begin{definition}
Based on the definition above we define
\eq{
\mathbf F^{\ee{}}_N[\Phi,f]\equiv \sum_{(*)} \int_{H_\rho} \chi\left(\ab{\hat \zz}^{C}\rho^{1/2}{\Big|\underline{q}_{\geq 2,\ee{}}^{A,B,K}(\Phi)\Big|}\cdot\Big| L_{U,V}^\pi f\Big| \right)d\mu_{H_\rho},
}
where conditions \eqref{labru} hold with the additional restriction
\eq{
B\geq 1
}
and $\underline{q}_{\geq 2,\ee{}}^{A,B,K}(\Phi)$ is defined analogously to \eqref{aueflkn} with the additional condition:
 \eq{
\mbox{ Each summand in \eqref{aueflkn} contains at least one factor $L_{A_i,B_i}\Phi$ with $\ab{B_i}\geq 1$}.
}
In particular, this energy contains at least one generalized translation acting on $\Phi$ while it has one less $\rho^{-1/2}$ factor. The corresponding low order energy $\mathbf F^{\ee{},\circ}_N[\Phi,f]$ is defined analogous with one additional $v^0$ weight.
\end{definition}

\subsection{Set-up and bootstrap assumptions}\label{sec:bootstrap}
Let $(f_0, \phi_0, \phi_1)$ be regular initial data as in \eqref{eq:id}.
By a standard local well-posedness argument, there exists a unique maximal solution $(f, \phi)$ defined on an hyperboloidal time interval of the form $[1, P')$, where $1<P'= \infty$ possibly.
We assume that the initial data is small enough so that
\begin{equation}\label{eq:smallID}
E_{N+3}[f](1), \,\,\mathscr{E}_N(\phi)(1) \le\varepsilon.
\end{equation}
By a standard continuity argument, there exists a largest $1<P \le P'$ such that on $[1, P)$, we have the following estimates
\begin{eqnarray}
E_{N}(f)&\leq& 2\varepsilon\rho^{3C \varepsilon^{1/16}}, \label{eq:bs1} \\
E_{N-1}(f)&\leq& 2\varepsilon \rho^{2C\varepsilon^{1/16}},  \label{eq:bs2}  \\
E^{\,\,\,\circ}_{N}(f)&\leq& 2\varepsilon\rho^{C\varepsilon^{1/16}/2}, \label{eq:bslf} \\
\mathscr{E}_{N}(\phi)&\leq&  2\varepsilon \rho^{C\varepsilon^{1/16}}\label{eq:bs3},\\
\mathscr{E}_{N-1}(\phi)&\leq& 2\varepsilon, \label{eq:bs4}\\
\mathbf F_{N}^{\circ}[\Phi,f](\rho)&\leq&2 \varepsilon \rho^{C\varepsilon^{1/16}/2}, \\
\mathbf F_{N-1}[\Phi,f](\rho)&\leq&2 \varepsilon \rho^{2C\varepsilon^{1/16}}, \\
\mathbf F_{N}[\Phi,f](\rho)&\leq&2 \varepsilon \rho^{5C\varepsilon^{1/16}/2},\label{eq:bs8}\\
\mathbf F_{N}^{\ee{}}[\Phi,f]&\leq&2\varepsilon\rho^{5C\varepsilon^{1/16}/2},\\
\mathbf F_{N-1}^{\ee{}}[\Phi,f]&\leq&2\varepsilon\rho^{2C\varepsilon^{1/16}},\\
\mathbf F_{N}^{\ee{},\circ}[\Phi,f]&\leq&2\varepsilon\rho^{C\varepsilon^{1/16}/2},
\end{eqnarray}
where $C$ is some universal constant depending only on $N$ which is specified below. The choice of $\varepsilon^{1/16}$ is convenient in that context; a value smaller than $1/4$ would work as well in what follows.

The remainder of the paper is devoted to improve the above bootstrap assumptions and prove
\begin{proposition}
Let $N\geq 10$. Then, under the assumptions \eqref{eq:bs1}, \eqref{eq:bs2}, \eqref{eq:bs3} and \eqref{eq:bs4}, the following estimates hold.
\begin{eqnarray}
E_{N}(f)&\leq& \frac{3}{2} \varepsilon\rho^{3C\varepsilon^{1/16}} \label{eq:ibs1}, \\
E_{N-1}(f)&\leq& \frac{3}{2} \varepsilon \rho^{2C\varepsilon^{1/16}}  \label{eq:ibs2}, \\
E_{N}^{\,\,\,\circ}(f)&\leq& \frac{3}{2} \varepsilon  \rho^{C\varepsilon^{1/16}/2},\label{eq:ibs2-1}\\
\mathscr{E}_{N}(\phi)&\leq&  \frac{3}{2}\varepsilon\rho^{ C \varepsilon^{1/16}},\label{eq:ibs3}\\
\mathscr{E}_{N-1}(\phi)&\leq& \frac{3}{2} \varepsilon, \\
\mathbf F_{N}^{\circ}[\Phi,f](\rho)&\leq& \frac{3}{2} \varepsilon\rho^{C\varepsilon^{1/16}/2},\\
\mathbf F_{N-1}[\Phi,f](\rho)&\leq& \frac{3}{2}\varepsilon\rho^{2C\varepsilon^{1/16}},\\
\mathbf F_{N}[\Phi,f](\rho)&\leq& \frac{3}{2}\varepsilon \rho^{5C\varepsilon^{1/16}/2}, \label{eq:ibs4}\\
\mathbf F_{N}^{\ee{}}[\Phi,f]&\leq&\frac{3}{2}\varepsilon\rho^{5C\varepsilon^{1/16}/2},\\
\mathbf F_{N-1}^{\ee{}}[\Phi,f]&\leq&\frac{3}{2}\varepsilon\rho^{2C\varepsilon^{1/16}},\\
\mathbf F_{N}^{\ee{},\circ}[\Phi,f]&\leq&\frac{3}{2}\varepsilon\rho^{C\varepsilon^{1/16}/2}.
\end{eqnarray}
\end{proposition}
It follows that $P=P'=+\infty$, which proves the main Theorem.

\section{Estimates for the $\Phi$ coefficients} \label{sec:Phiest}
This section is devoted to establishing estimates for the $\Phi$ coefficients. First, we recall Duhamel's formula for the inhomogeneous transport equation.
\begin{lemma}\label{lem:repform} Let $U_{\rho'} (\rho, y, v)$ be the solution to the Cauchy problem
$$
T_{\phi} U_{\rho'}  =  0
$$
with initial data prescribed at $\rho:=\rho'$ by
$$ U_{\rho'} (\rho', y, v) = \dfrac{h}{v^\rho}(\rho',y,v).$$
where
$$
v^\rho = \dfrac{v^0 t -v_ix^i}{\rho}
$$
and $h$ is some (sufficiently regular) source term.

Then, the unique solution to the problem
$$
T_{\phi}  \left( \Psi\right)  =  h \text{ with } \Psi|_{H_1} = 0
$$
can be written as:
$$
\Psi(\rho, y, v) = \int_1^\rho U_{\rho'} (\rho, y, v)  d \rho'.
$$
\end{lemma}

We need the following estimates on $v^\rho$.
\begin{lemma}\label{lem:vrhoest} $v^\rho = \dfrac{tv^0 - v_ix^i}{\rho}$ satisfies the following estimates in $J^+(H_1)$
$$
v^\rho \geq 1, \quad v^\rho \geq \dfrac{u}{2\rho}v^0, \quad v^\rho \geq \dfrac{t}{2\rho v^0}.
$$
\end{lemma}
\begin{remark}\label{sdf;iuhef} In practice, in this section, we only use the inequality
$$
v^\rho \geq \dfrac{t}{2\rho v^0}.
$$
\end{remark}

\begin{proof} The first two inequalities can be found in [Lemma 2.11 and Remark 2.12] of  \cite{fjs:vfm}. The proof of the last inequality is as follows:
\begin{eqnarray*}
 \rho v^\rho &\geq & (v^0 t -|x||v|) \dfrac{(v^0 t +|x||v|)}{(v^0 t +|x||v|)}\\
 &\geq & ((v^0)^2 t^2 -|x|^2|v|^2) \dfrac{1}{2 v^0 t } \text{ since } v^0 \geq |v| \text{ and } t\geq |x| \\
  &\geq & ((1+|v|^2) t^2 -|x|^2|v|^2) \dfrac{1}{2 v^0 t }\\
    &\geq & (t^2 + |v|^2\rho^2) \dfrac{1}{2 v^0 t }\\
        &\geq &  \dfrac{t}{2 v^0  }.
\end{eqnarray*}
\end{proof}

The last preparatory lemma deals with conservation of the $L^\infty$-norm for solutions of the transport equation:
\begin{lemma}\label{lem:consinfty} Let $U_{\rho'}$ be the solution of the Cauchy problem
$$
\Tp U_{\rho'} =0 \text{ with }U_{\rho'} (\rho', y, v) = \dfrac{h(\rho', y, v)}{v^\rho} \text{ on }H_{\rho'},
$$
then, if
$$
\sup_{H_{\rho'}}\left(\dfrac{h(\rho', y, v)}{v^\rho}\right)<\infty,
$$
$U_{\rho'}$ satisfies, in the future of the hyperboloid $H_{\rho'}$:
$$
|U_{\rho'}| \leq \sup_{H_{\rho'}}\left(\dfrac{h(\rho', y, v)}{v^\rho}\right).
$$
\end{lemma}

We can finally state the pointwise estimates for $\Phi$ based on the equation that $\Phi$ satisfies, which we recall here
 $$
\Tp \left(\Phi_a^i \right) = - \frac{t}{v^0}\left( \sum_{ | \alpha | \le 1} b_{a,\alpha}^{\beta i} \partial_{x^\beta} Z^\alpha(\phi) +\dfrac{Z_a \phi }{(1+u)}\cdot  a^i \right).
$$

We thus define
$$
h  = - \frac{t}{v^0}\left( \sum_{ | \alpha | \le 1} b_{a,\alpha}^{\beta i} \partial_{x^\beta} Z^\alpha(\phi) +\dfrac{Z_a \phi }{(1+u)}\cdot  a^i \right).
$$
From the bootstrap assumption \eqref{eq:bs4}, Proposition \eqref{prop:hypwave} and Lemma \eqref{lem:hypwave1}, we have the following pointwise estimates on the source term for $\Phi$

$$
|h| \lesssim \dfrac{\sqrt{\varepsilon}}{(1+u)^{1/2} v^0}.
$$
The pointwise estimates for $\Phi$ are then as follows.

 \begin{lemma}\label{es:inht} Assume that $h$ satisfies in $J^+(H_1)$
$$
|h| \lesssim \dfrac{\sqrt{\varepsilon}}{ (1+u)^{1/2}v^0}.
$$
 Then the solution to the inhomogeneous Cauchy problem
 $$
 T_{\phi}  \left( \Psi\right)  =  h \text{ with } \Psi|_{H_1} = 0
 $$
 satisfies in $J^+(H_1)$
 $$
 \vert \Psi \vert \lesssim   \sqrt{\varepsilon\rho}.
 $$
 \end{lemma}
 \begin{proof} The proof of this fact is a direct consequence of the two previous Lemmata \ref{lem:repform}, \ref{lem:vrhoest} and \ref{lem:consinfty}: the solution $\Psi$ can be represented as
 $$
 \Psi(\rho, y, v) = \int_1^\rho U_{\rho'} (\rho, y, v)  d \rho',
 $$
 where $U_{\rho'} (\rho, y, v) $ satisfies the Cauchy problem with datum on the hyperboloid $H_{\rho'}$:
 $$
 T_{\phi} U_{\rho'}  =  0 \text{ with } U_{\rho'} (\rho', y, v) = \dfrac{h(\rho', v, y)}{v^\rho}.
 $$

By Lemma \ref{lem:consinfty}, one needs to compute the supremum over a given hyperboloid $H_{\rho'}$ of $h/v^\rho$.

We need then to use the only compatible inequality of Lemma \ref{lem:vrhoest} to perform the estimates
$$
\dfrac{|h|}{v^{\rho}} \lesssim  \dfrac{\sqrt{\varepsilon}\rho}{\sqrt{1+u} t }  \lesssim  \dfrac{\sqrt{\varepsilon}}{\rho^{\frac12}}.
$$
The estimates for $\Psi$ can then be written, using Lemmata \ref{lem:repform} and \ref{lem:consinfty}:
\begin{eqnarray*}
 |\Psi| &\lesssim & \int_1^\rho \sup_{H_{\rho'}}\left(\dfrac{h}{v^\rho}\right) d \rho'\\
 &\lesssim & \int_1^\rho  \dfrac{\sqrt{\varepsilon}}{\rho^{\frac12}} d \rho'\\
  &\lesssim &\sqrt{\varepsilon \rho}.
\end{eqnarray*}
 \end{proof}

 \begin{lemma}\label{lem-foe} The terms $\YY \Phi$ and $\ee{} \Phi$ satisfy, in the future unit hyperboloid $J^+(H_{1})$:
   $$
 \vert \YY \Phi \vert \lesssim \sqrt{\rho \varepsilon} \text{ and } \vert \ee {} \Phi \vert \lesssim \sqrt{\varepsilon}\ln(\rho).
   $$
 \end{lemma}
 \begin{remark} In principle, higher order derivatives of $\Phi$ could be estimated pointwise, since $\Phi$ satisfies a inhomogeneous equation whose source terms can be estimated pointwise. Nonetheless, terms of the form $\Phi \zz$ pile up in the source term, preventing similar estimates to be established.
 \end{remark}
 \begin{proof}
First, one easily check that the terms coming from the commutators $[T_\phi, \YY]$ and $[T_\phi, \ee ]$ are compatible with the statement of the lemma.
Then, the rest of the proof essentially relies on the following formula (see Lemma \ref{lem:firstorderYp} for the first one; up to lower order terms):
   \begin{align}
     \ee {}\left( \dfrac{t}{v^0} p\left(\partial Z\phi, \frac{Z\phi}{1+u}\right)\right) =&
      \dfrac{t}{v^0} p\left(\partial^2 Z\phi, \frac{\partial Z\phi}{1+u}, \frac{Z\phi}{(1+u)^2}\right)\\
     \YY {}\left( \dfrac{t}{v^0} p\left(\partial Z\phi, \frac{Z\phi}{1+u}\right)\right) =& \dfrac{t}{v^0} \left(\left(1+\frac{\Phi}{t}\right)p(\partial Z^2\phi,u^{-1}Z^2\phi)
     +\frac{\Phi}{t}\zz p\left(\partial^2Z\phi,u^{-1}(\tfrac{\partial}{u})Z\phi\right)\right)
   \end{align}
 Using the bootstrap assumption \eqref{eq:bs4}, one immediately obtains for each of the source terms
 \begin{eqnarray}
 \left \vert \dfrac{1}{v^\rho}  \ee {}\left( \dfrac{t}{v^0} p\left(\partial Z\phi, \frac{Z\phi}{1+u}\right)\right) \right \vert &\lesssim &  \rho \cdot \dfrac{\sqrt{\varepsilon}}{(1+u)^{\frac12}\rho \sqrt{t}}  \\
 &\lesssim & \dfrac{\sqrt{\varepsilon}}{(1+u)^{\frac12}\rho}\\
   \left \vert \dfrac{1}{v^\rho}\YY {}\left( \dfrac{t}{v^0} p\left(\partial Z\phi, \frac{Z\phi}{1+u}\right)\right) \right \vert &\lesssim & \rho \cdot \dfrac{\sqrt{\varepsilon}}{\rho \sqrt{t}} + \rho \cdot \dfrac{\sqrt{\varepsilon\rho}}{t} \cdot t \cdot \dfrac{\sqrt{\varepsilon}}{(1+u)^{1/2+ 1} t}\\
  & \lesssim & \dfrac{\sqrt{\varepsilon}}{\sqrt{\rho}}
 \end{eqnarray}
   where we have estimated $v^0 v^\rho$ as in Lemma \ref{lem:vrhoest}
   $$
 v^0 v^\rho \geq \dfrac{t}{2 \rho},
   $$
   as well as, accordingly to Lemma \ref{es:inht}
   $$
 \vert \Phi \vert \lesssim \sqrt{\varepsilon \rho}.
   $$
   The rest of the proof runs similarly as the proof of Lemma \ref{es:inht}.
 \end{proof}

\section{$L^1$-estimates for the distribution function} \label{sec:L1estimates}
We prove \eqref{eq:ibs1}, \eqref{eq:ibs2} and  \eqref{eq:ibs2-1} in the following proposition.

\begin{proposition}\label{prop:L1est}
Assume that the bootstrap assumptions hold on $[1,P)$. Then
\eq{\alg{
E_N[f](\rho)-E_N[f](1)&\lesssim \varepsilon^{5/4}\rho^{5C\varepsilon^{1/16}/2}\\
E_{N-1}[f](\rho)-E_{N-1}[f](1)&\lesssim \varepsilon^{5/4}\rho^{2C\varepsilon^{1/16}}\\
E^\circ_N[f](\rho)- E^{\circ}_{N}[f](1)&\lesssim\varepsilon^{5/4}\rho^{C\varepsilon^{1/16}/2}
}}
hold for $\rho\in[1,P)$.
\end{proposition}

\begin{proof}
From Lemma \ref{lem:macl}, we need to estimate all terms of the form

\eq{
\iint \Tp\left(\ab{\hat \zz}^{C}\Big| L_{A,B}^\pi f\Big|\right) dv d\mu_{H_\rho},
}
where $A$, $B$, $C$ take the required values. We begin by considering the case, when $\Tp$ hits the weight-term $\ab{\hat \zz}^{C}$. Using \eqref{zz-comm}, these terms are of the form
\eq{\label{sdflj}
\iint \left(\T(\phi) \cdot \ab{\hat \zz}^{C}+\frac{Z\phi}{v^0\sqrt{\rho}} \ab{\hat \zz}^{C'} \right)\Big| L_{A,B}^\pi f\Big| dv d\mu_{H_\rho}+\hdots,
}
where $\ab{C'}=\ab{C}-1$ and the symbol $\hdots$ denotes similar or lower order terms or those with a negative sign, which arise from $\sqrt{\rho}^{-1}$ in $\hat\zz$.

Estimating the coefficients in $\partial \phi
$ pointwise and using the second estimate in Lemma \ref{lem:vrhoest}, these terms are bounded by
\eq{
\lesssim \rho^{-1}\sqrt{\varepsilon} E_{N}(f).
}
The same estimate holds, if we consider the energy of order $N-1$, with $E_N$ replaced by $E_{N-1}$ or lower orders. We turn now to those terms which arise when $\Tp$ hits $|L^\pi_{A,B}f|$. These take the form
\eq{\label{l1-pr-1}
\iint \ab{\hat \zz}^{C} \,\, (**) \,\,d\mu_{H_\rho},
}
where $(**)$ stands for any of the different types of the forms \eqref{comm-first-type} -- \eqref{comm-eighth-type} with the respective conditions on the indices in these expressions such that $\ab A+\ab B=N$, $N-1$ or lower orders, depending on which energy we estimate.

\textbf{High orders.}  We consider in this part the case
\eq{
\ab A+\ab B>N/2.
}
We discuss term \eqref{comm-1-f} explicitly in detail. We recall it below.
\eq{\label{comm-1-f-loc}\alg{
\sum_{(\ast)}\Bigg\{&\left(\frac{q^{\alpha,S,\delta+\beta+\mu+2\gamma}(\Phi)}{t^{(\delta+\beta+2\sigma)/2}}\right)\left(\frac{\zz^{\delta+\nu}}{\sqrt{t}^{\delta}(v^0)^{1-\sigma}}\right)\cdot t^{\kappa}\\
&\cdot p\left(\partial^{\delta+I+1+\gamma-\sigma+\kappa+\lambda(1-\mu)}Z^{\beta-(\gamma-\sigma+\kappa+\lambda(1-\mu))}\phi\right)
\cdot L_{U',V' }(f)
\Bigg\}
}}
where the sum is taken over the variables $U', V'$ multi-indices and $\alpha, S, \delta, \beta, \mu, \gamma, \sigma, \nu, \kappa, \lambda$ such that
\eq{\label{comm-T-cond-loc}(\ast)\left\{\alg{
& |U'| =  U + \lambda, |V'| =  V +1-\lambda
&V+S+I\leq \ab{B}\\
&U+\al+\be+\delta+\nu\leq \ab{A}\\
&\al+S+U+V+\delta+I+\beta\leq N\\
&1\leq\delta+I+\beta\\
&\mu+\nu\leq 1\\
&\mu=\gamma=1\,\Rightarrow \sigma=1\\
&\sigma\leq\gamma\leq1\\
&\kappa=1 \Rightarrow \mu=\gamma=\nu=\lambda=\sigma=0, \kappa\leq 1\\
&\lambda=1 \Rightarrow \gamma=\nu=\sigma=0, \lambda\leq 1\\
&\lambda=\mu=1\Rightarrow U+\lambda+\alpha+\beta+\delta\leq\ab{A}
}\right.}
Estimating the integral \eqref{l1-pr-1} we distinguish between different cases. The first case we consider covers all terms, where the polynomial in $\partial\phi$ and its derivatives can be estimated pointwise. This is the case when
\eq{\label{L1-f-high}
\delta+I+\beta<N-1,
}
since we need $2$ commutations to apply the Klainerman-Sobolev inequality \eqref{eq:kswh}.

We then have access to pointwise estimates on the polynomial in derivatives of $\phi$, which yield
\begin{align}
\left|p\left(\partial^{\delta+I+1+\gamma-\sigma+\kappa+\lambda(1-\mu)}Z^{\beta-(\gamma-\sigma+\kappa+\lambda(1-\mu))}\phi\right)\right|&\\ \leq&\frac{\sqrt{\varepsilon}\rho^{C/2\varepsilon^{1/16}}}{(1+u)^{1/2+\gamma-\sigma+\kappa+\lambda(1-\mu)}(1+t)},\nonumber
\end{align}
where we used bootstrap assumption \eqref{eq:bs3}. Note that the factor $\rho^{C\varepsilon^{1/16}/2}$ only appears when estimating the top order norm. Using these preliminary estimates the total term \eqref{l1-pr-1} which we wish to estimate under the current assumptions reduces to
\begin{align}\label{L1-est-ter}
\iint \ab{\hat \zz}^{C}\Big| \frac{q^{\al,S,\delta+\beta+\mu+2\gamma}(\Phi)}{t^{(\delta+\beta+2\sigma)/2}} \Big |\left|\frac{\zz^{\delta+\nu}}{\sqrt{t}^{\delta}(v^0)^{1-\sigma}}\right|\cdot t^{\kappa}&\\
\cdot\frac{\sqrt{\varepsilon}\rho^{C/2\varepsilon^{1/16}}}{(1+u)^{1/2+\gamma-\sigma+\kappa+\lambda(1-\mu)}(1+t)}&|L_{U',V' }(f)|d\mu_{H_\rho}\nonumber.
\end{align}
We then intend to estimate this by the combined energy of type $\mathbf F_N[\Phi,f]$. We distinguish now between the different cases for the switch indices (recall the conditions on $(\mu, \nu, \gamma, \sigma, \kappa, \lambda)$ \eqref{comm-T-cond-loc}). Before addressing the individual cases, we discuss the issue of the absorption of the $\hat{\zz}$-weight to the power $C$ into the energy of $f$. Due to the condition \eqref{L1-f-high} we always estimate $f$ in energy in this part. We remark that in the special case $\nu=1$, which we discuss individually below, we have to divide and multiply by $\sqrt \rho$ to obtain a rescaled weight $\hat\zz$ as encoded in the energy. Also, we see below another exceptional case ($\lambda=1$, $\mu=0$), when one weight has to be estimated pointwise. For now, we only discuss when the number of weights that appear can be absorbed. The term we intend to estimate in energy then reads
\eq{\label{absorb-loc}
\ab{\hat \zz}^{D}\frac{\ab{ q^{\al,S,\delta+\beta+\mu+2\gamma}(\Phi)}}{\rho^{(\delta+\beta+\mu+2\gamma)/2}}\ab{L_{U',V'}f},
}
where $\ab{D}=\ab{C}+\delta+\nu$. Note that $\ab{C}=N+3-\ab A$ as we are in the case $\ab A+\ab B>N/2$.
The term \eqref{absorb-loc} can be estimated in energy precisely if
\eq{
N+3-\ab A+\delta+\nu \leq N+3-(U+\lambda+\alpha)
}
holds (cf.~definition of $\mathbf F_N$ in \eqref{eq:combinedenergyPhif}). In the case $\lambda=0$ this follows immediately from the second condition in \eqref{comm-T-cond-loc} in the case $\lambda=\mu=1$ this follows from the last line in \eqref{comm-T-cond-loc}. In the complementary case $\lambda=1, \mu=0$ (which implies $\nu=0$) only $D-1$ weights can be absorbed into the energy and one weight $\hat\zz$ has to be estimated pointwise.

 We have so far shown that we can absorb the weights into the energy $\mathbf F_N$ unless $\lambda=1, \mu=0$, when one weight is estimated pointwise. It remains to be shown that in all possible configurations of switch indices, the decay of the low regularity terms is sufficiently strong to conclude the estimate. We distinguish the different cases for the switch indices in the following. We rewrite the term to estimate, \eqref{L1-est-ter}, as
\eq{\alg{
\iint  \left|\frac{1}{(v^0)^{1-\sigma}}\right|\cdot t^{\kappa-\sigma}&\frac{\sqrt{\varepsilon}\rho^{C\varepsilon^{1/16}/2}}{(1+u)^{1/2+\gamma-\sigma+\kappa+\lambda(1-\mu)}(1+t)}\rho^{(\nu+\mu+2\gamma)/2}\\
&\ab{\hat \zz}^{D}\frac{\ab{ q^{\al,S,\delta+\beta+\mu+2\gamma}(\Phi)}}{\rho^{(\delta+\beta+\mu+2\gamma)/2}}|L_{U',V' }(f)|dv d\mu_{H_\rho},
}}
where $|D|=|C|+\delta+\nu$ and the $\rho^{\nu/2}$ term appears since the additional weight in the case $\nu=1$ does not have a $\rho^{-1/2}$ factor. We do three more modifications to obtain the final form of the previous term. First, we estimate any zero or first order derivative of $\Phi$ with the corresponding $\rho^{-1/2}$ factor pointwise using Lemma \ref{lem-foe}, which allows to replace the $q$ term by the analogous $\underline q$ absorbing the $\rho$ terms in the denominator. As a second modification we introduce a $(t/\rho)^{1-\sigma}$ factor which in combination with $(v^0)^{-1+\sigma}$ can be estimated by $v^\rho$ (cf.~Remark \ref{sdf;iuhef}). For the third modification we estimate one factor $\hat\zz$ pointwise when $\lambda(1-\mu)=1$ yielding $(t/\rho)^{\lambda(1-\mu)}$. In total we arrive at
\eq{\alg{
\iint&  \underbrace{\frac{\rho^{1-\sigma-\lambda(1-\mu)/2+(\nu+\mu+2\gamma)/2}}{t^{1-\kappa-\lambda(1-\mu)}}\frac{\sqrt{\varepsilon}\rho^{C\varepsilon^{1/16}/2}}{(1+u)^{1/2+\gamma-\sigma+\kappa+\lambda(1-\mu)}v}}_{(\star)}\\
&\qquad\qquad\qquad\ab{\hat \zz}^{D}{\ab{ \underline{q}_{\geq 2}^{\al,S,\delta+\beta+\mu+2\gamma}(\Phi)}}|L_{U+\lambda,V+1-\lambda }(f)|(v^{\rho})^{1-\sigma}dv d\mu_{H_\rho},
}}
where $D=C+\delta+\nu-\lambda(1-\mu)$. A straightforward evaluation of the conditions \eqref{comm-T-cond-loc} yields that in each admissible combination of switch indices, which are $\kappa=1;\lambda=\mu=1;(\lambda=1,\mu=0); (\gamma=1,\nu\leq1); \mu=\gamma=\sigma=1$, the estimate
\eq{
\ab{(\star)}\lesssim \sqrt{\varepsilon}\rho^{C\varepsilon^{1/16}/2-1}
}
holds. In total, we can estimate the integral by
\eq{
\sqrt{\varepsilon}\rho^{-1+C/2\varepsilon^{1/16}} \mathbf F_{\ab U+\ab V+\alpha+S+1}[\Phi,f](\rho),
}
where the factor $\rho^{C\varepsilon^{1/16}/2}$ only occurs when the highest order energy is estimated and in that case $\ab U+\ab V+\alpha+S<N$ - so not both terms can be of highest order. Invoking the bootstrap assumptions in both cases the term in top order can be estimated by
\eq{
{\varepsilon}^{3/2}\rho^{-1+5C/2\varepsilon^{1/16}}.
}
If the order of the energy is at most $N-1$, then the term can be estimated by
\eq{
\varepsilon^{3/2}\rho^{-1+2C\varepsilon^{1/16}}.
}
Integrating both cases yields the estimates claimed in the proposition. This finishes the case \eqref{L1-f-high} and we consider in the remainder the complementary case

\eq{\label{eq:L1-f-low}
\delta+I+\beta\geq N-1.
}
By conditions \eqref{comm-T-cond-loc} this implies
\eq{
U+V+\alpha+S\leq 1.
}
This has the immediate consequence that only up to first order derivatives of $\Phi$ appear, which can be estimated using Lemma \ref{lem-foe}. Furthermore, only up to second order derivatives of $f$ appear, which implies that we require $N\geq 5$ to estimate the $v$-integral of those terms pointwise. However, we wish to estimate those terms by the low order energy ($|A|+|B|\leq N/2$) to avoid the growth of the top order norm of $f$. Therefore we choose $N\geq 10$. The pointwise estimates on $\Phi$ (cf.~Lemma \ref{lem-foe}) and its first order derivatives in combination with the fact $\alpha+S\leq1$ imply

\eq{
\frac{\ab{ q^{\al,S,\delta+\beta+\mu+2\gamma}(\Phi)}}{\rho^{(\delta+\beta+\mu+2\gamma)/2}}\lesssim 1,
}
which leaves from \eqref{comm-1-f-loc} the term to estimate in this case in the form

\eq{\label{L1-est-ter-2}\alg{
\iint & {\rho^{(\mu+2\gamma+\nu)/2}} t^{\kappa-\sigma}\left|p\left(\partial^{\delta+I+1+\gamma-\sigma+\kappa+\lambda(1-\mu)}Z^{\beta+1-(\gamma-\sigma+\kappa+\lambda(1-\mu))}\phi\right)\right|\\
&\qquad\qquad\qquad\qquad\underbrace{(v^0)^{\sigma-1}\ab{\hat \zz}^{D}|L_{U',V' }(f)|}_{(\star)}dvd\mu_{H_\rho},
}}
where $\ab D=N+3-\ab{A}+\delta+\nu$. Before analysing the different cases for the switch indices we clarify that we can estimate the momentum-integral over the term $(\star)$ pointwise by the Klainerman-Sobolev embedding in terms of the \emph{low order energy} $E^{\,\,\,\circ}_N[f]$. Applying the inequality \eqref{ineq:ksmsv} we increase the weight by three and increase the number of derivatives acting on $f$ also by three. If these three derivatives are all not generalized translations then the number of weights that can be absorbed by this term is maximally decreased, so we assume this case, which is
\eq{
\ab{\hat \zz}^{D'}|L_{U'',V''}(f)|,
}
where $U'', V''$ are multi-indices such that $|U''| =  U+\lambda+3, |V''|= V+1-\lambda$ and with $\ab{D'}=\ab{D}+3=N+3-\ab{A}+\delta+\nu+3$. Estimating the corresponding integral by the norm $E^{\,\,\,\circ}_N[f]$ we can absorb a number of weights according to \eqref{sdhufh} which is $N+3-(U+\lambda+3)+(N+3)$.
Indeed, then the formula
\eq{\alg{
\ab {D'}&=N+3-\ab A+\delta+\nu+3\leq N+3-\ab A+ \ab A+3\\
&\leq N+3+3\\
&\leq N+3 - (U+\lambda+3)+N+3 +\underbrace{(U+\lambda+3-N)}_{<0\mbox{, by } \ab{U},\lambda\leq 1}
}}
holds. In particular, no pointwise estimate for a weight is necessary in this case. As a conclusion of the above we can estimate
\eq{
\int \ab{\hat \zz}^{D}(v^0)^{\sigma-1}|L_{U',V' }(f)| dv\lesssim \frac1{t^3}\left(1+\varepsilon^{1/2}\rho^{C\varepsilon^{1/16}}\right) \cdot E^{\,\,\,\circ}_N[f].
}

It is important to note that we obtain pointwise estimates in terms of the low order energy $E^{\,\,\,\circ}_{N}(f)$. Note that this energy has an additional $v^0$ to compensate in case $\sigma=1$. The latter is crucial to be able to improve the bootstrap assumption on $E_N[f]$ and $E_{N-1}[f]$, which only works since estimating by $E^{\,\,\,\circ}_N[f]$ is possible here. For the low order energy itself, this problem does not occur as in the corresponding estimates (see below) only the wave terms are estimated pointwise.  Using the considerations above, replacing translations acting on $\phi$ by vector fields $Z$ gaining $u^{-1}$ terms and using the bootstrap assumption on $E_N^{\circ}[f]$ the term \eqref{L1-est-ter-2} reduces to

\eq{\label{loc-term-est-2}
{\varepsilon}\sqrt{\rho}^{\mu+2\gamma+\nu}\rho^{3C\varepsilon^{1/16}/2}\int \frac{t^{\kappa-\sigma-3}}{(1+u)^{\gamma-\sigma+\kappa+\lambda(1-\mu)}}  \left|p\left(\partial^{\delta+I+1}Z^{\beta}\phi\right)\right|d\mu_{H_\rho}.
}

Using H\"older's inequality and the definition of the volume forms this is bounded by
\begin{align}\label{erafguh}
&{\varepsilon}\sqrt{\rho}^{\mu+2\gamma+\nu}\rho^{3C\varepsilon^{1/16}/2}\\
\times &\left(\int \left|p\left(\partial^{\delta+I+1}Z^{\beta}\phi\right)\right|^2\frac\rho td\mu_{H_\rho}\right)^{1/2}
\cdot \left(\underbrace{\int\frac{t^{2(\kappa-\sigma-3)}r^2}{(1+u)^{2(\gamma-\sigma+\kappa+\lambda(1-\mu))}}dr}_{(\star)}\right)^{1/2} \nonumber
.
\end{align}
The first bracket is the energy of the wave and the integral in the second bracket can be evaluated using Lemma \ref{lem-int-est} to
\eq{\label{sdfikl}
\ab{(\star)}\leq \frac{\rho^{-3-\boldsymbol{\delta}+2(\kappa-\sigma)}}{3-2(\kappa-\sigma)-\boldsymbol{\delta}} \mbox{ such that } \boldsymbol{\delta}<2(\gamma-\sigma+\kappa+\lambda(1-\mu)), \boldsymbol{\delta}+2(\kappa-\sigma)<3.
}
In view of this, \eqref{erafguh} reduces to
\eq{
{\varepsilon}\sqrt{\rho}^{\mu+2\gamma+\nu}\rho^{3C\varepsilon^{1/16}/2} \mathscr E_{\delta+I+\beta}[\phi]^{1/2}\frac{\sqrt{\rho}^{-3-\boldsymbol{\delta}+2(\kappa-\sigma)}}{\sqrt{3-2(\kappa-\sigma)-\boldsymbol{\delta}}}.
}
Distinguishing between the different cases for the switch indices
in all cases except $\kappa=1$ and $(\gamma= 1,\nu\leq1, \sigma=0)$ this can be estimated by
\eq{
\lesssim{\varepsilon}\rho^{-1}\rho^{3C\varepsilon^{1/16}/2} \mathscr E_{\delta+I+\beta}[\phi]^{1/2},
}
choosing $\boldsymbol{\delta}=0$. In the case $(\gamma=\nu=1, \sigma=0)$ we choose $\boldsymbol{\delta}=2-C\varepsilon^{1/16}/2$ yielding
\eq{
\lesssim{\varepsilon}\rho^{-1}\rho^{7C\varepsilon^{1/16}/4} \mathscr E_{\delta+I+\beta}[\phi]^{1/2}.
}
Finally, the in the case $\kappa=1$ we choose $\boldsymbol{\delta}=1-C\varepsilon^{1/16}/2$ yielding
\eq{
\lesssim\varepsilon^{1-1/32}\rho^{-1}\rho^{7C\varepsilon^{1/16}/4} \mathscr E_{\delta+I+\beta}[\phi]^{1/2}.
}
The last term is the worst in terms of decay and smallness so using the bootstrap assumption on the wave in total we can bound \eqref{erafguh} by

\eq{
\lesssim\varepsilon^{3/2-1/32}\rho^{-1}\rho^{7C\varepsilon^{1/16}/4} \rho^{C\varepsilon^{1/16}/2},
}
where the last factor is only present in the highest order case. So in the highest order case we can estimate the previous term by
\eq{
\lesssim\varepsilon^{3/2-1/32}\rho^{-1}\rho^{9C\varepsilon^{1/16}/4}
}
and in the lower order case by
\eq{
\lesssim\varepsilon^{3/2-1/32}\rho^{-1}\rho^{7C\varepsilon^{1/16}/4}.
}
Integrating in both cases yields the claimed estimates in the proposition. This finishes the discussion of the case \eqref{eq:L1-f-low} and the proof for the high order energy.
As a summary, using the energy estimates of Lemma \ref{lem:macl}, one
obtains the following
\begin{align}
E_{N}[f](\rho) - E_{N}[f](1) \lesssim & \int_{1}^\rho s^{-1}
\mathbf{F}_{N-1}[\Phi, f](s)ds + \sqrt{\varepsilon}\int_{1}^\rho s^{-1 +
C\varepsilon^{1/16}/2} \mathbf{F}_{N}[\Phi, f](s)ds\\
&+  \varepsilon \int_{1}^{\rho}
\rho^{-1 +  9C \varepsilon^{1/16}}\mathscr{E}^{1/2}_N[\phi](s)\\
\lesssim & \varepsilon^{5/4}\rho^{5C\varepsilon^{1/16}/2},
\end{align}
when inserting the bootstrap assumptions.

As a summary, using the energy estimates of Lemma \ref{lem:macl}, one obtains the following
\begin{align}
E_{N}[f](\rho) - E_{N}[f](1) \lesssim & \int_{1}^\rho s^{-1} \mathbf{F}_{N-1}[\Phi, f](s)ds + \sqrt{\varepsilon}\int_{1}^\rho s^{-1 +  C\varepsilon^{1/16}/2} \mathbf{F}_{N}[\Phi, f](s)ds\\
&+  \varepsilon \int_{1}^{\rho}
 \rho^{-1 +  9C \varepsilon^{1/16}}\mathscr{E}^{1/2}_N[\phi](s)\\
 \lesssim & \varepsilon^{5/4}\rho^{5C\varepsilon^{1/16}/2},
\end{align}
when inserting the bootstrap assumptions.

\textbf{Low orders} ($\ab A+\ab B\leq N/2$) \\
For low order terms, the weight in $\hat\zz$ is high, but this does not change the way to estimate the terms according to \eqref{sdflj}. In addition, as the order of derivatives is low, we always estimate the terms containing the wave $\phi$ pointwise, where we proceed analogously to the case of \eqref{L1-f-high} which finishes the proof.
\end{proof}

\section{High order estimates for products}\label{sec:Phif}

In this section we derive energy estimates for the norms $\mathbf F_{N}[\Phi,f]$ and $\mathbf F^{\ee{}}_{N}[\Phi,f]$. Recall that we have the equation for $\Phi$ in the form
\eq{
\T_{\phi}\Phi =h,
}
where
\eq{
h= \frac{t}{v^0}p(\partial Z\phi)+C,
}
where $C$ is of lower regularity and similar or better decay than the explicitly written term. To derive the energy estimates for $\mathbf F_N$ we compute the commutator to
\eq{\label{dfsjlkdfsljkfds}\alg{
&\T_{\phi}\left(\underline q^{A,B,K}_{\geq 2}(\Phi)\cdot L_{U,V}f\right)=\\
&\qquad\Bigg(\sum_{K'\leq K, 1\leq j\leq K'}\sum_{(A_i),(B_i)}\left(\rho^{-K'/2}\prod_{i\in\{1,\hdots,K'\}\setminus\{j\}}\ab{L_{A_i,B_i}(\Phi)}\right)\cdot\\
&\qquad\qquad\qquad\qquad\qquad\qquad\qquad\cdot\frac{L_{A_j,B_j}(\Phi)}{\ab{L_{A_j,B_j}(\Phi)}}([\Tp,L_{A_j,B_j}]\Phi+L_{A_j,B_j}h)\Bigg)\cdot L_{U,V}f\\
&\qquad+\underline{q}^{A,B,K}_{\geq 2}(\Phi)\cdot\left[\Tp,L_{U,V} \right]f+4\underline{q}^{A,B,K}_{\geq 2}(\Phi)\cdot L_{U,V}(\T(\phi)\cdot f)+\hdots \Tp \rho^{-K'/2}\hdots,
}}
where the last terms comes with an overall negative sign and  hence we can ignore it in the estimates.
To evaluate the term containing $L_{A_j,B_j}h$ the following Lemma is used.
\begin{lemma}\label{lem-h-exp}
Let $A, B$ be two mult-indices such that $|A|+|B|\leq N$. The following identity holds:
\begin{align*}
L_{A,B} (h) =& \frac{t}{v^0}\left[\mathlarger{\mathlarger{\sum_{(\ast)}}}\left(1 +  \dfrac{q^{\alpha, P, \beta}\left(\Phi\right)}{t^\beta}  \right)\cdot \left(\dfrac{q^{\gamma, Q, \delta} \left(\Phi\right)}{t^{\frac{\delta}{2}}}  \right) p\left( \partial^{\delta +R+1}Z^{\beta} \phi  \right) \cdot \left(\dfrac{\zz}{\sqrt{t}}\right) ^{\mu}\right]\\
&+C,
\end{align*}
where the sum is taken over the non-negative integers $\alpha, P, \beta, \gamma, Q, \delta, R, \mu$ such that
\eq{ (\ast)\left\{
\begin{array}{l}
\al \leq |A|- 1,\, \be \leq |A|,\, \gamma\leq |A|- 1,\, \mu\leq \delta, \\
\al + \be + \gamma +  \delta  \leq |A|\\
P+Q+R \leq B\\
1\leq |B| \Rightarrow 1\leq P+Q+R
\end{array}\right.}
and $C$ denotes terms which are lower in regularity and similar or better in decay.
\end{lemma}

\begin{proof}
The formula follows immediately from Lemma \ref{Lonp}.
\end{proof}

The energy estimate for the norm $F_N[\Phi,f]$ is given in the following proposition.

\begin{proposition} \label{prop:Phif}
Assume that the bootstrap assumptions hold on $[1,P)$, then
\eq{\alg{
\mathbf F_{N-1}[\Phi,f](\rho)&\lesssim \varepsilon^{5/4} \rho^{2C\varepsilon^{1/16}}\\
\mathbf F_N[\Phi,f](\rho)&\lesssim \varepsilon^{5/4} \rho^{5C\varepsilon^{1/16}/2}\\
\mathbf F_N^{\circ}[\Phi,f](\rho)&\lesssim \varepsilon^{5/4} \rho^{C\varepsilon^{1/16}/2}\\
\mathbf F_{N}^{\ee{}}[\Phi,f]&\lesssim\varepsilon^{5/4}\rho^{5C\varepsilon^{1/16}/2}\cdot\log(\rho)\\
\mathbf F_{N-1}^{\ee{}}[\Phi,f]&\lesssim\varepsilon^{5/4}\rho^{2C\varepsilon^{1/16}}\cdot\log(\rho)\\
\mathbf F_{N}^{\ee{},\circ}[\Phi,f]&\lesssim\varepsilon^{5/4}\rho^{C\varepsilon^{1/16}/2}\cdot\log(\rho)\\
}}
hold for $\rho\in[1,\rho)$.
\end{proposition}
\begin{proof} We begin with the proof for the first three estimates.
According to \eqref{dfsjlkdfsljkfds} we need to estimate the four following terms.
\eq{\label{sdfoih}\alg{
\mathrm{I}=&\sum_{K'\leq K, 1\leq j\leq K'}\sum_{(A_i),(B_i)}\\
\int_{H_{\rho}}\int&\ab{\hat{\zz}}^C\rho^{-K'/2}\left(\left(\prod_{i\in\{1,\hdots,K'\}\setminus\{j\}}\ab{L_{A_i,B_i}(\Phi)}\right)[\Tp,L_{A_j,B_j}]\Phi\right)\cdot \ab{L_{U,V}f}dvd{\mu_\rho}\\
\mathrm{II}=&\sum_{K'\leq K, 1\leq j\leq K'}\sum_{(A_i),(B_i)}\\
\int_{H_{\rho}}\int&\ab{\hat{\zz}}^C\rho^{-K'/2}\left(\left(\prod_{i\in\{1,\hdots,K'\}\setminus\{j\}}\ab{L_{A_i,B_i}(\Phi)}\right)L_{A_j,B_j}h\right)\cdot \ab{L_{U,V}f}dvd{\mu_\rho}\\
\mathrm{III}&=\int_{H_{\rho}}\int\ab{\hat{\zz}}^C\underline{q}^{A,B,K}_{\geq 2}(\Phi)\cdot\ab{\left[\Tp,L_{U,V} \right]f} dvd{\mu_\rho}\\
\mathrm{IV}&=4\int_{H_{\rho}}\int\ab{\hat{\zz}}^C \underline{q}^{A,B,K}_{\geq 2}(\Phi)\cdot\ab{L_{U,V}(\T(\phi)\cdot f)}dvd{\mu_\rho}
}}
 We treat the terms individually in the following starting with \textrm{I}.
Using the commutator formula \eqref{comm-1-f} for the commutator acting on $\Phi$, we obtain for $\mathrm{I}$, suppressing the sums and replacing some $t^{-1}$ by $\rho^{-1}$ terms,

\eq{\label{fdsljhfsd}\alg{
&\int_{H_{\rho}}\int\ab{\hat{\zz}}^C\rho^{-K'/2}\left(\prod_{i\in\{1,\hdots,K'\}\setminus\{j\}}\ab{L_{A_i,B_i}(\Phi)}\right)\\
&\qquad\Bigg(\sum_{(*)}\rho^{\mu/2+\gamma+\nu/2}t^{\kappa-\sigma}\frac{q^{\al,S,\delta+\beta+\mu+2\gamma}(\Phi)}{\rho^{(\delta+\beta+\mu+2\gamma)/2}}\frac{\hat{\zz}^{\delta+\nu}}{(v^0)^{1-\sigma}}\\
&\qquad\quad  \cdot p(\partial^{\delta+I+1+\gamma-\sigma+\kappa+\lambda(1-\mu)}Z^{\beta-(\gamma-\sigma+\kappa+\lambda(1-\mu))}\phi)L_{P',Q'}(\Phi) \Bigg)\cdot \ab{L_{U,V}f}dvd{\mu_\rho}
}}
where the sum is taken over the variables $P', Q'$, multi-indices, and $\mu, \nu, \al, S, \delta, \be, \gamma, \sigma, \kappa, \lambda, I$ non-negative integers, such that
\eq{\label{sdfhfds}(*)\left\{\alg{
&|P'| = P+\lambda, |Q'| = Q+1-\lambda\\
&Q+S+I\leq \ab{B_j}\\
&P+\al+\be+\delta+\nu\leq \ab{A_j}\\
&\al+S+P+Q+\delta+I+\beta\leq \ab{A_j}+\ab{B_j}\\
&1\leq\delta+I+\beta\\
&\mu+\nu\leq 1,\mu=\gamma=1\Rightarrow \sigma=1\\
&\sigma\leq \gamma \leq 1\\
&\kappa=1\Rightarrow \mu=\gamma=\nu=\lambda=\sigma=0, \kappa \leq 1\\
&\lambda=1\Rightarrow \gamma=\nu=\sigma=0, \lambda\leq 1\\
&\lambda=\mu=1 \Rightarrow P+\lambda+\alpha+\beta+\delta\leq \ab{A_j}.
}\right.}
We distinguish between two different cases. In the first case, the wave can be estimated pointwise, while in the second case the wave is estimated in energy and all other terms are estimated pointwise. We begin with the first case, in which the condition
\eq{\label{eq:L1phifhigh}
\delta+I+\beta\leq N-2
}
holds. This implies an estimate on the wave term in the form
\eq{
\ab{p(\partial^{\delta+I+1+\gamma-\sigma+\kappa+\lambda(1-\mu)}Z^{\beta-(\gamma-\sigma+\kappa+\lambda(1-\mu))}\phi)}\lesssim \frac{\sqrt{\varepsilon}}{(1+u)^{1/2+\gamma-\sigma+\kappa+\lambda(1-\mu)}(1+t)}\rho^{C\varepsilon^{1/2}/2},
}
where the last $\rho$ factor only occurs at highest order. This reduces \eqref{fdsljhfsd} to
\eq{\alg{
\sqrt{\varepsilon}\rho^{C/2\varepsilon^{1/16}}\int_{H_{\rho}}\int&\ab{\hat{\zz}}^C\rho^{-K'/2}\left(\prod_{i\in\{1,\hdots,K'\}\setminus\{j\}}\ab{L_{A_i,B_i}(\Phi)}\right)\\
&\qquad\qquad\cdot\Bigg(\sum_{(*)}\frac{\rho^{\mu/2+\gamma+\nu/2}}{(1+u)^{1/2+\gamma-\sigma+\kappa+\lambda(1-\mu)}(1+t)}t^{\kappa-\sigma}\frac{q^{\al,S,\delta+\beta+\mu+2\gamma}(\Phi)}{\rho^{(\delta+\beta+\mu+2\gamma)/2}}\\
&\qquad\qquad\qquad\frac{\hat{\zz}^{\delta+\nu}}{(v^0)^{1-\sigma}}L_{P',Q'}(\Phi) \Bigg)\cdot \ab{L_{U,V}f}dvd{\mu_\rho}.
}}
Turning to the terms with $\Phi$ we estimate
\eq{
\rho^{-K'/2}\left(\prod_{i\in\{1,\hdots,K'\}\setminus\{j\}}\ab{L_{A_i,B_i}(\Phi)}\right)\frac{q^{\al,S,\delta+\beta+\mu+2\gamma}(\Phi)}{\rho^{(\delta+\beta+\mu+2\gamma)/2}}\ab{L_{P',Q'}(\Phi)}\lesssim \underline q^{A', B',\tilde K}_{\geq 2}(\Phi)
}
for some suitable $\tilde K$. This is possible since each factor comes with a corresponding $\rho^{-1/2}$ term and moreover if a term of zeroth or first order in $\Phi$ occurs we estimate it pointwise by Lemma \ref{lem-foe} in conjunction with $\rho^{-1/2}$. It remains to show that the parameters $A'$ and $B'$ fulfil the necessary bound to estimate the corresponding energy. Invoking the conditions \eqref{sdfhfds} we can conclude directly
\eq{
A'+B'\leq \ab{A}+\ab{B}.
}
The condition on $\tilde K$ is again implicit and automatically fulfilled by the bound on the total number of derivatives and the lower bound on the derivatives per factor. We can therefore reduce the integral to

\eq{\alg{
\sqrt{\varepsilon}\rho^{C/2\varepsilon^{1/16}}\sum_{(*)}\int_{H_{\rho}}\int&\ab{\hat{\zz}}^{C'}\left(\frac{\rho^{\mu/2+\gamma+\nu/2}}{(1+u)^{1/2+\gamma-\sigma+\kappa+\lambda(1-\mu)}(1+t)}t^{\kappa-\sigma}\frac{1}{(v^0)^{1-\sigma}} \right)\underline q^{A',B',\tilde K}_{\geq 2}(\Phi)\cdot \ab{L_{U,V}f}dvd{\mu_\rho}
}}
with $|C'|=C+\delta+\nu$ and the conditions on $A'$ and $B'$ inherited from the discussion above, which we evaluate at the appropriate places below again. Before addressing the different cases for the switch indices we discuss how the weights can be absorbed into the energy. By definition we can absorb the weights when
\eq{
C+\delta+\nu \leq  N+3-(A'+\ab{U}).
}

We estimate
\eq{
C+\delta+\nu=N+3-(A+|U|)\leq N+3- (A'+|U|)
}
if, and only if, $\ab A'+\delta+\nu\leq \ab {A}$. This holds since
\eq{
\ab {A'}+\delta+\nu=\ab{A}-\ab{A_j}+P+\alpha+\lambda+\delta+\nu\leq \ab {A}+\lambda
}
by \eqref{sdfhfds}, when $\lambda=0$ or $\lambda=\mu=1$. When $\lambda=1$ and $\mu=0$ we need to estimate one weight pointwise, which yields an additional factor of the form $(t/\sqrt\rho)^{\lambda(1-\mu)}$. This yields the final form before estimating to
\eq{\label{sdfihu}\alg{
\sqrt{\varepsilon}\rho^{C/2\varepsilon^{1/16}}\sum_{(*)}\int_{H_{\rho}}\int&\Bigg(\frac{\rho^{\mu/2+\gamma+\nu/2-\lambda(1-\mu)/2}}{(1+u)^{1/2+\gamma-\sigma+\kappa+\lambda(1-\mu)}(1+t)}t^{\kappa+\lambda(1-\mu)-\sigma}\left(\frac{\rho}{t}\right)^{1-\sigma}\Bigg)\\
&\qquad\cdot\left(\
\frac{t}\rho\frac{1}{v^0} \right)^{1-\sigma}\ab{\hat{\zz}}^{C'-\lambda}\underline q^{A',B',\tilde K}_{\geq 2}(\Phi)\cdot \ab{L_{U,V}f}dvd{\mu_\rho},
}}
where the last three factors in the integral are estimated jointly in energy. We estimate then the first factor in the integral  depending on the different cases of switch indices, which are $(\lambda=1,\mu\leq 1), \kappa=1,(\mu=1,\lambda=0),(\gamma=1,\mu=0)$ and $\nu=1$, by
\eq{\label{sdlh}
\frac{\rho^{\mu/2+\gamma+\nu/2-\lambda(1-\mu)/2}}{(1+u)^{1/2+\gamma-\sigma+\kappa+\lambda(1-\mu)}(1+t)}t^{\kappa+\lambda(1-\mu)-\sigma}\left(\frac{\rho}{t}\right)^{1-\sigma}\leq\rho^{-1}.
}
In turn, we estimate the respective integrals uniformly either by
\eq{
\sqrt{\varepsilon}\rho^{-1} \mathbf F_{N}[\Phi, f](\rho),
}
when $N-3$ or less derivatives act on $\partial \phi$ and the corresponding norm for the wave is not highest order, so the additional growth is not present; or by
\eq{
\sqrt{\varepsilon}\rho^{C/2\varepsilon^{1/16}-1} \mathbf F_{N-1}[\Phi, f](\rho),
}
when indeed $N-2$ derivatives act on $\partial \phi$ and we require the highest order energy for the wave. In that case however the combined norm cannot be of highest order.
The small $\rho$-growth disappears anyway when we estimate the corresponding norm of order $N-1$ or lower. In each case, the proposition is deduced using the bootstrap assumptions and the fact that the initial data for $\Phi$ is trivial.

 We turn now the the case when we estimate the wave in energy, which is characterized by
\eq{\label{eq:finL1philow}
I+\delta+\beta\geq N-1.
}
Invoking conditions \eqref{sdfhfds} this implies
\eq{
|Q|+S+|P|+\alpha+\sum_{i\neq j} (\ab{A_i}+\ab{B_i})+|U|+|V|\leq 1.
}
In particular this implies that all but one term containing $\Phi$ are at most first order in derivatives and hence can be estimated by using the pointwise estimates on $\Phi$. As these terms come with the appropriate $\rho^{-1/2}$-weight, we estimate them by a constant. In the case $|P|+|Q|=1$, the last $\Phi$ term can be of second order in derivatives. We therefore reduce the term \eqref{fdsljhfsd} to
\eq{\label{loc-term-f-3}\alg{
\sum_{(*)}\int_{H_{\rho}}&\left(\rho^{\mu/2+\gamma+\nu/2}t^{\kappa-\sigma}p(\partial^{\delta+I+1+\gamma-\sigma+\kappa+\lambda(1-\mu)}Z^{\beta-(\gamma-\sigma+\kappa+\lambda(1-\mu))}\phi)\right)\\
&\underbrace{\int\frac{\ab{\hat{\zz}}^{C'}}{(v^0)^{1-\sigma}}\rho^{-1/2}\ab{L_{P',Q'}(\Phi)} \cdot \ab{L_{U,V}f}dv}_{(\star)}d{\mu_\rho},
}}
where the conditions $|P|+|Q+|U|+|V|\leq 1$ and $|C'|=C+\delta+\nu$ hold. The term $(\star)$ is estimated pointwise by the low regularity norm $\mathbf F_N^{\circ}[\Phi,f]$ using the estimate \eqref{ineq:ksmsv}, implying
\eq{
\ab{(\star)}\lesssim \varepsilon t^{-3}\rho^{C\varepsilon^{1/16}}  \rho^{C\varepsilon^{1/16}/2}.
}
Note that due to the high number of weights in the low norm, we can always absorb and do not need to estimate a weight pointwise. Replacing in addition derivatives acting on the wave by vector fields we can estimate the previous integral by
\eq{\label{sdfiub}\alg{
&\varepsilon \rho^{3C\varepsilon^{1/16}/2}\rho^{\mu/2+\gamma+\nu/2}\int_{H_{\rho}}\frac{t^{\kappa-3-\sigma}}{(1+u)^{\gamma-\sigma+\kappa+\lambda(1-\mu)}}p(\partial^{\delta+I+1}Z^{\beta}\phi)d{\mu_\rho}\\
&\lesssim \varepsilon \rho^{3C\varepsilon^{1/16}/2}\rho^{\mu/2+\gamma+\nu/2}\mathscr E_{\delta+I+\beta}[\phi]^{1/2}\left(\int \frac{1}{(1+u)^{2(\gamma-\sigma+\kappa+\lambda(1-\mu))}}\frac{r^2}{t^{6-2(\kappa-\sigma)}} dr\right)^{1/2}\\
&\lesssim \varepsilon \rho^{3C\varepsilon^{1/16}/2}\rho^{\mu/2+\gamma+\nu/2}\mathscr E_{\delta+I+\beta}[\phi]^{1/2}\left(\frac{\rho^{-3-\boldsymbol\delta+2(\kappa-\sigma)}}{3-2(\kappa-\sigma)-\boldsymbol\delta}\right)^{1/2},
} }
where analogously to \eqref{sdfikl}, $\boldsymbol{\delta}<2(\gamma-\sigma+\kappa+\lambda(1-\mu))$ and $\boldsymbol{\delta}+2(\kappa-\sigma)<3$ hold. In all cases except $\kappa=1$ and $(\gamma=1,\nu\leq 1,\sigma=0)$ setting $\boldsymbol{\delta}=0$ the integral can be estimated by
\eq{
\varepsilon \rho^{3C\varepsilon^{1/16}/2-1}\mathscr E_{\delta+I+\beta}[\phi]^{1/2}.
}
In the case $\kappa=1$ we choose $\boldsymbol\delta=1-C\varepsilon^{1/16}/2$. Then we estimate the previous integral by
\eq{\label{lkjsad}
\varepsilon^{1-1/32}\rho^{7C\varepsilon^{1/16}/4-1}\mathscr E_{\delta+I+\beta}[\phi]^{1/2},
}
which yields the claim in the case of maximal regularity and below. In the left case ($\gamma=1$, $\nu\leq1$, $\sigma=0$) we choose $\boldsymbol{\delta}=2-C\varepsilon^{1/16}/2$ and obtain
\eq{
\varepsilon \rho^{7C\varepsilon^{1/16}/4-1}\mathscr E_{\delta+I+\beta}[\phi]^{1/2}.
}

In view of the bootstrap assumptions this yields the claim for the highest and lower orders and finishes the estimates for term $\mathrm{I}$. We proceed with the term $\mathrm{II}$, where we use Lemma \ref{lem-h-exp} to evaluate the term containing $h$. We deduce that, suppressing the sums again and estimating zeroth or first order terms in $\rho^{-1/2}\Phi$ pointwise by a constant using Lemma \ref{lem-foe}, $\ab{\mathrm{II}}$ can be estimated, suppressing terms better in regularity and/or decay, by
\eq{\label{sdfiuhsdf}\alg{
\sum_{(*)}&\int_{H_{\rho}}\int\ab{\hat{\zz}}^C\rho^{-K'/2}\left(\prod_{i\in\{1,\hdots,K'\}\setminus\{j\}}\ab{L_{A_i,B_i}(\Phi)}\right)\\
&\qquad\qquad\cdot\left(\frac{t}{v^0}\underline{q}_{\geq 2}^{\al,P,\beta+\delta}(\Phi)\ab{p^{\delta+R+1}Z^\beta(\phi)}\ab{\hat\zz}^{\delta}\right)\cdot \ab{L_{U,V}f}dvd{\mu_\rho},
}}
with conditions
\eq{(\ast)\left\{
\begin{array}{rl}
\al&\leq \ab{A_j}-1\\
P+R&\leq \ab{B_j}\\
\al+\beta+\delta&\leq \ab{A_j}.
\end{array}\right.}
Note first that due to the upper bound on the total number of derivatives acting on $\Phi$ factors, which is $N-1$, $h$ is always at most hit by $N-1$ derivatives, which corresponds to $N+1$ derivatives acting on $\phi$. Comparing this with the first term \eqref{fdsljhfsd} estimated before we see that the current term corresponds to the case $\kappa=1$ in \eqref{fdsljhfsd} since we have the factor $t/v^0$ and all switch indices vanish. The only difference is the fact that the additional partial derivative acting on $\phi$ is missing in the present case, but instead we have an explicit $\rho^{-1/2}$ coming from the $\rho^{-K'/2}$, which in \eqref{fdsljhfsd} compensated the last $\Phi$ term, which is here not present since it has been replaced by the right-hand side, $h$.

When estimating the decay this amounts to having a $\rho^{-1/2} $ instead of a $(1+u)^{-1}$. Checking the places where the $\kappa=1$ term has been estimated, which are \eqref{sdlh} and \eqref{lkjsad}, we conclude that replacing $(1+u)^{-1}$ by $\rho^{-1/2}$ therein yields identical or better conclusions in terms of decay. This implies that the previous integral can be estimated similar to term $\mathrm{I}$ and finishes the proof for the second term. The third term, $\mathrm{III}$, yields again the same structure by the higher order commutator structure and can be estimated analogously. Finally, the fourth term is simpler and can be estimated straightforwardly with either absorbing the $v^0$ that occurs in the low order norm, when $\Phi f$ is estimated pointwise or by estimating it by $t/\rho$ when $\Phi$ and $f$ are estimated in energy.

In the remainder we discuss the proof for the estimates concerning the energy $\mathbf F_N^{\ee{}}$ with at least one generalized translation based on the previous explicit estimates, which do not have to be repeated in all detail as the strategy of proof is identical. However, it is important to analyze how the presence of generalized translations is affected by the commutator relations.

As above the energy estimate builds on four terms of the form \eqref{sdfoih} with one $\rho^{1/2}$ weight and one generalized translation contained within the $L_{A_i,B_i}$ operators. As before the crucial terms to estimate are the first and the second one of \eqref{sdfoih} and what might occur is that the $\ee{}$ acting on $\Phi$ in the original energy by the commutation now acts on the wave leaving possibly no generalized translation acting on $\Phi$. In that case this could not be estimated by the $\mathbf{F}_N^{\ee{}}$-energy making an additional $\rho^{-1/2}$ factor necessary to estimate by the standard energy $\mathbf F_N$. However, by equation \eqref{fdsljhfsd} we deduce immediately that there is always one $\ee{}$ hitting $\Phi$ as long as $\lambda=0$. So in these cases we are in a situation when at least one $\Phi$, present in the product, is hit by a generalized translation and in turn can be either estimated by the corresponding pointwise estimate for $\ee{}\Phi$, yielding a $\log(\rho)$ (by the second estimate in Lemma \ref{lem-foe}) or it is absorbed into the energy. In both cases the presence of $\rho^{1/2}$ is consistent with the estimates as performed for term $\mathrm{I}$ corresponding to $\mathbf F_N$ and we do not have to repeat them in detail. It remains to consider the case $\lambda=1$. In that case the term corresponding to \eqref{sdfihu} takes the form
\eq{\alg{
\sqrt{\varepsilon}\rho^{C/2\varepsilon^{1/16}}\log(\rho)\sum_{(*)}\int_{H_{\rho}}\int&\underbrace{\left(\frac{\rho^{1/2+\mu/2-\lambda(1-\mu)/2}}{(1+u)^{3/2+\lambda(1-\mu)}(1+t)}t^{\lambda(1-\mu)}\left(\frac{\rho}{t}\right)\right)}_{(\star)}\\
&\left(\
\frac{t}\rho\frac{1}{v^0} \right)\ab{\hat{\zz}}^{C'-\lambda}\underline q^{A',B',\tilde K}_{\geq 2}(\Phi)\cdot \ab{L_{U,V}f}dvd{\mu_\rho},
}}
where the additional $\rho^{1/2}$ factor occurs as it cannot be absorbed by the $\underline q$ term and the additional $(1+u)^{-1}$ results from the $\ee{}$ acting on $\phi$. The $\log(\rho)$ term appears in the case that $\ee{} \Phi$ is estimated pointwise. For both cases $\mu\leq 1$ the term $(\star)$ is bounded by $\rho^{-1}$ and we can conclude as above.

 The case when the wave is estimated in energy can be handled analogous to \eqref{sdfiub}. Then again, the generalized translation hits the wave, yielding another $(1+u)^{-1}$ factor while the additional $\rho^{1/2}$ cannot be absorbed into the $\Phi$ terms. The corresponding estimate to \eqref{sdfiub} then yields a term of the form
\eq{
\varepsilon \rho^{3C\varepsilon^{1/2}/2}\rho^{\mu/2+1/2}\mathscr E_{\delta+I+\beta}[\phi]^{1/2}\left(\frac{\rho^{-3-\boldsymbol\delta}}{3-\boldsymbol\delta}\right)^{1/2},
}
where due to the additional $(1+u)^{-1}$, $\boldsymbol{\delta}=1$ is allowed and we choose it such. In particular the term above can be estimated by
\eq{
\lesssim\varepsilon \mathscr E_{\delta+I+\beta}[\phi]^{1/2}\rho^{-1+3C\varepsilon^{1/16}/2},
}
which allows to proceed as above. The last case we discuss explicitly concerns the analogue to term $\mathrm{II}$ which arises from the source term. Here, one of the $\Phi$ terms is replaced by the source term. If this term carried the $\ee{}$ derivative we simply have no $\rho^{1/2}$ to add as a factor and have an additional translation acting on the wave. Replacing this term by a $(1+u)^{-1}$ as usual, the term $t/v^0$ in \eqref{sdfiuhsdf} is replaced by $t/v^0 (1+u)^{-1}$. This corresponding exactly to the case $\kappa=1$ above and is handled accordingly. In the other case when the generalized translation is in one of the $\Phi$ factors, we can use a $\rho^{-1/2}$, which is not required to normalize this $\Phi$ and get a term as in \eqref{sdfiuhsdf} with $t/v^0$ replaced by $t/v^0 \rho^{-1/2}$, which can be dealt with in the same way. The analogues of $\mathrm{III}$ and $\mathrm{IV}$ are handled as above and do not require further evaluation. This finishes the proof.
\end{proof}

\section{$L^2$-estimates for the transport equation} \label{sec:L2wave}

This section is devoted to the proof of $L^2$-estimates for terms of the form $q^{\al, \be, S}(\Phi)\YY^\al f$, when $|\alpha| \geq N-2$, that is to say when pointwise estimates cannot be performed on $\YY^\al f$. The strategy developed in this section is similar to the strategy developed in the the original work \cite{fjs:vfm}, the only difference is coming from the modification $\Phi$ of the vector fields. It should be noticed that this section can be used to prove again the $L^1$-estimates, in the sense that the basic structure which is derived in that section is exactly the one used to derive the \(L^1\) and combined \(L^1\)-estimates.

In what follows, we are working symbolically with vectors and matrices. Up to completion with components with value 0, we assume that all matrices are square matrices of the same order, and all the vectors have the same length, compatible with the multiplication with these matrices. $|X|$ denotes the standard $L^\infty$-norm on matrices defined as the supremum of all the components.

Finally, we emphasize that, in this section, we do need a decay of the form $\rho^{-3+ \delta(\varepsilon)}$, where $\delta(\varepsilon)<1/2$. The actual value of $\delta(\varepsilon)$ does not actually matter, contrary to Section \ref{sec:L1estimates}. \bfseries{Hence, we drop systematically in that section the constant in front of the $\varepsilon^{1/16}$, and write instead a generic $C$ constant that differs from line to line.}\mdseries

\subsection{The basic equations}

We introduce the vector $G^h$ containing the terms
$$
\hat{\zz}^{C}  \underline{q}_{\geq 0}^{A,B,K}(\Phi) \cdot L_{U,V}^\pi f
$$
where $U, V$ are multi-indices, and $A, B, K$ are non-negative integers such that:
\eq{ (\ast)\left\{
\begin{array}{l}
N-2\leq |U|+|V|\leq N\\
A+B+|U|+|V|\leq N\\
\pi \mbox{ any permutation of }  A+ B \mbox{ elements}\\
\ab{C}= N+3-(A+|U|).
\end{array}\right.
}
By construction, and using the pointwise estimates for $\Phi$ (see Lemmata \ref{es:inht} and \ref{lem-foe}), the norm of $G^h$ satisfies
$$
E_0[\vert G^h \vert] \lesssim E_N[f] + \mathbf{F}_N[\Phi, f].
$$

We also introduce the vector $\tilde{G}^{h}$, containing the same terms, \emph{ up to the order } $N-1$, that is to say that $\tilde{G}^{h}$ contains the term
$$
\hat{\zz}^{C}  \underline{q}_{\geq 0}^{A,B,K}(\Phi) \cdot L_{U,V}^\pi f
$$
where $U, V$ are multi-indices, and $A, B, K$ are non-negative integers such that:
\eq{ (\ast)\left\{
\begin{array}{l}
N-2\leq |U|+|V|\leq N-1\\
A+B+|U|+|V|\leq N-1\\
\pi \mbox{ any permutation of }  A+ B \mbox{ elements}\\
\ab{C}= N+2-(A+|U|).
\end{array}\right.
}
By construction, and using the pointwise estimates for $\Phi$ (see Lemmata \ref{es:inht} and \ref{lem-foe}), the norm of $\tilde{G}^{h}$ satisfies
$$
E_0[\vert \tilde{G}^h \vert] \lesssim E_{N-1}[f] + \mathbf{F}_{N-1}[\Phi, f].
$$

Using the commutator formula of Corollary \ref{cor-comm-f}, as well as the commutator formula of Equation \eqref{dfsjlkdfsljkfds}, one obtains
\eq{\label{eq:commrepeated}\alg{
&\T_{\phi}\left(\underline q^{A,B,K}_{\geq 2}(\Phi)\cdot L_{U,V}f\right)=\\
&\qquad\Bigg(\sum_{K'\leq K, 1\leq j\leq K'}\sum_{(A_i),(B_i)}\left(\rho^{-K'/2}\prod_{i\in\{1,\hdots,K'\}\setminus\{j\}}\ab{L_{A_i,B_i}(\Phi)}\right)\cdot\\
&\qquad\qquad\qquad\qquad\qquad\qquad\qquad\cdot\frac{L_{A_j,B_j}(\Phi)}{\ab{L_{A_j,B_j}(\Phi)}}([\Tp,L_{A_j,B_j}]\Phi+L_{A_j,B_j}h)\Bigg)\cdot L_{U,V}f\\
&\qquad+\underline{q}^{A,B,K}_{\geq 2}(\Phi)\cdot\left[\Tp,L_{U,V} \right]f+4\underline{q}^{A,B,K}_{\geq 2}(\Phi)\cdot L_{U,V}(\T(\phi)\cdot f)+\hdots \Tp \rho^{-K'/2}\hdots,
}}
with (see Lemma \ref{lem-h-exp})
\eq{
L_{A,B} (h) = \frac{t}{v^0}\left[\mathlarger{\mathlarger{\sum_{(\ast)}}}\left(1 +  \dfrac{q^{\alpha, P, \beta}\left(\Phi\right)}{t^\beta}  \right)\cdot \left(\dfrac{q^{\gamma, Q, \delta} \left(\Phi\right)}{t^{\frac{\delta}{2}}}
\right) p\left( \partial^{\delta +R+1}Z^{\beta} \phi  \right) \cdot \left(\dfrac{\zz}{\sqrt{t}}\right) ^{\mu}\right],
}
where the sum is taken over the non-negative integers $\al, P, \be, \gamma, Q, \delta, R, \mu$ such that
\eq{ (\ast)\left\{
\begin{array}{l}
\al \leq |A|- 1,\,\be \leq |A|,\, \gamma\leq |A|- 1,\, \mu\leq \delta, \\
\al + \be + \gamma +  \delta  \leq |A|\\
P+Q+R \leq |B|
\end{array}\right..
}

The strategy is the following: using the commutator formula of Lemma \ref{cor-comm-f}, and its variation \eqref{eq:commrepeated}, when $f$ is combined with $\Phi$, we identify the components of an equation satisfied by \(\Tp G^h\) and \(\Tp \tilde{G}^h\). The analysis is based on a re-evaluation of Sections \ref{sec:L1estimates} and \ref{sec:Phif}, where the dominant term of the commutator (identified by \(\ast\) in Equation \eqref{eq:commrepeated}) is dealt with in detail. All the other terms appearing are either of the same nature, or of lower order.

In the previous formula, we now distinguish between two cases: either the derivatives of $\phi$ can be estimated pointwise (when $\delta+I+\beta <N-1$), or they cannot (when $\delta+I+\beta \geq N-1$), in the commutator formula.

Consider first $\delta+I+\beta <N-1$; this corresponds in the $L^1$-estimates to Equations \eqref{L1-f-high}--\eqref{eq:L1-f-low}, and in the combined $L^1$- estimates to Equations \eqref{eq:L1phifhigh}--\eqref{eq:finL1philow}. We observe therein that all the terms in factor of the derivatives of $\phi$ are listed in $G^h$ or $\tilde{G}^{h}$. We define by $\mathbf{A}$
the matrix containing the terms of the form (completed by component with value 0 whenever necessary)
$$
\left(\rho^{\mu/2+\gamma-\sigma+\nu/2}t^{\kappa}p(\partial^{\delta+I+1+\gamma-\sigma+\kappa+\lambda}Z^{\beta-(\gamma-\sigma+\kappa+\lambda)}\phi)\right),
$$
 with $\delta+I+\beta < N-2$. Using the equations from \eqref{L1-f-high}--\eqref{eq:L1-f-low}, and from \eqref{eq:L1phifhigh} to the end of the corresponding proof, one obtains that
$$
\vert \mathbf{A}\vert \lesssim  \sqrt{\varepsilon} \rho^{-1},
$$
since $\mathbf{A}$ does not contain the highest order in $\phi$ for which the loss of decay occurs. The components of the matrix $\mathbf{A}$ are in factor of the term $G^h$.
We define by $\mathbf{B_2}$
the matrix containing the terms of the form (completed by 0 whenever necessary)
$$
\left(\rho^{\mu/2+\gamma-\sigma+\nu/2}t^{\kappa}p(\partial^{\delta+I+1+\gamma-\sigma+\kappa+\lambda}Z^{\beta-(\gamma-\sigma+\kappa+\lambda)}\phi)\right),
$$
 with $\delta+I+\beta < N-1$. By the same argument as previously, the components of $\mathbf{B}_2$ satisfy
 $$
 \vert \mathbf{B}_2\vert \lesssim  \sqrt{\varepsilon} \rho^{-1 +C\varepsilon^{1/16}}.
 $$

Consider now the complementary case $\delta+I+\beta \geq N-1$. This corresponds to the situation when $\phi$ is estimated in energy. This situation is addressed in Equations \eqref{eq:L1-f-low} to the end of the proof for the  $L^1$-estimates for $f$, and in Equations \eqref{eq:finL1philow} to the end of the proof for the combined $L^1$- estimates. Equation \eqref{eq:commrepeated} reduces to (this corresponds to Equations \eqref{loc-term-est-2} and \eqref{loc-term-f-3}):
\begin{gather}
 \left(\dfrac{\rho}{t}\right)^{1-\sigma}p(\partial^{\delta+I+1+\gamma-\sigma+\kappa+\lambda(1-\mu)}Z^{\beta-(\gamma-\sigma+\kappa+\lambda(1-\mu))}\phi)
 \rho^{\mu/2+\gamma+\nu/2}t^{\kappa-\sigma}\left(\dfrac{t}{\sqrt{\rho}}\right)^{\lambda(1-\mu)} \label{eq:terma}\\
\cdot \left(\dfrac{\sqrt{\rho}}{t}\right)^{\lambda(1-\mu)} \left(\dfrac{t}{\rho}\right)^{1-\sigma} \frac{q^{\al,S,\delta+\beta+\mu+2\gamma}(\Phi)}{\rho^{(\delta+\beta+\mu+2\gamma)/2}}
 \frac{\hat{\zz}^{\delta+\nu}}{(v^0)^{1-\sigma}}  L_{A_i,B_i}(\Phi)L_{P+\lambda,Q+1-\lambda}(\Phi)\cdot \ab{L_{U,V}f} \nonumber
\end{gather}
Define $\mathbf{B}_1$ as
 the matrix containing the terms of Equation \eqref{eq:terma}, for the given value of $\delta$ and $\beta$, and where the other Greek letters stand for the switch symbols as described in the parametrization of the commutator formula \eqref{sdfhfds}. The term \(\left(\frac{t}{\sqrt{\rho}}\right)^{\lambda(1-\mu)} \) is added manually to compensate for the extra $\zz$-weight that needs to be absorbed when $\kappa= 0, \lambda =1, \mu=1$ (see the discussion between Equations \eqref{L1-est-ter} and \eqref{absorb-loc}).The term containing the derivatives in $\phi$ can be bounded by
 $$
 p(\partial Z^{1+\delta+I}\phi)\underbrace{
 (1+u)^{-\gamma+\sigma-\kappa-\lambda(1-\mu)} \rho^{\mu/2+\gamma+\nu/2}t^{\kappa-\sigma} \left(\dfrac{\rho}{t}\right)^{1-\sigma}\left(\dfrac{t}{\sqrt{\rho}}\right)^{\lambda(1-\mu)}}_{(\star)}.
 $$
We now need to discuss the value of the term $(\star)$ in factor of the derivative in $\phi$; the discussion is based on the system \eqref{sdfhfds}:
\begin{itemize}
\item for $\kappa =1$, then all other coefficients vanish and remains
$$
\dfrac{t}{1+u}\dfrac{\rho}{t}\lesssim \dfrac{t}{\rho};
$$
\item for \(\kappa=0\), \(\lambda =1\); in this situation, depending on the value of $\mu$, it is necessary to absorb one weight then all other coefficients vanish but \(\mu\) (see the discussion between Equations \eqref{L1-est-ter} and \eqref{absorb-loc});
\begin{itemize}
  \item in the case \(\mu = 0\):
  $$
\dfrac{\rho}{t} \cdot (1+u)^{-1} \ab{\dfrac{t}{\sqrt{\rho}}} \lesssim \sqrt{\rho};
  $$
  \item in the case \(\mu=1\):
  $$
  \dfrac{\rho}{t} \cdot \rho^{1/2} \lesssim \sqrt{\rho};
  $$
  \end{itemize}
\item for $\kappa=0$, $\lambda =0$, we have $\mu+\nu\leq 1$, and we shall now discuss on the value of $\gamma$ and $\sigma$:
\begin{itemize}
  \item if $\gamma=\sigma=0$, then remains
  $$
\rho^{\mu/2+\nu/2}\dfrac{\rho}{t}\leq \sqrt{\rho};
  $$
  \item if $1=\gamma>\sigma=0$, then remains
  $$
\rho^{\mu/2+\nu/2}\dfrac{\rho}{t}\cdot \dfrac{\rho}{(1+u)} \lesssim \sqrt{\rho};
  $$
  \item if $\gamma=\sigma=1$, then remains
  $$
\rho^{\mu/2+\nu/2} \dfrac{\rho}{t} \leq \sqrt{\rho}.
  $$
\end{itemize}
\end{itemize}
Hence, using the bootstrap assumption \eqref{eq:bs3}, there exists a positive function $\psi$ such that
 $$
\vert\mathbf{B}_1 \vert \lesssim  \max\left(\dfrac{t}{\rho}, \sqrt{\rho} \right)  \psi
 $$
where
$$
\mathscr{E}_0[\psi](\rho) \lesssim \varepsilon \rho^{ C\varepsilon^{1/16} }
$$
Finally, define $G^l$ as the vector containing components of the form
$$
\left(\dfrac{t}{\rho}\right)^{1-\sigma} \dfrac{ \hat{\zz}^{C'}  }{ (v^0)^{1-\sigma} }
\rho^{-1/2} L_{ P', Q'}(\Phi)
\cdot L_{U,V}f
$$
where
$$
|P'| =  P+\lambda, \, |Q'|= Q+1-\lambda, \,P+Q+U+V\leq 1\, \text{ and } |C'| \leq N+5-|U| -P.
$$
For this term, the Klainerman-Sobolev estimates \eqref{ineq:ksmsv}, as argued in Section \ref{sec:Phif}, guarantee that pointwise estimates can be performed, that is to say that
$$
\int_{v} \vert G^l \vert d \mu(v) \lesssim \dfrac{\rho^{C\varepsilon^{1/16}} }{t^3}
\varepsilon
$$
after inserting the result of Proposition \eqref{prop:Phif}.

Altogether, up to good symbols and lower order terms, we can conclude that $G^h$ satisfies a transport equation whose coefficients are controlled by the Klainerman-Sobolev estimates and the bootstrap assumptions, and all other terms of Equation \eqref{eq:commrepeated} fit in this discussion. This is summarized in the following lemma.

\begin{lemma} \label{lem:eqggfull} With the definition above, the commutator formula \eqref{eq:commrepeated} for the combined estimates can be rewritten as
  $$
\Tp G^h + \mathbf{A} G^h =  \mathbf{B}_1G^l +  \mathbf{B}_2
  \tilde{G}^h
  $$
where
\begin{itemize}
\item for $\mathbf{A}$:
$$
\vert \mathbf{A}\vert  \lesssim \sqrt{\varepsilon} \rho^{-1 };
$$
\item for $\mathbf{B}_1$: there exists a positive function $\psi$ such that
$$
\vert\mathbf{B}_1 \vert \lesssim  \max\left(\dfrac{t}{\rho}, \sqrt{\rho} \right)  \psi
$$
where
$$
\mathscr{E}_0[\psi](\rho) \lesssim \varepsilon \rho^{ C\varepsilon^{1/16}};
$$
 \item the pointwise estimates hold for $G^l$:
 $$
 \int_{v} \vert G^l \vert d \mu(u) \lesssim \dfrac{\rho^{C\varepsilon^{1/16}} }{t^3}
 \varepsilon
 $$
 \item the matrix $\mathbf{B}_2$ satisfies
 $$
\vert \mathbf{B}_2\vert  \lesssim \sqrt{\varepsilon} \rho^{-1 + C\varepsilon^{1/16}};
$$
\end{itemize}
\end{lemma}

Similarly, one notices that the only difference between $G^h$ and $\tilde{G}^h$ is the fact that in the commutation formula, $\partial\phi$ is, for $\tilde{G}^h$, never at the highest order. Hence, the same equation as for $G^h$ holds, without the correction term $\mathbf{B}_2$ originating from the lack of decay of the solution to the wave equation at higher order:

\begin{lemma} \label{lem:eqgfull} With the definition above, the commutator formula \eqref{eq:commrepeated}, at order $N-1$, for the combined estimates can be rewritten as
  $$
\Tp \tilde{G}^h + \mathbf{\tilde{A}} \tilde{G}^h =  \mathbf{\tilde{B}}_1\tilde{G}^l
  $$
where
\begin{itemize}
\item for $\mathbf{\tilde{A}}$:
$$
\vert  \mathbf{\tilde{A}}\vert  \lesssim \sqrt{\varepsilon} \rho^{-1 };
$$
\item for $\tilde{\mathbf{B}_1}$: there exists a positive function $\psi$ such that
$$
\vert \mathbf{\tilde{B}}_1 \vert \lesssim  \max\left(\dfrac{t}{\rho}, \sqrt{\rho} \right)  \psi
$$
where
$$
\mathscr{E}_0[\psi](\rho) \lesssim \varepsilon \rho^{ C\varepsilon^{1/16}};
$$
 \item the pointwise estimates hold for $G^l$:
 $$
 \int_{v} \vert \tilde{ G} ^l \vert d \mu(u) \lesssim
 \varepsilon  \dfrac{\rho^{C \varepsilon^{1/16}} }{t^3}
 $$
\end{itemize}
\end{lemma}

An immediate by-product of the choice of the definition of $G^l$ is the following remark: Since $G^l$ contains only terms which are of very low order, when applying $\Tp$ to $G^l$, the resulting commutator contains derivatives in $\phi$ which all can be estimated pointwise. Hence, the estimates of these terms in the $L^1$-estimates section (Section \ref{sec:L1estimates}) from Equation \eqref{eq:L1-f-low} to the end of the proof, and in the combined estimate section (Section \ref{sec:Phif}) from Equation \eqref{eq:L1phifhigh} to the end of the proof are the only to be considered. As a consequence, the following Lemma holds, for $G^l$ as well as $\tilde{G}^l$:
\begin{lemma} \label{lem:eqgl}There exists a matrix $\hat{\mathbf{A}}$ satisfying
  $$
\vert \mathbf{\hat{A}} \vert \lesssim \sqrt{\varepsilon}  \rho^{-1 }
  $$
  such that
  $$
\Tp G^l =  \mathbf{\hat{A}} G^l.
  $$
  In the same way, there exists a matrix $\hat{\tilde{\mathbf{A}}}$ satisfying
  $$
\vert \mathbf{\hat{\tilde{A}}} \vert \lesssim \sqrt{\varepsilon}  \rho^{-1 }
  $$
  such that
  $$
\Tp \tilde{G}^l =  \mathbf{\hat{\tilde{A}}} \tilde{G}^l.
  $$
\end{lemma}

\subsection{$L^2$-estimates for low orders} \label{sec:L2low}

We tackle in this section the $L^2$ estimates for the low order derivatives of $f$.

Our strategy to prove the $L^2$-estimates is based on a careful decomposition of $G^h$, i.e. we split the solution to the equation
$$
\T_\phi \tilde{G}^h  + \mathbf{\tilde{A}} \tilde{G}^h = \mathbf{B}_1 G^l.
$$
To explain the splitting, recall that the main obstruction to proving decay of $G^h$ by commuting lies in the right-hand side of the equation, where the regularity of the terms appearing there is still too low to perform commutation estimates. This motivates a splitting between the homogeneous part and the inhomogeneous part as follows:
$$
\tilde{G}^h =  \tilde{G}_{hom} +  \tilde{G}_{inh},
$$
where
\begin{equation}\label{eq:L2hom}
\T_\phi \tilde{G}_{hom}  + \mathbf{\tilde{A}}\tilde{ G}_{hom} = 0
\end{equation}
with the corresponding initial datum on the initial hyperboloid $H_1$, and
$$
\T_\phi \tilde{G}_{inh}  + \mathbf{\tilde{A}} \tilde{G}_{inh} = \mathbf{\tilde{B}}_1 \tilde{G}^l
\text{ with } \tilde{G}_{inh}|_{H_1} = 0.
$$

To deal with $\tilde{G}_{inh}$, one can use the strategy introduced in \cite{fjs:vfm}: Introduce $\mathbf{\tilde{K}}$, the solution to  \begin{equation}
  \Tp \mathbf{\tilde{K}}  + \mathbf{\tilde{A}} \mathbf{\tilde{K}}+ \mathbf{\tilde{K}}\hat{\mathbf{\tilde{A}}} = \mathbf{\tilde{B}}_1, \text{ with }\mathbf{K}|_{H_1} = 0
\end{equation}
and notice that $\tilde{G}_{inh}$ admits the representation
$$
\tilde{G}_{inh} = \mathbf{\tilde{K}} \tilde{G}^l.
$$
The same strategy and result as in \cite{fjs:vfm} then holds.

\bfseries{The homogeneous solution $\tilde{G}_{hom}$.}\mdseries~The strategy in that paragraph is exactly the same as in \cite[p. 57]{fjs:vfm}. Since we assume that we initially control $N+3$ derivatives of $f$, we can still commute Equation \eqref{eq:L2hom} three times with the vectors $\YY^\al$. For $|\al|\leq 3$, we have
\begin{equation}
\Tp (\YY^\al \tilde{G}_{hom} + \mathbf{\tilde{A}} \YY^\al \tilde{G}_{\hom} ) = - \YY^{\al}\left(\mathbf{\tilde{A}}\tilde{G}_{\hom} \right)  + \left[\Tp, \YY^{\al} \mathbf{\tilde{A}} \right] \tilde{G}_{hom}.
\end{equation}
Since $\mathbf{\tilde{A}}$ contains up to $N-5$ derivatives of $\partial \phi$, and since we control up to $N-2$ derivatives of $\phi$ pointwise in energy, we obtain, by the Klainerman-Sobolev estimates for the wave equation \eqref{eq:kswh}, as well as by the bootstrap assumption \eqref{eq:bs3}, for $\vert\nu \vert\leq 3$,
$$
\vert \YY^{\nu}\mathbf{\tilde{A}}\vert  \lesssim \dfrac{\sqrt{\varepsilon}  }{\rho^{3/2}}.
$$
Hence, one obtains immediately that, using the commutator estimates \eqref{ineq:ksmsv},
$$
E_0[\vert \tilde{G}_{hom} \vert ](\rho) \lesssim \varepsilon \rho^{C \varepsilon}.
$$
As a consequence, the following pointwise estimates for $G_{hom}$ hold:
\begin{lemma}\label{lem:ghomest} $\tilde{G}_{hom}$ satisfies:
  $$
\int_{v}\vert \tilde{G}_{hom}\vert d\mu(v) \lesssim \dfrac{\varepsilon \rho^{C\varepsilon^{1/16}} }{t^3}.
  $$
  As a consequence, the components $\tilde{G}^\al{hom}$ of $\tilde{G}_{hom}$ satisfy
  $$
  \left(\int_{H_\rho}\frac{t}{\rho} \left(\int_v |\tilde{G}^\al_{{hom}}| \frac{dv}{v^0}\right)^{2}d\mu_{H_\rho}\right)^{\frac12} \lesssim \varepsilon \rho^\frac{C\varepsilon^{1/16} -3}{2}
  $$
\end{lemma}

\bfseries{The inhomogeneous solution $\tilde{G}_{inh}$.} \mdseries The strategy developed for $\tilde{G}_{inh}$ is strictly the same as in \cite[p.57 and following]{fjs:vfm}. Before performing the estimates on $\tilde{G}_{inh}$, let us remind that
\begin{itemize}
 \item using the bootstrap assumption \eqref{eq:bs1} and \eqref{ineq:ksmsv}, the $v$ integrals of $|\tilde{G}^l|$ can be estimated pointwise, and, by Lemma \ref{lem:eqgfull},
 $$
 \int_v |\tilde{G}^l| dv \lesssim \varepsilon t^{C \varepsilon^{1/16}-3}
 $$
 \item the components of the source terms in the equation satisfied by $\mathbf{\tilde{K}}$ can be estimated in $L^2(H_\rho)$ (cf. Lemma \ref{lem:eqgfull}).
\end{itemize}

Following the strategy described  in \cite{fjs:vfm}, we define
$$
|\mathbf{\tilde{K}\tilde{K}\tilde{G}}| = \sum_{\alpha, \beta, \gamma, \kappa, \mu} |\tilde{K}_{\alpha}^{\beta}\tilde{K}_{\gamma}^{\kappa}\tilde{G}_{\mu}^l|,
$$
where $\tilde{K}_{\alpha}^{\beta}$, $\tilde{K}_{ \gamma}^{\kappa}$ are the components of the matrix $\mathbf{\tilde{K}}$ and $\tilde{G}_{\mu}^l$ the components of the vector $\tilde{G}^l$, and where the sum is taken over all possible combinations of two elements of $\mathbf{\tilde{K}}$ and $\tilde{G}^l$. One furthermore easily checks that each element of this sum satisfies the equation
\begin{eqnarray}
 \T_{\phi}\left(\tilde{K}_{\alpha}^{\beta}\tilde{K}_{\gamma}^{\kappa}\tilde{G}_{\mu}^l \right) &=& \left(\tilde{A}^\sigma_{\alpha}\tilde{K}_{\sigma}^{\beta} + \tilde{K}_{\beta}^{\sigma}\hat{A}_{\sigma}^\alpha \right)\tilde{K}_{\gamma}^{\kappa}\tilde{G}_{\mu}^l
 +\tilde{K}_{\alpha}^{\beta}  \left(\tilde{A}^\sigma_{\gamma}\tilde{K}_{\sigma}^{\kappa} + \tilde{K}_{\gamma}^{\sigma}\hat{A}_{\sigma}^\kappa \right)\tilde{G}_{\mu}^l\nonumber\\
 &&+ \tilde{K}_{\alpha}^{\beta}\tilde{K}_{\gamma}^{\kappa}\hat{\tilde{A}}^{\sigma}_{\mu}\tilde{G}_{\sigma}^l \label{eq:bgg1}\\
  &&- \left( \tilde{B}_{\beta}^{\alpha} \tilde{K}_{\gamma}^{\kappa} + \tilde{K}_{\beta}^\alpha \tilde{B}_{\gamma}^{\kappa}\right)\tilde{G}_{\mu}^l \label{eq:bgg2},
\end{eqnarray}
where $\tilde{B}_{\beta}^{\alpha} $
are the components of $\mathbf{\tilde{B}}_1$
$\tilde{A}_{\beta}^{\alpha}$
 and $\hat{\tilde{A}}_{\beta}^{\alpha}$ the components of $\mathbf{\tilde{A}}$,
 and $\mathbf{\hat{\tilde{A}}}$
  respectively. As a consequence,
\(|\mathbf{\tilde{K}\tilde{K}\tilde{G}}|\) satisfies an estimate of the form:
\begin{lemma}\label{lem:estbbginh} The function $|\mathbf{\tilde{K}\tilde{K}\tilde{G}}|$ satisfies $$
  E_0[|\mathbf{\tilde{K}\tilde{K}\tilde{G}}|](\rho) \lesssim  \varepsilon^{3/2-1/8} \rho^{C \varepsilon^{1/16}}.
      $$
\end{lemma}
\begin{proof}
The right-hand side of Equation \eqref{eq:bgg1} can easily be estimated, using the properties of the matrices $\mathbf{\tilde{A}}$ and $\mathbf{\hat{A}}$ stated in Lemmata \ref{lem:eqgfull} and \ref{lem:eqgl}, by
\begin{gather*}
\int_{H_\rho} chi\left(\left|\left(  \tilde{A}^\sigma_{\alpha}\tilde{K}_{\sigma}^{\beta} + \tilde{K}_{\beta}^{\sigma}\hat{A}_{\sigma}^\alpha \right)\tilde{K}_{\gamma}^{\kappa}\tilde{G}_{\mu}^l
 +\tilde{K}_{\alpha}^{\beta}  \left(\tilde{A}^\sigma_{\gamma}\tilde{K}_{\sigma}^{\kappa} + \tilde{K}_{\gamma}^{\sigma}\hat{\tilde{A}}_{\sigma}^\kappa \right)\tilde{G}_{\mu}^l+ \tilde{K}_{\alpha}^{\beta}\tilde{K}_{\gamma}^{\kappa}\hat{A}^{\sigma}_{\mu}\tilde{G}_{\sigma}^l  \right|\right) d\mu_{H_{\rho }}
 \\ \lesssim
\frac{\sqrt{\varepsilon}}{\rho} E_0[|\mathbf{\tilde{K}\tilde{K}\tilde{G}}|](\rho).
\end{gather*}

Furthermore, by the Cauchy-Schwarz inequality, as well as the property of the matrix $\tilde{\mathbf{B}}_1$ stated in Lemma \ref{lem:eqgfull}, one gets, when estimating Equation \eqref{eq:bgg2} :
\begin{eqnarray*}
 \int_v \left(\left| \left( \tilde{B}_{\beta}^{\alpha} \tilde{K}_{\gamma}^{\kappa} + \tilde{K}_{\beta}^\alpha \tilde{B}_{\gamma}^{\kappa}\right)\tilde{G}_{\mu}^l \right|\right) dv  &\lesssim  & |\mathbf{\tilde{B}}_1|
 \chi\left(|G^l|\right)^{1/2}
 \chi\left(|\mathbf{\tilde{K}\tilde{K}\tilde{G}}|\right)^{1/2}\\
&\lesssim &|\mathbf{\tilde{B}}_1|
 \frac{\varepsilon^{1/2}}{t^{3/2-2C \sqrt{\varepsilon} }}
 \chi\left(|\mathbf{\tilde{K}\tilde{K}\tilde{G}}|
 \right)^{1/2}.
 \end{eqnarray*}
 Hence, integrating on $H_{\rho}$,
 \begin{eqnarray*}
 \int_{H_{\rho}}
 \int_v
 \left(\left|   \left(
 \tilde{B}_{\beta}^{\alpha} \tilde{K}_{\gamma}^{\kappa} +
  \tilde{K}_{\beta}^\alpha \tilde{B}_{\gamma}^{\kappa}
  \right)
  \tilde{G}_{\mu}^l
  \right|\right) dv d\mu_{H_{\rho}} &\lesssim
  &   \varepsilon^{1/2} \rho^{C\varepsilon^{1/16}-1}
 \left\Vert \left(\max\left(\dfrac{t}{\rho}, \sqrt{\rho} \right)\right)^{-1}
  \mathbf{\tilde{B}}_1\right\Vert_{L^2(H_\rho)}\\
 && \times \left(E_0[|\mathbf{\tilde{K}\tilde{K}\tilde{G}}|](\rho)
  \right)^{\frac12}\\
 &\lesssim  &   \varepsilon \rho^{C\varepsilon^{1/16}-1} \left(E_0[|\mathbf{\tilde{K}\tilde{K}\tilde{G}}|](\rho) \right)^{\frac12},
\end{eqnarray*}
where we have used that
$$
\dfrac{ \max\left(\dfrac{t}{\rho}, \sqrt{\rho} \right)}{t^{\frac32}} \leq \rho^{-1}.
$$

As a consequence $|\mathbf{\tilde{K}\tilde{K}\tilde{G}}|$ satisfies the integral inequality, for all $\rho \in [1,P]$,
\begin{align*}
 E_0[|\mathbf{\tilde{K}\tilde{K}\tilde{G}}|](\rho)  \lesssim & \int_1^{\rho} \frac{\sqrt{\varepsilon}}{\rho' }
 E_0[|\mathbf{\tilde{K}\tilde{K}\tilde{G}}|](\rho' )
 +  \varepsilon^{3/2} \rho'{}^{C\varepsilon^{1/8} -  1 }\\
 \lesssim  &  \int_1^{\rho} \frac{\sqrt{\varepsilon}}{\rho' }
 E_0[|\mathbf{\tilde{K}\tilde{K}\tilde{G}}|](\rho' ) d \rho'\\
 & + \int_1^{\rho} \varepsilon^{3/4} \rho'{}^{C\varepsilon^{1/16} -  1/2 } \cdot \varepsilon^{1/4} \rho'{}^{ 1/2 } \cdot
 \left(E_0[|\mathbf{\tilde{K}\tilde{K}\tilde{G}}|](\rho' ) \right)^{1/2} d \rho' \\
  \lesssim  &  \int_1^{\rho} \frac{\sqrt{\varepsilon}}{\rho'} E_0[|\mathbf{\tilde{K}\tilde{K}\tilde{G}}|](\rho' ) +  \varepsilon^{3/2} \rho'{}^{ C\varepsilon^{1/8} -  1 }  d \rho'
\end{align*}
Then, the Gr\"onwall inequality implies immediately, for all $\rho$ in $[1, P]$,
$$
E_0[|\mathbf{\tilde{K}\tilde{K}\tilde{G}}|](\rho) \lesssim  \varepsilon^{3/2-1/8} \rho^{C \varepsilon^{1/16}}.
    $$
 \end{proof}

One can now state the following $L^2$-estimate for $G_{inh}$:
\begin{lemma}\label{lem:G1est} $\tilde{G}_{inh}$ satisfies:
  $$
  \int_{H_\rho}\frac{t}{\rho}\left(\int_v\vert \tilde{G}_{inh} \vert\frac{dv}{v^0}\right)^2d\mu_{H_\rho} \lesssim  \varepsilon^{2+1/4} \rho^{C\varepsilon^{1/16}}.
  $$
\end{lemma}
\begin{proof} We perform a Cauchy-Schwarz estimate on the components of $G_{inh}$:
\begin{eqnarray}
\left(\int_{H_\rho}\frac{t}{\rho}\left(\int_v\ab{\tilde{G}^\al}\frac{dv}{v^0}\right)^2d\mu_{H_\rho}\right)^{\frac12} &\lesssim&
\left(\int_{H_\rho}\frac{t}{\rho}\left(\int_v\ab{\tilde{K}_{ \al}^k \tilde{G}^l_k}\frac{dv}{v^0}\right)^2d\mu_{H_\rho}\right)^{\frac12}\nonumber\\
  &\lesssim & \sum_k\int_{H_\rho}\frac{t}{\rho}\left(\int_v\ab{\tilde{G}^l_k}\frac{dv}{v^0}\right)\left(\int_v\ab{(\tilde{K}_{\al}^k)^2 \tilde{G}^l_k}\frac{dv}{v^0}\right)d\mu_{H_\rho},\label{eq:sum-est-4}
\end{eqnarray}
where the last sum over $k$ is taken of all the components of $\mathbf{\tilde{K}}_1$ and $\tilde{G}^l$ and is consequently finite.
In combination with the pointwise decay for $\tilde{G}^l$ and the estimate in Lemma \ref{lem:estbbginh}, this implies
$$
\int_{H_\rho}\frac{t}{\rho}\left(\int_v\ab{\tilde{G}^\al}\frac{dv}{v^0}\right)^2d\mu_{H_\rho} \lesssim \varepsilon^{2+1/4} \rho^{C\varepsilon^{1/16}-3}
$$
In combination with the bound on the homogeneous part above, this yields the claim. \end{proof}

An immediate consequence of Lemmata \ref{lem:ghomest} and \ref{lem:G1est} are the $L^2$-estimates for the low order:
\begin{lemma}\label{lem:G2estlow} $\tilde{G}^l$ satisfies:
  $$
  \int_{H_\rho}\frac{t}{\rho}\left(\int_v\ab{\tilde{G}^l}\frac{dv}{v^0}\right)^2d\mu_{H_\rho} \lesssim \varepsilon^{2} \rho^{C\varepsilon^{1/16}-3}
  $$
\end{lemma}
\begin{remark} Note the power in $\varepsilon$ has decreased since there is no gain in $\varepsilon$ for the homogenous part in Lemma \ref{lem:ghomest}.
\end{remark}

\subsection{$L^2$-estimates for higher order}
We perform in this section the $L^2$ estimates for terms of the highest order. The only difference between is the source term $\mathbf{B}_2 \tilde{G}^h$ that requires a specific treatment. Consider $G^l$ the solution to
  $$
\Tp G^h + \mathbf{A} G^h =  \mathbf{B}_1G^l +  \mathbf{B}_2
  \tilde{G}^h.
  $$
By linearity, $G^h$ can be written as a superposition between two vectors $G_{inh0}$ and $G_{inh}$ satisfying
$$
\Tp G_{inh0} + \mathbf{A} G_{inh0} =  \mathbf{B}_1G^l
$$
with $G_{inh0} = G^h$ on $H_1$ and
\begin{equation}\label{eq:inh2}
\Tp G_{inh}  + \mathbf{A}G_{inh} = \mathbf{B}_2 \tilde{G}^h.
\end{equation}
$G_{inh0}$ satisfies exactly the same equation as $\tilde{G}^h$, and, consequently, satisfies the same decay property as $\tilde{G}^h$:
\begin{lemma}\label{lem:Ginh0} $G_{inh0}$ satisfies:
  $$
  \int_{H_\rho}\frac{t}{\rho}\left(\int_v\ab{G_{inh0}}\frac{dv}{v^0}\right)^2d\mu_{H_\rho} \lesssim  \varepsilon^{2} \rho^{C\varepsilon^{1/16}-3}
  $$
\end{lemma}

 Hence, we only need to consider here  the problem
\begin{equation}
\Tp G_{inh}  + \mathbf{A}G_{inh} = \mathbf{B}_2 \tilde{G}^h,
\end{equation}
with vanishing initial data.

Let $\mathbf{K}_2$ be the solution to the equation
\begin{equation}
\Tp \mathbf{K}_2 + \mathbf{A} \mathbf{K}_2  - \mathbf{K}_2  \tilde{\mathbf{A}}  = \mathbf{B}_2,
\end{equation}
with vanishing initial data. We notice that
$$
\Tp ( G_{inh} - \mathbf{K}_2  \tilde{G}^h) + \mathbf{A}( G_{inh} - \mathbf{K}_2  \tilde{G}^h)  = - \mathbf{K}_2 \tilde{\mathbf{B}}_1 \tilde{G}^l.
$$
Hence, if one introduces the matrix $\mathbf{K}_3$ such that
$$
\Tp \mathbf{K}_3 + \mathbf{ A} \mathbf{K}_3  + \mathbf{K}_3  \mathbf{\hat{\tilde{A}}}  = - \mathbf{K}_2 \tilde{\mathbf{B}}_1 ,
$$
by uniqueness of solutions to the Cauchy problem, the solution to the inhomogeneous problem \eqref{eq:inh2} satisfies the identity
$$
G_{inh} = \mathbf{K}_2  \tilde{G}^h + \mathbf{K}_3  \tilde{G}^l.
$$
To perform $L^2$-estimates on these terms, we notice first that $\mathbf{K}_2$ has a small pointwise growth. Hence, by Lemma \ref{lem:G1est}, the product $\mathbf{K}_2  \tilde{G}^h$ satisfies the required $L^2$-estimates. To handle the second term, we follow the strategy as in Section \ref{sec:L2low}.

We state first the pointwise estimates for $\mathbf{K}_2$:
\begin{lemma}\label{lem:estK2}
  The solution $\mathbf{K_2}$ to the equation
$$
\Tp \mathbf{K}_2 + \mathbf{A} \mathbf{K}_2  -  \mathbf{\tilde{A}} \mathbf{K}_2  = \mathbf{B}_2
$$
   satisfies:
  $$
\vert \mathbf{K_2} \vert \lesssim \varepsilon^{1/2-1/16} \rho^{C\varepsilon^{1/16}}.
  $$
\end{lemma}
\begin{proof}
 Consider the homogeneous problem with data on the hyperboloid $H_{\rho'}$
  $$
\Tp G_{\rho'} =  0 \text{ with } U_{\rho'}|_{H_{\rho'}} =
\dfrac{-\mathbf{AK}_2 + \mathbf{K}_2 \tilde{\mathbf{A}}  + \mathbf{B}_2} {v^\rho}
  $$
  The solution to the equation
  $$
\Tp \mathbf{K}_2  + \mathbf{AK}_2 - \mathbf{K}_2 \mathbf{\tilde{A}} =  \mathbf{B}_2
  $$
is then given by the standard representation formula. Hence, plugging in the decay for $\mathbf{A}, \tilde{\mathbf{A}}$ and $\mathbf{B}_2$, one obtains:
\begin{eqnarray}
\mathbf{K_2} &=& \int_{1}^\rho U_{\rho'}(\rho)  d \rho'\\
\vert \mathbf{K_2}\vert &\lesssim & \int_1^{\rho}\dfrac{\sqrt{\varepsilon}}{s}\vert \mathbf{K_2}\vert ds +  \sqrt{\varepsilon} \rho^{-1+C\varepsilon^{1/16}}.
\end{eqnarray}
The conclusion holds by an immediate application of Gr\"onwall's lemma.
\end{proof}

We now turn our attention to the second term $\mathbf{K}_3\tilde{G}^l$. As mentioned before, we follow for this term the exact same strategy as in Section \ref{sec:L2low}. We define
$$
|\mathbf{KK\tilde{G}}| = \sum_{\alpha, \beta, \gamma, \kappa, \mu} |K_{\alpha}^{\beta}K_{\gamma}^{\kappa}\tilde{G}_{\mu}^l|,
$$
where $K_{\alpha}^{\beta}$, $K_{ \gamma}^{\kappa}$ are the components of the matrix $\mathbf{K}_3$ and $\tilde{G}_{\mu}^l$ the components of the vector $\tilde{G}^l$, and where the sum is taken over all possible combinations of two elements of $\mathbf{K}_3$ and $\tilde{G}^l$. One furthermore easily checks that each element of this sum satisfies the equation
\begin{eqnarray}
 \T_{\phi}\left(K_{\alpha}^{\beta}K_{\gamma}^{\kappa}\tilde{G}_{\mu}^l \right) &=& \left(A^\sigma_{\alpha}K_{\sigma}^{\beta} + K_{\beta}^{\sigma}\hat{\tilde{A}}_{\sigma}^\alpha \right)K_{\gamma}^{\kappa}\tilde{G}_{\mu}^l
 +K_{\alpha}^{\beta}  \left(A^\sigma_{\gamma}K_{\sigma}^{\kappa} + K_{\gamma}^{\sigma}\hat{\tilde{A}}_{\sigma}^\kappa \right)\tilde{G}_{\mu}^l\nonumber\\
 &&+ K_{\alpha}^{\beta}K_{\gamma}^{\kappa}\hat{\tilde{A}}^{\sigma}_{\mu}
 \tilde{G}_{\sigma}^l \label{eq:bgg11}\\
  &&- \left( C_{\beta}^{\alpha} K_{\gamma}^{\kappa} + K_{\beta}^\alpha C_{\gamma}^{\kappa}\right)\tilde{G}_{\mu}^l \label{eq:bgg22},
\end{eqnarray}
where $\mathbf{C}$ is the matrix $\mathbf{K}_2 \tilde{\mathbf{B}}_1$, whose components are denoted by $C^\al_\be$.
As a consequence, $|\mathbf{KK\tilde{G}}|$ satisfies an estimate of the form:
\begin{lemma}\label{lem:estbbginh2} The function $|\mathbf{KK\tilde{G}}|$ satisfies $$
  E_0[|\mathbf{KK\tilde{G}}|](\rho) \lesssim  \varepsilon^{1-1/16} \rho^{C\varepsilon^{1/16}}.
      $$
\end{lemma}
\begin{proof}
The right-hand side of Equation \eqref{eq:bgg1} can easily be estimated, using the properties of the matrices $\mathbf{A}$ and $\mathbf{\hat{\tilde{A} } }$ stated in Lemmata \ref{lem:eqgfull}, \ref{lem:eqggfull} and \ref{lem:eqgl}, by
\begin{gather*}
\int_{H_\rho} \chi\left(\left|\left(  A^\sigma_{\alpha}K_{\sigma}^{\beta} + K_{\beta}^{\sigma}\hat{\tilde{A}}_{\sigma}^\alpha \right)K_{\gamma}^{\kappa}\tilde{G}_{\mu}^l
 +K_{\alpha}^{\beta}  \left(A^\sigma_{\gamma}K_{\sigma}^{\kappa} + K_{\gamma}^{\sigma}\hat{\tilde{A}}_{\sigma}^\kappa \right)\tilde{G}_{\mu}^l+ K_{\alpha}^{\beta}K_{\gamma}^{\kappa}\hat{\tilde{A}}^{\sigma}_{\mu}
 \tilde{G}_{\sigma}^l  \right|\right) d\mu_{H_{\rho }}
 \\ \lesssim
\frac{\sqrt{\varepsilon}}{\rho} E_0[|\mathbf{KK\tilde{G}}|](\rho).
\end{gather*}

Furthermore, by the Cauchy-Schwarz inequality, as well as the property of the matrix $\tilde{\mathbf{B}}_1$ stated in Lemma \ref{lem:eqgfull}, and the pointwise estimate on $\mathbf{K}_2$ stated in Lemma \ref{lem:estK2}, one gets, when estimating Equation \eqref{eq:bgg2} :
\begin{eqnarray*}
 \int_v \left(\left| \left( C_{\beta}^{\alpha} K_{\gamma}^{\kappa} + K_{\beta}^\alpha C_{\gamma}^{\kappa}\right)\tilde{G}_{\mu}^l \right|\right) dv  &\lesssim  & |\mathbf{B}|\chi\left(|G^l|\right)^{1/2}\chi\left(|\mathbf{KK\tilde{G}}| \right)^{1/2}\\
 &\lesssim  &  |\tilde{\mathbf{B}}_1| \frac{\varepsilon^{1/2}}{t^{3/2-C\varepsilon^{1/16} }}\chi\left(|\mathbf{KK\tilde{G}}| \right)^{1/2}\\
 \int_{H_{\rho}} \int_v \left(\left|   \left( C_{\beta}^{\alpha} K_{\gamma}^{\kappa} + K_{\beta}^\alpha C_{\gamma}^{\kappa}\right)\tilde{G}_{\mu}^l \right|\right) dv d\mu_{H_{\rho}} &\lesssim  &   \sqrt{\varepsilon} \rho^{C\varepsilon^{1/16} -1}
 \Vert \left(\max\left(\dfrac{t}{\rho}, \sqrt{\rho} \right)\right)^{-1} \tilde{\mathbf{B}}_1\Vert_{L^2(H_\rho)}\\
 && \times \left(E_0[|\mathbf{KK\tilde{G}}|](\rho) \right)^{\frac12}\\
 &\lesssim  &   \varepsilon \rho^{C\varepsilon^{1/16}-1} \left(E_0[|\mathbf{KK\tilde{G}}|](\rho) \right)^{\frac12},
\end{eqnarray*}
where we have used that
$$
\dfrac{\max\left(\dfrac{t}{\rho}, \sqrt{\rho} \right)}{t^\frac32} \leq \rho^{-1},
$$
and where we also have factored with the relevant powers of $\varepsilon$.

As a consequence $|\mathbf{KK\tilde{G}}|$ satisfies the integral inequality, for all $\rho$,
\begin{align*}
 E_0[|\mathbf{KK\tilde{G}}|](\rho)  \lesssim  &  \int_1^{\rho} \frac{\sqrt{\varepsilon}}{\rho' }
 E_0[|\mathbf{KK\tilde{G}}|](\rho' )
 +  \varepsilon \rho'{}^{C\varepsilon^{1/16} -  1 }
 \left(E_0[|\mathbf{KK\tilde{G}}|](\rho' ) \right)^{1/2} d \rho' \\
  \lesssim  &  \int_1^{\rho} \frac{\sqrt{\varepsilon}}{\rho'} E_0[|\mathbf{KK\tilde{G}}|](\rho' ) +  \varepsilon^{3/2} \rho'{}^{ C\varepsilon^{1/8} -  1 }  d \rho'
\end{align*}
Then, the Gr\"onwall inequality implies immediately, for all $\rho$,
$$
E_0[|\mathbf{KK\tilde{G}}|](\rho) \lesssim  \varepsilon^{3/2-1/8} \rho^{C\varepsilon^{1/16}}.
    $$
 \end{proof}

One can now state the following $L^2$-estimate for $G_{inh}$:
\begin{lemma}\label{lem:G2est} $G_{inh}$ satisfies:
  $$
  \int_{H_\rho}\frac{t}{\rho}\left(\int_v\vert G_{inh} \vert\frac{dv}{v^0}\right)^2d\mu_{H_\rho} \lesssim \varepsilon^{2+1/4}. \rho^{C\varepsilon^{1/16}-3}
  $$
\end{lemma}
\begin{proof} The proof of this lemma is a direct consequence of the $L^2$ estimates for the low order terms of Lemma \ref{lem:G1est}, the pointwise estimates for $\mathbf{K}_2$ of Lemma \ref{lem:estK2}, as well as, for $\mathbf{K}_3\tilde{G}^l$ the same argument as in the proof of Lemma \ref{lem:G1est}, based, this time, on the energy estimate of Lemma \ref{lem:estbbginh2}.
\end{proof}

\subsection{Statement of the $L^2$-estimates for the distribution function}

 Finally, gathering Lemmata \ref{lem:Ginh0} and \ref{lem:G2est}, one obtains the following estimates:
 \begin{proposition}\label{pro:L2estfull} The following estimate holds:
   $$
     \int_{H_\rho}\dfrac{t}{\rho}\left(
     \int_{v}
     \left| \hat {\zz}^{C}
  \underline{q}_{\geq 0}^{A,B,K}(\Phi)
     \cdot L_{U,V}^\pi f\right|
      d\mu(v) \right) ^ 2d \mu_{H_\rho}
      \lesssim \varepsilon^{2} \rho^{C\varepsilon^{1/16}-3}
   $$
   where $A, B, K$ are non-negative integers and $C, U, V$ multi-indices such that
  \eq{
   (\ast)\left\{
   \begin{array}{l}
   N-2\leq |U|+|V|\leq N\\
   A+B+|U|+|V|\leq N\\
   \pi \mbox{ any permutation of }  A+ B \mbox{ elements}\\
   \ab{C}= N+3-(A+|U|).
   \end{array}\right.
}
 \end{proposition}

\section{Energy estimates for the wave equation on hyperboloids}\label{sec:btwave}

We start by writing the energy identity for the wave equation.
\begin{lemma}\label{lem:inhwavehyp}
Let $\psi$ be defined in hyperboloidal time for all $\rho \in [1,P]$ and assume that $\psi$ solves $\square \psi = h$.
Let $\rho \in [1, P]$.
Then,
$$
\int_{H_\rho} T[\psi](\partial_t, \nu_\rho) d\mu_{H_\rho} =   \int_{H_1} T[\psi](\partial_t, \nu_\rho) d\mu_{H_1} + \int_1^\rho \int_{H_{\rho'}} \left(\partial_t \psi \right)(\rho') h(\rho')  d\mu_{H_{\rho'}} d\rho'.
$$
\end{lemma}

To close the energy estimates for $Z^\alpha(\phi)$, we need the right-hand side of \eqref{eq:comwave} to decay. Since for $|\alpha| \le N-3$, the required decay follows from our Klainerman-Sobolev inequality \eqref{ineq:ksmsv} as well as the bootstrap assumptions \eqref{eq:bs1}-\eqref{eq:bs2}, we have the following lemma

\begin{lemma}\label{lem:energymassivewavelow}  \label{lem:mwavelow}
The following inequality holds, for all $\rho \in [1,P]$:
$$
\mathscr{E}_{N-3}[\phi](\rho) \leq \varepsilon (1+C\varepsilon^{1/2}),
$$
where $C$ is a constant depending only on $N$. In particular, for $\varepsilon$ small enough, for all $\rho \in [1,P]$:
$$
\mathscr{E}_{N-3}[\phi](\rho) \leq \frac{3}{2} \varepsilon.
$$
\end{lemma}
\begin{proof}
We first apply Lemma 8.1 for $\psi=Z^\alpha(\phi)$ with $|\alpha| \le N-3$. Writing $\square Z^\alpha(\phi)=h^\alpha$, we obtain

\begin{align*}
\int_{H_\rho} T[Z^\alpha\phi]&(\partial_t, \nu_\rho) d\mu_{H_\rho} \\
  &= \int_{H_1} T[Z^\alpha\phi](\partial_t, \nu_\rho) d\mu_{H_1} + \int_1^\rho \int_{H_{\rho'}} \left(\partial_t Z^\alpha\phi \right)(\rho') h^\alpha(\rho')  d\mu_{H_{\rho'}} d\rho'  \\
&\le \int_{H_1} T[Z^\alpha\phi](\partial_t, \nu_\rho) d\mu_{H_1} + \int_1^\rho  \mathscr{E}_{N-3}^{1/2}(\rho')\left( \int_{H_{\rho'}}\frac{t}{\rho'} (h^\alpha)^2 (\rho')  d\mu_{H_{\rho'}}\right)^{1/2} d\rho', \\
&\le \int_{H_1} T[Z^\alpha\phi](\partial_t, \nu_\rho) d\mu_{H_1} +\varepsilon^{1/2} \int_1^\rho \left(\int_{H_{\rho'}} \frac{t}{\rho'} (h^\alpha)^2 (\rho')  d\mu_{H_{\rho'}}\right)^{1/2} d\rho',
\end{align*}
using the Cauchy-Schwarz inequality and the bootstrap assumption.

Using that $\mathscr{E}_{N-3}(\rho=1) \le \varepsilon$, we only need to show that
$$
\int_1^\rho \left(\int_{H_{\rho'}} \frac{t}{\rho'} (h^\alpha)^2 (\rho')  d\mu_{H_{\rho'}} \right)^{1/2}d\rho' \le C \varepsilon.
$$
From the commutation formula \eqref{eq:comwave}, we have
$$
h^\alpha = \sum_i \int_v F_i \frac{dv}{v^0},
$$
where the $F_i$ are, modulo a multiplication by good symbols, of the form

$$
\left( \frac{\zz}{ t} \right)^{r_1} p^{d_1, k_1} ( \ee{}( \Phi) ) \frac{q^{2d_2+r, k_2}({\Phi})}{t^{d_2+r_2}} Q^{d_3+r_3,k_3}\left(  {\phi} \right)[ \mathbf Y^{\delta}(f) ]
$$
where $r_1+r_2+r_3= r$ , $r +d_1+d_2+d_3 \le N-3$,  $|k|+|\delta|+d_1 \le N-3$, $r + |\delta|-|\delta(\ee{})| + |k|-|k(\ee{})| \le N-3$, $k_3(\ee{}) \ge r_3$, with $|k|:=k_1+k_2+k_3$ and with $k(\ee{})$, $k_3(\ee{})$, $\delta(\ee{})$ denoting the total number of $\ee{}$ in the corresponding terms.

First note that since $|k| \le N-2$, we have pointwise estimates on the $Q$ form

\begin{eqnarray*}
|Q^{d_3+r_3,k_3}\left(  {\phi} \right)| &\lesssim& \varepsilon^{(d_3+r_3)/2}\frac{(1+u)^{d_3/2}}{t^{d_3+r_3}(1+u)^{r_3/2}} \\
& \lesssim &\varepsilon^{(d_3+r_3)/2}\frac{1}{r_3},
\end{eqnarray*}
where we have dropped some extra decay in $t$ and $u$  since they are unnecessary for our applications here.

Using the estimate, we have
\begin{eqnarray*}
|F_i|  &\lesssim& \varepsilon^{(d_3+r_3)/2} \left( \frac{|\zz|}{t^{1/2}} \right)^{r_1} p^{d_1, k_1} ( \ee{}( \Phi) ) \frac{q^{2d_2+r, k_2}({\Phi})}{t^{d_2+r_3+r_2+r_1/2}}  |\mathbf Y^\delta (f)| \\
&\lesssim& \varepsilon^{(d_3+r_3)/2} p^{d_1, k_1} ( \ee{}( \Phi) ) \frac{q^{2d_2+r, k_2}({\Phi})}{\rho^{d_2+r/2}} \left( \frac{|\zz|}{t^{1/2}} \right)^{r_1}  |\mathbf Y^\delta (f)|.
\end{eqnarray*}

Let us write $p$ for $p^{d_1, k_1} ( \ee{}( \Phi) )$ and $q$ for $q^{2d_2+r, k_2}({\Phi})$.

using the Cauchy-Schwarz inequality in the variable $v$, it follows that
\begin{align*}
\left| \int_v F_i \right. &\left.\frac{dv}{v^0} \right| \\
&\lesssim \varepsilon^{(d_3+r_3)/2} \left( \int_v \frac{p^2 q^2}{\rho^{2d_2 +r}} \left( \frac{|\zz|}{t^{1/2}} \right)^{r_1}  |\mathbf Y^\delta (f)|  \frac{dv}{v^0} \right)^{1/2} \left( \int_v \left( \frac{|\zz|}{t^{1/2}} \right)^{r_1} |\mathbf Y^\delta (f)|  \frac{dv}{v^0} \right)^{1/2}
\end{align*}

So that
\begin{align*}
\int_{H_\rho'} \frac{t}{\rho'} &(h^\alpha)^2  d\mu_{\rho'}  \\
& \lesssim \varepsilon^{(d_3+r_3)/2}\times\int_{H_\rho'}\frac{t}{\rho'} \left(\int_v \frac{p^2 q^2}{\rho^{2d_2 +r}}  \left( \frac{|\zz|}{ t^{1/2}} \right)^{r_1}|\mathbf Y^\delta (f)| \frac{dv}{v^0}\right)\left( \int_v \left(\frac{|\zz|}{ t^{1/2}} \right)^{r_1}\mathbf Y^\delta (f)| \frac{dv}{v^0}\right)d\mu_{\rho'} \\
& \lesssim \varepsilon^{1+(d_3+r_3)/2}(\rho')^{C\varepsilon^{1/2}-3}\int_{H_\rho'}\left( \int_v \frac{p^2 q^2}{\rho^{2d_2 +r}}  \left(\frac{|\zz|}{t^{1/2}} \right)^{r_1}  |\mathbf Y^\delta (f)| \frac{dv}{v^0}\right)d\mu_{\rho'},
\end{align*}
where we have used the Klainerman-Sobolev inequality to estimate the second $v$ integral.

Now the second integral is bounded by $\varepsilon \rho^{C \varepsilon^{1/16}}$ in view of the combined estimates. Thus, all together, we obtain
$$
\left( \int_{H_\rho'} \frac{t}{\rho'} (h^\alpha)^2 d\mu_{\rho'}\right)^{1/2} \lesssim \varepsilon \rho^{-3/2+ C \varepsilon^{1/16}}
$$
and since the decay rates is integrable in $\rho$, the result follows.

Since the decay is sufficient, the result follows if we can prove a bound on
\begin{align*}
\int_{H_\rho'}\left( \int_v \left(\frac{|\zz|}{v^0 t^{1/2}} \right)^{r_1}|\mathbf Y^{k'} \ee{}( \Phi)|^2 \right.&\left.|\mathbf Y^\delta (f)| \frac{dv}{v^0}\right)d\mu_{\rho'} \\
&\lesssim  \int_{H_\rho'} \chi\left(  \left(\frac{|\zz|}{v^0 t^{1/2}} \right)^{r_1}|\mathbf Y^{k'} \ee{}( \Phi)|^2 |\mathbf Y^\delta (f)| \right)d\mu_{\rho'}
\end{align*}

For this, we consider the transport equation satisfied by $\left(\frac{\zz}{v^0 t^{1/2}} \right)^{r_1}\mathbf Y^{k'} \ee{}( \Phi)^2 \mathbf Y^\delta (f)$
\begin{align} \label{eq:horw1}
\Tp\left( \left(\frac{\zz}{v^0 t^{1/2}} \right)^{r_1}\mathbf Y^{k'} \ee{}( \Phi)^2 \mathbf Y^\delta (f) \right)= &\, \Tp  \left( \left(\frac{\zz}{v^0 t^{1/2}} \right)^{r_1} \mathbf Y^\delta (f) \right)\left[\mathbf Y^k\ee{}( \Phi)\right]^2 \\
&+ \Tp \left( \mathbf Y^k\ee{}( \Phi)^2 \right) \left(  \left(\frac{\zz}{v^0 t^{1/2}} \right)^{r_1} \mathbf Y^\delta (f) \right)\nonumber
\end{align}
The first term on the RHS can be estimated using a Gr\"onwall argument. More precisely, using the commutation formula \eqref{cor-comm-f}, the first term can be estimated by
$$\varepsilon/\rho \left|\frac{\zz}{v^0 t^{1/2}} \right|^{r_1'} \mathbf{Y}^{\delta'}(f)\left[\mathbf Y^k\ee{}( \Phi)\right]^2,
$$
so that it can indeed be eventually absorbed using a Gr\"onwall argument.

The error coming from the second term in \eqref{eq:horw1} is
$$
2\mathbf Y^k\ee{}( \Phi) \Tp \left( \mathbf Y^k\ee{}( \Phi) \right) \left(  \left(\frac{\zz}{v^0 t^{1/2}} \right)^{r_1} \mathbf Y^\delta (f) \right).
$$
Using the commutation formula \eqref{cor-comm-f}, we can either estimate the $p$ form containing the $\phi$ terms pointwise and obtain terms similar to the first, i.e. giving only a small growth, or are facing terms such as
\begin{align}
\mathbf Y^k\ee{}(\Phi) \frac{t}{v^0} \partial^2 Z^{k+1}( \phi) &\left(  \left(\frac{\zz}{v^0 t^{1/2}} \right)^{r_1} \mathbf Y^\delta (f) \right)\\
= &
\mathbf Y^k\ee{}(\Phi) \frac{t}{(1+u) v^0} \partial Z^{k+2}( \phi) \left(  \left(\frac{\zz}{v^0 t^{1/2}} \right)^{r_1} \mathbf Y^\delta (f) \right),\nonumber
\end{align}
which are integrated in $v$ and in spacetime. To estimate them, we only have access to $L^2$ bounds on $\partial Z^{k+2}( \phi)$, so we start by applying Cauchy-Schwarz,
\begin{gather*}
\int_{H_\rho} \int_v \mathbf Y^k\ee{}(\Phi) \frac{t}{(1+u) v^0} \partial Z^{k+2}( \phi) \left(  \left(\frac{\zz}{v^0 t^{1/2}} \right)^{r_1} \mathbf Y^\delta (f) \right)d\mu(v) d\mu_\rho\\
 \lesssim  \left(\int_{H_\rho} \frac{\rho}{t} \left[\partial Z^{k+2}( \phi)\right]^2 d\mu_\rho\right)^{1/2} \cdot \left( \int_{H_\rho}\frac{t}{\rho} \left(  \int_v \mathbf Y^k\ee{}(\Phi) \frac{t}{(1+u) v^0} \left(  \left(\frac{\zz}{v^0 t^{1/2}} \right)^{r_1} \mathbf Y^\delta (f)d\mu(v) \right)^2  \right) \right)^{1/2}
\end{gather*}
The first term can be estimated by $\mathscr{E}_N$ giving only $\rho^{\varepsilon^{1/2}}$ growth. For the second term, we again apply Cauchy-Schwarz in $v$ this times, to obtain
\begin{eqnarray*}
&& \left( \int_{H_\rho}\frac{t}{\rho} \left(  \int_v \mathbf Y^k\ee{}(\Phi) \frac{t}{(1+u) (v^0)} \left(  \left(\frac{\zz}{v^0 t^{1/2}} \right)^{r_1} \mathbf Y^\delta (f)d\mu(v) \right)^2  \right) d\mu_\rho\right)^{1/2}\\
&& \lesssim \left[ \int_{H_\rho} \int_v \frac{t}{\rho v^0}\left[\mathbf Y^k \ee{} \Phi\right]^2 \left(\frac{\zz}{v^0 t^{1/2}} \right)^{r_1}\mathbf Y^\delta (f) dv. \int_{v}\frac{t^2}{(1+u)^2}\left(\frac{\zz}{v^0 t^{1/2}} \right)^{r_1}(\Phi)\mathbf Y^\delta (f)dv d\mu_{H_\rho} \right]^{1/2} \\
&& \lesssim \rho^{-1} \left[ \int_{H_\rho} \int_v v^\rho \left|\mathbf Y^k \ee{}(\Phi) \left(\frac{\zz}{v^0 t^{1/2}} \right)^{r_1}\mathbf Y^\delta (f) \right| dv d\mu_{H_\rho} \right]^{1/2},
\end{eqnarray*}
where we have used the Klainerman-Sobolev inequality to estimate the last term in the last line. It follows that this term can be integrated in $\rho$ which concludes the proof of the lemma.
\end{proof}

Replacing the $L^\infty$ decay estimates on the distribution function by $L^2$ decay estimates, the above estimates still hold up to the order $N-1$.

\begin{lemma} \label{lem:mwavenh}
Let $0<\gamma <1$.  Assume that for all multi-indices $\alpha$ such that  $N-3 +1\le |\al| \le N-1$, the following $L^2_x$ decay estimate holds
\begin{equation}\label{ass:l2x}
 \int_{H_\rho} \frac{t}{\rho}\left(\int_{v}\left| \frac{\zz}{t^{1/2}} \right| ^{r_1}  |\mathbf Y^{\al} f| \frac{dv}{v^0}\right)^{2} d\mu_{H_\rho} \lesssim \varepsilon^{2} \rho^{\gamma-3},
\end{equation}
 for $|\alpha|-|\alpha(\ee{})| + r_1 \le N-1$.
Then, the following inequality holds, for all $\rho \in [1,P]$:
$$
\mathscr{E}_{N-1}[\phi](\rho) \leq \varepsilon (1+C\varepsilon^{1/16}),
$$
where $C$ is a constant depending only on $N$. In particular, for $\varepsilon$ small enough, for all $\rho \in [1,P]$:
$$
\mathscr{E}_N[\phi](\rho) \leq \frac{3}{2} \varepsilon.
$$
\end{lemma}
\begin{proof}
We use again the commutation formula \eqref{eq:comwave}. Using the same notations as above, we need to bound,
$$
\int_1^\rho \left(\int_{H_{\rho'}} \frac{t}{\rho'} (h^\alpha)^2 (\rho')  d\mu_{H_{\rho'}} \right)^{1/2}d\rho' \le C \varepsilon.
$$

where
$$
h^\alpha = \sum_i \int_v F_i \frac{dv}{v^0},
$$
where the $F_i$ are, modulo a multiplication by good symbols, of the form

$$
\left( \frac{\zz}{t} \right)^{r_1} p^{d_1, k_1} ( \ee{}( \Phi) ) \frac{q^{2d_2+r, k_2}({\Phi})}{t^{d_2+r_2}} Q^{d_3+r_3,k_3}\left(  {\phi} \right)[ \mathbf Y^{\delta}(f) ]
$$
where $r_1+r_2+r_3= r$ , $r +d_1+d_2+d_3 \le N-1$,  $|k|+|\delta|+d_1 \le N-1$, $r + |\delta|-|\delta(\ee{})| + |k|-|k(\ee{})| \le N-1$, $k_3(\ee{}) \ge r_3$, with $|k|:=k_1+k_2+k_3$ and with $k(\ee{})$, $k_3(\ee{})$, $\delta(\ee{})$ denoting the total number of $\ee{}$ in the corresponding terms.

If $|\delta| \le N-3$, then we have access to pointwise estimates as in the previous lemma on $\int_v \left(\frac{|\zz|}{ t^{1/2}} \right)^{r_1}\mathbf Y^\delta (f)| \frac{dv}{v^0}$, so we can just estimates everything as in the previous lemma.

If $|\delta|=N-1$ or $|\delta|=N-2$, then we can estimate the $p$ and $q$ form pointwise. Thus, we have,

\begin{eqnarray*}
|F_i| &\lesssim& \left( \frac{|\zz|}{v^0 t} \right)^{r_1} \varepsilon^{d_1/2} \rho^{d_1C\varepsilon^{1/2}} \varepsilon^{d_2+r/2} \frac{\rho^{d_2+r/2}}{t^{d_2+r_2}}\frac{(1+u)^{d_3/2}}{t^{d_3+r_3}(1+u)^{r_3/2}} |\mathbf Y^\delta (f)| \\
&\lesssim&     \rho^{C \varepsilon^{1/2}}\varepsilon^{d_1/2+d_2+r/2} \frac{1}{t^{r_2/2+d_3/2+r_3/2}(1+u)^{r_3/2}}\left( \frac{|\zz|}{v^0 t^{1/2}} \right)^{r_1}|\mathbf Y^\delta (f)| \\
&\lesssim& \rho^{C \varepsilon^{1/2}}\left( \frac{|\zz|}{t^{1/2}} \right)^{r_1}|\mathbf Y^\delta (f)|.
\end{eqnarray*}
The result then follows directly from the $L^2$ decay estimates \eqref{ass:l2x}.

\end{proof}

Finally, at top order, we repeat the previous proof, but in this case, we use the modified commutator formula of Lemma \ref{lem:topw}. We obtain in this case,
\begin{lemma}\label{es:topw}
Assume that for all multi-indices $\alpha$ such that  $|\alpha|=N$, the following $L^2_x$ decay estimate holds
\begin{equation}
 \int_{H_\rho} \frac{t}{\rho}\left(\int_{v}\left| \frac{\zz}{ t^{1/2}} \right| ^{r_1}  |\mathbf Y^{\al} f| \frac{dv}{v^0}\right)^{2} d\mu_{H_\rho} \lesssim \varepsilon^{2} \rho^{\tilde{C} \varepsilon^{1/16}-3},
\end{equation}
 for $|\alpha|-|\alpha(\ee{})| + r_1 \le N$, and $\tilde{C}$ is a constant.

Then, the following inequality holds, for all $\rho \in [1,P]$:
$$
\mathscr{E}_{N}[\phi](\rho) \leq \varepsilon(1+ D \varepsilon) \rho^{ C\varepsilon^{1/16}}\
$$
where $D$ is a constant depending only on $N$ and the constant $C$ is the constant from the bootstrap estimates.
\end{lemma}
\begin{remark} Note that the assumption on the $L^2$-estimates is actually proven in Section \ref{sec:L2wave}, see Proposition \ref{pro:L2estfull}.
\end{remark}
\begin{proof}
The proof is similar to the previous lemma. We consider the three type of terms appearing in Lemma \ref{lem:topw}.

For type $1$ and $2$, the overall $L^2_{H_\rho}$ decay expected is $1/\rho^{-3/2-C\varepsilon^{1/16}}$ (in particular integrable in $\rho$), since we recall that these terms do not have an extra power of $\Phi$. In fact, all these terms can be estimates as in the previous lemma, with the exception of terms arising from product of the form (neglecting the extra $t$, $\zz$ weights here)

$$
\mathbf  Y^2( \Phi) Y^{N-2}(f)\quad \mathrm{and\,\, } \mathbf Y( \ee{} \Phi) \mathbf Y^{N-2}(f),
$$
because we do not have access to pointwise estimates on $\mathbf Y^2( \Phi)$ or $\int_v |\mathbf Y^{N-2}f| dv$.

Consider thus the terms arising from $\mathbf Y^2( \Phi) \mathbf Y^{N-2}(f)$, the other case being similar.

By H\"older, we have

$$
|| \mathbf Y \Phi \mathbf  Y^{N-2}(f) ||_{L^2_x L^1_v} \le ||  ( \mathbf Y \phi )^4 \mathbf{Y}^{N-2} f ||^{1/4}_{L^1_x L^1_v} || \mathbf Y ^{N-2} f ||^{3/4}_{L^3_xL^1_v}
$$

Thus, the required decay follows from the Klainerman-Sobolev inequalitiy of Proposition \ref{KSL3}. In particular, these terms do not lead to any growth of $\mathscr{E}_{N}[\phi]$.

For the third term, first note that when $\alpha = N-1$ or $\alpha=N-2$, then we can estimate the $p$ and $q$ form pointwise. We can then conclude as above, with the difference that the $q$ form always has an extra $\Phi$ coefficient. Thus, these terms are of the form

$$
\Phi \mathbf{Y}^N(f)
$$

 which cannot always be absorbed. More precisely, the worse terms are of the form
$$
\int_v \mathbf{Y}^{N-1}( \Phi) \ee{}(  f ) \frac{dv}{v^0}.
$$
and can be estimated similarly as in the last case in the previous lemma, with an extra loss of $\rho^{1/2}$ (since we have a priori no $\ee{}$ hitting $\Phi$). Note that, here, it is important to use Equation \eqref{eq:bslf} to estimate the terms containing $f$ in the worse case of the previous expression, as otherwise the growth rate could increase.

This leads to
$$
\mathscr{E}_N[ \phi] (\rho) \le \varepsilon + \int^\rho_1 \frac{\varepsilon^{3/2}}{ (\rho')^{1- C/2 \varepsilon^{1/16}}} \mathscr{E}_N^{1/2}(\rho') d \rho'.
$$
The lemma follows when using the boostrap assumption \eqref{eq:bs3}.

\end{proof}

\section{Klainerman-Sobolev inequality with modified vector fields} \label{sec:kse}

\begin{proposition}
Under the bootstrap assumptions, we have the Klainerman-Sobolev inequality, for any sufficiently regular distribution function $g$,
\begin{equation} \label{ineq:ksmsv}
\int_v | g |(t,x,v) d\mu(v)  \lesssim \frac{1}{t^3}\left(1+\varepsilon^{1/2}\ln \rho \right) \sum_{ | \alpha | \le 3}\int_{H_\rho} \chi \left( \left| \frac{\zz}{v^0 t^{1/2}} \right|^3 | \mathbf Y^\alpha f | \right) d\mu_\rho.
\end{equation}
\end{proposition}
\begin{remark}
In fact, a slightly better estimate holds, in the sense that we can replace $\left| \frac{\zz}{v^0} \right|^3 | \mathbf Y^\alpha f |$ by $\left| \frac{\zz}{v^0} \right|^\beta | \mathbf Y^\alpha f |$, where $|\beta|+ |\alpha|- \ee{}[\alpha] \le 3$, where $\ee{}[\alpha]$ denotes the number of generalized translations contained in  $Y^\alpha$.
\end{remark}
\begin{proof}
As in the non-modified case, we use that
$$
\int_{H_\rho} \chi( |f|) (t,x) d \mu_\rho \ge \int_{H_\rho} m^2 \frac{t}{2 \rho} \int_{v \in \mathbb{R}^n} |f|(t,x,v) \frac{dv}{|v^0|} d\mu_\rho.
$$
Let $(t,x)$ be fixed in $J^+(H_1)$ and defined the function $\psi$ in the $(y^\alpha)$ system of coordinates as
$$
\psi(y^0, y^i) \equiv \int_{v} |f|(y^0, x^j + t y^j) d\mu(v).
$$
We follow the non-modified case (see \cite{fjs:vfm}) and first apply a 1-d Sobolev inequality in the variable $y^1$:
$$
\int_{v \in \mathbb{R}^n} |f|(y^0, x^j) d \mu_m = | \psi(y^0, 0) | \lesssim \int_{y^1} \frac{1}{(8n)^{1/2}} \left[ \left| \partial_{y^1} \psi \right| (y^0, y^1, 0, .., 0) + | \psi|(y^0, y^1, 0, .., 0) \right] dy^1.$$

As before, we can estimate $\partial_{y^1} \psi$ uniformly in the region of integration by $\Omega_{01} \psi$, leading to
$$
\left| \partial_{y^1} \psi \right|(y^0,y^1,0,..,0) \lesssim \left|\int_v \Omega_{01}|f|(y^0,x^1+ty^1,x^2,..,x^n,v)  d\mu(v)\right|.$$
Now, we have
\begin{eqnarray*}
&&\int_v \Omega_{01}|f|(y^0,x^1+ty^1,x^2,..,x^n,v)  d\mu(v)\\
&=& \int_v \left( \Omega_{01}+v^0 \partial_{v^1}+ \Phi^j_1 \mathbf X_j  + \Omega_{0i}(\phi)v^i \partial_{v^i} \right)|f|(y^0,x^1+ty^1,x^2,..,x^n,v)  d\mu(v) \\
&& -  \int_v \left( \Phi^j_1 \mathbf X_j + \Omega_{0i}(\phi)v^i \partial_{v^i} \right)|f|(y^0,x^1+ty^1,x^2,..,x^n,v)  d\mu(v) \\
&=& \int_v \left(\mathbf{Y}_1 - \Phi^j_1 \mathbf X_j  - \Omega_{0i}v^i \partial_{v^i}  \right)|f|(y^0,x^1+ty^1,x^2,..,x^n,v)  d\mu(v)
\end{eqnarray*}
The last term can easily be integrated by parts in $v$

$$ \int_v  \Omega_{0i}v^i \partial_{v^i} |f|(y^0,x^1+ty^1,x^2,..,x^n,v)  d\mu(v) = \int_v \Omega_{0i} (\phi) s(v) | f | d \mu_m,$$
where $s(v)$ is a good symbol in $v$. This term then has a stronger decay and can be safely ignored.

The first term on the right-hand side is easily estimated
$$
\left| \int_v \mathbf Y_1(f)  d\mu(v) \right| \lesssim \int_v \left| \mathbf Y_1(f)  \right| d\mu(v)
$$
and this is controlled by $\chi(|\mathbf Y(f)|)$ after integration on $H_\rho$.
For the second term, we first use again that $\mathbf X_i= \frac{\mathbf Z_i}{t} + \frac{\zz_i}{v^0 t} \ee 0$ and we then force again the introduction of our modified vector fields
\begin{align*}
\int_v \Phi^j_1 \mathbf X_j  &|f|(y^0,x^1+ty^1,x^2,x^3,v)  d\mu(v) \\
=& \int_v \Phi^j_1 \left( \frac{\mathbf Z_j}{t} + \frac{\zz_j}{v^0 t} \ee 0 \right)|f|(y^0,x^1+ty^1,x^2,x^3,v)  d\mu(v) \\
=&\int_v \Phi^j_1 \frac{1}{t(y)} \left( \mathbf Z_j +v^0 \partial_{v^j} + \Phi^k_j  \mathbf X_k \right)|f|(y^0,x^1+ty^1,x^2,x^3,v)  d\mu(v) \\
&\hbox{}- \int_v \Phi^j_1 \frac{1}{t(y)}\left( v^0 \partial_{v^j} + \Phi^k_j \mathbf X_k \right)|f|(y^0,x^1+ty^1,x^2,x^3,v)  d\mu(v) \\
&\hbox{} +\int_v  \Phi^j_1 \frac{\zz_j}{v^0 t} \ee 0 |f|(y^0,x^1+ty^1,x^2,..,x^n,v), \\
=&\int_v \Phi^j_1 \frac{1}{t(y)} \left( \mathbf Y_j \right)|f|(y^0,x^1+ty^1,x^2,x^3,v)  d\mu(v) \\
&\hbox{}- \int_v \Phi^j_1 \frac{1}{t(y)}\left( v^0 \partial_{v^j} + \Phi^k_j \mathbf X_k \right)|f|(y^0,x^1+ty^1,x^2,x^3,v)  d\mu(v) \\
&\hbox{} +\int_v  \Phi^j_1 \frac{\zz_j}{v^0 t} \ee 0 |f|(y^0,x^1+ty^1,x^2,x^3,v),
\end{align*}
where the $t(y)$ above are those of the points $(x^j+t(y) y^j)$.

Now from the bootstrap assumptions, we have $\frac{|\Phi|}{t} \lesssim \varepsilon^{1/2} \frac{\rho^{1/2}}{t}$ and thus the first term can be easily estimated.

For the last term, we have :
\begin{align*}
\int_v  \Phi^j_1 \frac{\zz_j}{v^0 t} \ee 0 |f|&(y^0,x^1+ty^1,x^2,x^3,v) d\mu(v)\\
 &\lesssim \varepsilon^{1/2} \frac{\rho^{1/2}}{t^{1/2}} \int_v  \frac{|\zz|}{v^0 t^{1/2}}   |\ee 0 f|(y^0,x^1+ty^1,x^2,x^3,v) d\mu(v) \\
&\lesssim   \varepsilon^{1/2} \int_v  \frac{|\zz|}{v^0 t^{1/2}}   |\ee 0f|(y^0,x^1+ty^1,x^2,x^3,v) d\mu(v),
\end{align*}
which can thus be controlled after integration on $H_\rho$ integrals by $\chi(|\frac{\zz}{v^0 t^{1/2}}\ee 0 f|)$.

For the second term, we note first $$
|\Phi^j_1 \frac{1}{t(y)} \Phi^k_j\mathbf X_j (f)| \lesssim \frac{\rho}{t(y)} |\ee{}(f)| \lesssim | \ee{}(f)|,
$$
which can therefore be estimated as above.
We are left with the term containing derivatives in $v$, which we integrate by parts
$$
-\int_v \Phi^j_1 \frac{1}{t(y)}v^0 \partial_{v^j}|f|(y^0,x^1+ty^1,x^2,..,x^n,v)  dv= \int_v v^0 \partial_{v^j} \Phi^j_1 \frac{1}{t(y)} |f|(y^0,x^1+ty^1,x^2,..,x^n,v)  dv.
$$
Rewriting once more $v^0 \partial_{v^j} \Phi^j_1$ as
$$v^0 \partial_{v^j} \Phi^j_1= \mathbf{Y}_j \Phi^j_1 - \mathbf Z^j\Phi^j_1- \Phi^k_j \mathbf X_k \Phi^j_1,
$$
the contribution of the first and last term are easily seen to be of lower order. For the second term, we only need to use that
$$\left|\frac{\mathbf Z \Phi}{t}\right| \lesssim \ee{}(\Phi) \lesssim \varepsilon^{1/2}  \ln \rho.$$

The rest of the proof of then follows by repeating the above steps for each variable.
\end{proof}

It is also useful to have the following $L^3$-version of the (modified) Klainerman-Sobolev inequality. The proof is similar to the previous one, so ommited.
\begin{proposition} \label{KSL3}
Under the bootstrap assumptions, we have the Klainerman-Sobolev inequality, for any sufficiently regular distribution function $g$,
\begin{equation} \label{ineq:ksmsvl3}
\left|\left | \int_v | g |(t,x,v) d\mu(v) \right|\right|_{L^3_{H_\rho}} \lesssim \frac{1}{\rho^2}\left(1+\varepsilon^{1/2}\ln(\rho)\right) \sum_{ | \alpha | \le 2}\int_{H_\rho} \chi \left( \left| \frac{\zz}{v^0 t^{1/2}} \right|^2 | \mathbf Y^\alpha f | \right) d\mu_\rho.
\end{equation}
\end{proposition}

\appendix

\section{An integral estimate} \label{sec:intes}

This section explains how to perform integral estimates on hyperboloids. This has already been used in the previous work \cite{fjs:vfm}, but for the sake of completeness, we precise here these estimates can be obtained.

\begin{lemma}\label{lem-int-est} Let $n,m,p, r$ be non-negative integers such that
  $$
p-n+r<- 1,\quad m>r
  $$
  and $\rho\geq 1$ be a real number. The following integral estimate holds: there exists a constant $C$, depending on $n,p,r$ such that
  $$
\int_{0}^{+\infty} \dfrac{r^p}{(1+t(\rho, r) - r)^m t(\rho, r)^n}  dr   \leq \dfrac{\rho^{p+1 - n - r }}{n-p-r-1}
$$
where
$$
t(\rho, r) = \sqrt{\rho^2 + r^2 }
$$
  \end{lemma}
\begin{proof}
  The base of the proof is the change of variable
  $$
r = \rho  y.
  $$
  The integral then becomes
\begin{align*}
  \int_{0}^{+\infty} \dfrac{r^p}{(1+t(\rho,r) - r)^m t(\rho, r)^n}  dr =& \rho^{1+p - n}  \int_{0}^{+\infty} \dfrac{y^p}{\left(1+\rho \left(( t(1, y) - y\right)\right)^m t(1, y)^n}  dy.
\end{align*}
Noticing that
$$
t(1,y) -  y =  \sqrt{1 + y^2 } - y = \dfrac{1 }{ \sqrt{1 + y^2 } + y }.
$$
Hence, there is a constant $C$, depending on $m,n,p$ such that
\begin{align*}
  \int_{0}^{+\infty} \dfrac{r^p}{(1+t(\rho,r) - r)^m t(\rho, r)^n}  dr \leq & \rho^{p+1 - n}  \int_{0}^{+\infty} <y>^{p-n}\dfrac{<y>^m}{\left(\rho+<y>\right)^m}  dy\\
  \leq & \rho^{p+1 - n}  \int_{0}^{+\infty} <y>^{p-n}\dfrac{<y>^m}{\left(\rho+<y>\right)^{m-r}\left(\rho+<y>\right)^{r} }  dy\\
    \leq & \rho^{p+1 - n}  \int_{0}^{+\infty} <y>^{p-n}\dfrac{<y>^m}{<y>^{m-r}\rho^{r} }  dy\\
        \leq & \rho^{p+1 - n - r }  \int_{0}^{+\infty} <y>^{p-n +r } dy\\
        \leq & \rho^{p+1 - n - r }  \int_{1}^{+\infty} y^{p-n +r } dy \\
          \leq & \dfrac{\rho^{p+1 - n - r }}{n-p-r-1}
\end{align*}
\end{proof}
In the case of interest for us, $p = 2$. Consider the first integral, with $m = r=  0$, then, if $n> 3$,
$$
\int_{0}^{+\infty} \dfrac{r^2}{t(\rho, r)^n}  dr  \leq \dfrac{1}{n-3} \rho^{3-n}.
$$
Consider now the case when $n$ is arbitrary, $ m = 1$, then, if $n > 3 + r,\, r\leq1$, the estimate holds:
$$
\int_{0}^{+\infty} \dfrac{r^2}{(1+t(\rho, r) - r)t(\rho, r)^n}  dr  \leq \dfrac{1}{n-3-r} \rho^{3-n - r}.
$$

\section{Estimates for a rescaling by $\rho$} \label{sec:rhomult}

We first make the following remark: if $(t,x)$ is a point on the hyperboloid $H_\rho$,
 \begin{equation}
\T \rho  = \dfrac{tv^0 - v^ix_i}{\rho}.
\end{equation}
Hence, one notices that, by the reverse Cauchy-Schwarz inequality,
\begin{equation}
  \T \rho \geq m >0.
\end{equation}

Let $f$ be a non-negative solution to the massive equation
 \begin{equation}
  \T  f = h.
\end{equation}
Let $s$ be any non-negative number, and consider the function $\rho^{-s} f$ which satisfies the equation
\begin{equation}
\T \left( \rho^{-s} f \right) = -s \rho^{-s-1} \left(\T \rho\right) \rho f + \rho^{-s} h.
\end{equation}
Hence, for $\rho^{-s} f$, the energy estimates writes, by Lemma \ref{lem:macl},
\begin{align}
\int_{H_\rho} \chi( \rho^{-s} f)(\rho,r,\omega)d\mu_{H_\rho}
- \int_{H_1} \chi(\rho^{-s} f)(1,r,\omega)d\mu_{H_1}=&\\
\int_{1}^\rho \int_{H_{\rho'}} \rho_m\left(-\dfrac{s}{v^0\rho'{}^{s+1}}  \left(\T_p\rho\right)  f+\right.&\left.\dfrac{h(\rho',r,\omega)}{\rho'{}^{s}v^0}\right)d\mu_{H_{\rho'}}d\rho'\\
\int_{H_\rho} \chi( \rho^{-s} f)(\rho,r,\omega)d\mu_{H_\rho}
- \int_{H_1} \chi(\rho^{-s} f)(1,r,\omega)d\mu_{H_1} \leq & \int_{1}^\rho \int_{H_{\rho'}} \rho_m\left(\dfrac{h(\rho',r,\omega)}{\rho'{}^{s}v^0}\right) d\mu_{H_{\rho'}}d\rho',
\end{align}
since $\T \rho \geq 0$.

\printbibliography
\end{document}